\documentclass[a4paper,12pt]{amsbook}
\usepackage[utf8]{inputenc}
\usepackage[T1]{fontenc}
\usepackage[]{amsmath}
\usepackage[]{amssymb}
\usepackage[]{amsthm}
\usepackage{bbding}
\usepackage{stmaryrd}
\usepackage{graphicx}
\usepackage{tikz}
\usepackage{xcolor}
\usepackage[noadjust]{cite}
\usepackage{upgreek}
\usetikzlibrary{matrix,arrows,decorations.pathmorphing}
\usetikzlibrary{matrix,positioning}
\usetikzlibrary{shapes,snakes}
\usetikzlibrary{matrix,fit}
\usetikzlibrary{matrix,arrows,decorations.pathmorphing}
\usepackage{tikz-cd}

\newcommand{\git}{/\mkern-6mu/}
\def\W{\mathfrak{W}}
\def\T{\ensuremath{T}}
\def\t{\ensuremath{t}}
\def\S{\ensuremath{S}}
\def\s{\ensuremath{s}}
\def\Ka{\ensuremath{K}}
\def\ka{\ensuremath{k}}
\def\G{\ensuremath{G}}
\def\g{\ensuremath{g}}
\def\F{\ensuremath{F}}
\def\q{\ensuremath{q}}

\def\z{\mathbf{z}}
\def\C{\mathbb{C}}
\def\QQ{\mathbb{Q}}
\def\R{\mathcal{R}}
\def\RR{\mathbb{R}}

\def\P{\mathbb{P}}
\def\E{\mathbb{E}}
\def\B{\mathbb{B}}
\def\X{\ensuremath{X}}

\def\lie#1{\mathfrak{#1}}
\def\M{\mathcal{M}}

\def\eq#1{\overset{\textcircled{\tiny{#1}}}{=\joinrel=\joinrel=}}
\def\pr#1#2{\prod_{\substack{#1 \\ #2}}}
\def\fr#1#2{\frac{\displaystyle #1 }{\displaystyle #2}}
\def\mju#1#2{\mu_{#1}^{-1}(#2)}

\def\Vand{V}
\def\emf#1{\textit{\textbf{#1}}}

\DeclareMathOperator{\Res}{Res}
\DeclareMathOperator{\image}{im}
\DeclareMathOperator{\ann}{ann}
\DeclareMathOperator{\h}{H}

\DeclareMathOperator{\U}{U}

\DeclareMathOperator{\Fl}{Fl}
\DeclareMathOperator{\diag}{diag}
\DeclareMathOperator{\homm}{Hom}
\DeclareMathOperator{\grr}{Gr}
\DeclareMathOperator{\Id}{Id}
\DeclareMathOperator{\Imm}{Im}
\DeclareMathOperator{\tr}{tr}
\DeclareMathOperator{\Sym}{Sym}
\DeclareMathOperator{\vol}{vol}

\newcommand\Grass[2]{\grr(#1, #2)}
\newcommand\Hom[2]{\homm(#1,#2)}
\newcommand\HomC[2]{\homm(\C^{#1}, \C^{#2})}
\newtheorem{thm}{Theorem}[chapter]
\newtheorem*{thm*}{Theorem}
\newtheorem{lemma}[thm]{Lemma}

\newtheorem{cor}[thm]{Corollary}
\newtheorem{corollary}[thm]{Corollary}
\newtheorem{proposition}[thm]{Proposition}
\theoremstyle{definition}

\newtheorem{remark}[thm]{Remark}
\definecolor{darkgreen}{rgb}{0.01, 0.75, 0.24}
\definecolor{darkpink}{rgb}{0.91, 0.33, 0.5}

\author{Magdalena Zielenkiewicz}
\date{}
\title{The Gysin homomorphism for homogeneous spaces via residues}

\calclayout

\begin{document}

\begin{center}
{\Large University of Warsaw}

\vspace{0.2cm}

{\large Faculty of Mathematics, Informatics and Mechanics}

\vspace{3cm}

{\large Magdalena Zielenkiewicz}

\vspace{2cm}

{\Large The Gysin homomorphism for homogeneous spaces via residues}

\vspace{0.2cm}
\emph{PhD dissertation}
\end{center}

{\raggedleft\vfill{%
Supervisor

dr hab. Andrzej Weber

Institute of Mathematics

University of Warsaw

}\par}

\newpage

\topskip0pt
\vspace*{\fill}
\begin{flushleft}
Author’s declaration:

I hereby declare that this dissertation is my own work.

\vspace{0.5cm}

December 12, 2016 \hfill  ...................................................\\
\hfill {\small \it Magdalena Zielenkiewicz}

\vspace{2cm}

Supervisor's declaration:

The dissertation is ready to be reviewed.

\vspace{0.5cm}

December 12, 2016 \hfill  ...................................................\\
\hfill {\small \it dr hab. Andrzej Weber}

\end{flushleft}

\vspace*{\fill}

\newpage

\begin{flushleft}
 {\LARGE \bf Abstract}
 \vspace{1cm}
\end{flushleft}

\noindent The subject of this dissertation is the Gysin homomorphism in equivariant cohomology for spaces with torus action. We consider spaces which are quotients of classical semisimple complex linear algebraic groups by a parabolic subgroup with the natural action of a maximal torus, the so-called partial flag varieties. We derive formulas for the Gysin homomorphism for the projection to a point, of the form
\[\int_\X \alpha = \Res_{\mathbf{z}=\infty} \mathcal{Z}(\mathbf{z}, \mathbf{t}) \cdot \alpha(\mathbf{t}),\] 
for a certain residue operation and a map $\mathcal{Z}(\mathbf{z}, \mathbf{t})$, associated to the root system. The mentioned description relies on two following generalizations of theorems in symplectic geometry. We adapt to the equivariant setting (for torus actions) the Jeffrey--Kirwan nonabelian localization and a theorem relating the cohomology of symplectic reductions by a compact Lie group and by its maximal torus, following the approach by Martin. Applying the generalized theorems to certain contractible spaces acted upon by a product of unitary groups we derive the push-forward formula for partial flag varieties of types $A$, $B$, $C$ and $D$. 

\begin{flushleft}
 \vspace{0.5cm}
 
{\bf Keywords:} Gysin homomorphism, equivariant cohomology, torus action, homogeneous space

{\bf AMS Subject Classification:} 57T15, 53D20, 14F43, 14M15 

\end{flushleft}

\newpage

\begin{flushleft}
 {\LARGE \bf Streszczenie}
 \vspace{0.8cm}
 \end{flushleft}
 
\noindent Przedmiotem niniejszej rozprawy jest homomorfizm Gysina w kohomologiach ekwiwariantnych dla przestrzeni z działaniem torusa. Rozważane są przestrzenie będące ilorazami klasycznych półprostych zespolonych liniowych grup algebraicznych przez podgrupę paraboliczną, wraz z naturalnym działaniem torusa maksymalnego, zwane inaczej przestrzeniami częściowych flag. Niniejsza rozprawa opisuje homomorfizm Gysina dla rzutowania na punkt za pomocą wzorów postaci
\[\int_\X \alpha = \Res_{\mathbf{z}=\infty} \mathcal{Z}(\mathbf{z}, \mathbf{t}) \cdot \alpha(\mathbf{t}),\] 
dla pewnego residuum i funkcji zespolonej wielu zmiennych $\mathcal{Z}(\mathbf{z}, \mathbf{t})$, związanej z układem pierwiastków grupy. Wspomniany opis opiera się na uogólnieniach dwóch twierdzeń z geometrii symplektycznej, udowodnionych w pierwszej części rozprawy. Pierwszym z nich jest uogólnienie (w kontekście ekwiwariantnym dla działania torusa) twierdzenia Jeffrey--Kirwan o nieabelowej lokalizacji, drugim zaś twierdzenie wiążące kohomologie redukcji symplektyczniej przez zwartą grupę Liego z kohomologiami redukcji przez torus maksymalny w tej grupie, w sformułowaniu Martina. W drugiej części rozprawy uogólnione twierdzenia zostały użyte do redukcji symplektycznych pewnych przestrzeni ściągalnych z działaniem produktu grup unitarnych, w celu otrzymania wzorów opisujących homomorfizm Gysina dla przestrzeni częściowych flag typów $A$, $B$, $C$ i $D$. 

\newpage

\begin{flushleft}
 {\LARGE \bf Acknowledgements}
 \vspace{0.8cm}
 
\end{flushleft}

\noindent I first and foremost thank my supervisor Andrzej Weber for the time and patience he devoted to supporting me during my PhD. He guided me throughout the process without restricting my research freedom and always provided me with good advice. Writing a thesis under his supervision was a very enjoyable experience. I especially thank him for the effort of reading all preprints I produced, sometimes multiple times. \newline

I'm grateful to Piotr Pragacz for his interest in my research and for fruitful discussions. I also thank Andras Szenes, Allen Knutson and Christian Okonek for their comments and suggestions and Lionel Darondeau for his remarks on my work as well as for useful tips on editing a mathematical manuscript. \newline

Lastly, I thank all those who supported me in various ways during the years of my PhD, making this period of my life a happy adventure. The list includes my family, friends and the fellow PhD students at the Faculty of Mathematics, Informatics and Mechanics. I also want to express gratitude to Jaros\l{}aw Wi\'{s}niewski for inviting me to join his research group, which was a supportive and creative environment to work in. \newline

My work during the PhD has been supported by the National Science Center (NCN) grant 2013/08/A/ST1/00804 (October 2015 - May 2016) lead by Jaros\l{}aw Wi\'{s}niewski and by the National Science Center (NCN) grant 2015/17/N/ST1/02327 (from June 2016).

\newpage\null\thispagestyle{empty}\newpage

\tableofcontents

\chapter*{Introduction}

The Gysin homomorphism, introduced under the name \emph{Umkehrungshomomorphismus} in \cite{gysin1941}, originally associated a homomorphism $f_*: \h^*(M) \to \h^*(N)$ between cohomology groups to a map $f: M \to N$ of closed oriented manifolds and was later generalized to many different settings. The Gysin map for fiber bundles and de Rham cohomology has a natural interpretation as integration along fibers. Gysin maps have become an important tool in algebraic geometry, due to the growing importance of characteristic classes. An illustrative example is the definition of Segre classes of a vector bundle $E \to X$ as push-forwards of powers of the hyperplane class of the projective bundle $\P(E)$ (\cite{fulton2013intersection}). \newline

Gysin maps have proved useful in singularity theory and have provided a tool to study the degeneracy loci of morphisms of flag bundles, which are related to Schubert cycles by the Thom--Porteous formula (\cite{porteous1971}) and similar results of Kempf and Laksov (\cite{kempf1974}), Lascoux (\cite{lascoux1974}) and Pragacz (\cite{pragacz1988}). Most of the early results on push-forwards for flag bundles relied on inductive procedures, reducing the problem to studying projective bundles. The inductive approach is still used by numerous authors (see for example Fulton \cite{fulton1992}, Kazarian \cite{kazarian}). The study of the degeneracy loci of bundles lead to the development of combinatorial techniques, concentrating on the study of Schur polynomials and their generalizations and modifications. For example, the so-called factorial Schur polynomials represent the degeneracy loci in the Kempf--Laksov formula and provide a useful basis of the equivariant cohomology of the complex Grassmannian $\h^*_{\T}(\Grass{k}{n})$ over $\h^*_{\T}(pt)$ (\cite{knutson2003}). Analogous P- and Q-polynomials introduced in \cite{pragacz1988} provide formulas for the Lagrangian and orthogonal degeneracy loci and Fulton's universal Schubert polynomials (\cite{fulton1999}) generalize the previous notions. The case of isotropic Grassmannians was studied by Buch, Kresch and Tamvakis (\cite{buch2015}), resulting in construction of the so-called theta polynomials. \newline 

Another direction of study of Gysin homomorphisms was motivated by Quillen's description of the push-forward in the complex cobordism theory using a certain notion of a residue (\cite{quillen1969}), which provided a background for the later results of Damon (\cite{damon1973}) and Akyildiz and Carrell (\cite{akyildiz1987}) expressing the Gysin maps for flag bundles as Grothendieck residues. \newline

The adaptation of equivariant cohomology techniques (\cite{borel1960seminar}) to algebraic geometry had enriched the theory with new tools and a different perspective. Numerous classical theorems have been rephrased in terms of equivariant characteristic classes, including the mentioned formulas for degeneracy loci. Many formulas that had originally been proven geometrically have simpler proofs using equivariant cohomology, for example the proof of the Józefiak--Lascoux--Pragacz (\cite{jozefiak1981}) formula in \cite{weber2012}. 
\newline

A powerful tool to study Gysin maps in equivariant cohomology are localization theorems (see Ch.~\ref{ch:prelim}, Sect.~\ref{subsc:localization} for details). For example, B\'erczi and Szenes use localization techniques to prove a formula expressing the equivariant push-forward for flag bundles as a residue at infinity of a complex function (\cite{berczi2012}, Sect.~6). 
\newline

In \cite{zielenkiewicz2014} we have presented a new approach to push-forwards in equivariant cohomology of homogeneous spaces of classical Lie groups, inspired by the push-forward formula for flag varieties of B\'erczi and Szenes in \cite{berczi2012}. For a homogeneous space $\G/P$ of a classical semisimple Lie group $\G$ and its maximal parabolic subgroup $P$ with an action of a maximal torus $\T$ in $\G$ we express the Gysin homomorphism $\pi_*: \h_{\T}^*(\G/P) \to \h_\T^*(pt)$ associated to a constant map $\pi: \G/P \to pt$ in the form of an iterated residue at infinity of a certain complex variable function. The residue formula depends only on the combinatorial properties of the homogeneous space involved, in particular on the Weyl groups of $\G$ and $P$ acting on the set of roots of $\G$. The formulas were obtained using localization techniques. \newline

In this dissertation we show how to obtain the formulas of \cite{zielenkiewicz2014} in the context of symplectic reductions, using the Jeffrey--Kirwan nonabelian localization theorem to provide a natural geometric interpretation for the mentioned formulas. For a compact Lie group $\Ka$ acting in a Hamiltonian way on a symplectic manifold $\M$, one considers the symplectic reduction $\M \git \Ka$ together with the Kirwan map $\kappa: \h^*_{\Ka}(\M) \to \h^*(\M \git \Ka)$. By a theorem by Kirwan (\cite{kirwan1984cohomology}), $\kappa$ is an epimorphism. Moreover, the push-forwards of images under $\kappa$ of singular cohomology classes can be expressed as residues at infinity of a certain expression. We extend this result to push-forwards to a point in equivariant cohomology, in the following theorem proven in Ch.~\ref{ch:equiv-jk}. Let $\Ka$ be a compact Lie group with a maximal torus $\S$ and Weyl group $\W$ and let $\T$ be a torus. Assume $\M$ is a compact manifold equipped with commuting actions of $\Ka$ and $\T$. Consider the action of $\S$ on $\M$ via the restriction of the action of $\Ka$ and assume this action has a finite number of fixed points. Assume $\Ka$ acts in a Hamiltonian way and let $\kappa_\T$ be the equivariant Kirwan map for the symplectic reduction of $\M$ with respect to $\Ka$.

\begin{thm}
 Let $\alpha \in \h^*_{\T \times \Ka}(\M)$ be a $\T \times \Ka$-equivariant cohomology class. Let $\upvarpi$ denote the product of the roots of $\Ka$. Then the $\T$-equivariant push-forward of the class  $\kappa_\T(\alpha) \in \h^*_{\T}(\M \git \Ka)$ is given by the following formula
\[\int_{\M \git \Ka} \kappa_\T(\alpha) = \frac{1}{|\W|} \sum_{p \in D} \Res_{z=\infty} \frac{\upvarpi \cdot i^*_p \alpha}{e_p^{\T \times \S}}, \]
where $i_p: \{ p \} \hookrightarrow \M$ denotes the inclusion  of the fixed point $p$ and $e_p^{\T \times \S}$ is the $\T \times \S$-equivariant Euler class of the normal bundle at $p$. The set $D$ is a certain subset of the fixed point set $\M^{\S}$, described in Ch.~\ref{ch:equiv-jk}, Sect.~\ref{subsc:gk}. \label{thm:intro1}
\end{thm}
The assumptions on the finite number of fixed points of the $\S$-action can be removed, resulting in a more complex formula for the push-forward, indexed over a certain subset of connected components of the fixed point set. \newline

The homogeneous spaces of compact semisimple Lie groups can be presented as symplectic reductions of symplectic manifolds by a Hamiltonian action, or embed in such a reduction. One can therefore use the above theorem to obtain residue-type formulas for push-forward in equivariant cohomology. Ch.~\ref{ch:formulas} provides details on how to obtain such formulas for partial flag manifolds of types $A$, $B$, $C$ and $D$. The computations require a technical lemma, which is a generalization---in the context of equivariant cohomology---of the theorem by Martin (\cite{martin2000}), as stated below and proven in Ch.~\ref{ch:martin}.

\begin{thm} \label{thm:intro2}
To a weight $\gamma$ of the $\S$-action one associates a $\T$-equivariant line bundle on $\M \git \S$. We denote by $e^{\T}$ the product of $\T$-equivariant Euler classes $e^{ \T}(L_{\gamma})$ over the set of roots of $\Ka$, 
\[e^{\T} = \prod_{\gamma \in \Phi} e^{\T}(L_{\gamma}) \in \h^*_\T( \M \git \S).\]
Then the following relation between the push-forwards to a point in $\T$-equivariant cohomology holds
\[\int_{\M \git \Ka} \alpha = \frac{1}{|\W|} \int_{\M \git \S} \tilde{\alpha} \cdot e^{\T}.\]
\end{thm}


The partial flag variety of type $d=(d_1,\dots,d_k)$ in $W \simeq \C^n$ is defined as the quotient
\[\Fl_d(W) = \U(W) / \U(n_1) \times \U(n_2) \times \dots \times \U(n_k),\]
where $n_1=d_1$ and  $n_i = d_i - d_{i-1}$ for $i=2,\dots,k$.
We consider the action of the maximal torus $\T$ in $\U(n_1) \times \U(n_2) \times \dots \times \U(n_{k})$ on $\Fl_d(W)$. Using the presentation of $\Fl_d(W)$ as a symplectic reduction due to Kamnitzer (\cite{kamnitzer}) we apply the Theorems \ref{thm:intro1} and \ref{thm:intro2} to provide a formula for push-forward in $\T$-equivariant cohomology. The partial flag variety is a symplectic reduction of a contractible space $\M$ by the Hamiltonian action of the group $\Ka=\U(d_1) \times \dots \times \U(d_k)$ (note that this is a different group than in the definition of the flag variety as a homogeneous space). Let $\kappa_{\T}$ denote the equivariant Kirwan map. Then the equivariant Gysin map $\int\limits_{\Fl_d(W)} : \h^*_{\T}(\Fl_d(W)) \to \h^*_{\T}(pt)$ is given by the following formula.

\begin{thm} Let $\alpha \in \h^*_{\T \times \Ka}(\M)$ and let $\{ z_{1}, \dots, z_{d_k}\}$ denote the characters of the the last component of the maximal torus $\S$ in $\U(d_1)\times\dots \times \U(d_k)$. Then

\[
\int\limits_{\Fl_d(W)} \kappa_{\T}(\alpha) 
= \frac{1}{|\W|} \Res_{z_{1}, \dots, z_{d_k}=\infty} \fr{\alpha \cdot \prod_{\substack{i \neq j  \\ (i,j) \in I_{\Fl}}} (z_{i} - z_{j})}{ \prod_{l=1}^{d_k} \prod_{m=1}^{n}(-z_{l}+t_{m})}
\]
\end{thm}
For the description of the indexing set $I_{\Fl}$ see Ch.~\ref{ch:formulas}, Sect.~\ref{subsc:seriesA}. The push-forward formula is a simplification of a more involved residue, which takes into account variables corresponding to all the characters of the maximal torus $\S$. We derive analogous formulas for partial flag varieties of types $B$, $C$ and $D$ in Sect.~\ref{subsc:seriesC} and \ref{subsc:seriesBD} of Ch.~\ref{ch:formulas}. \newline

Recently, Darondeau and Pragacz have obtained push-forward formulas for flag bundles (\cite{darondeau2015}) in terms of Segre classes of the vector bundles involved. Both the techniques used and the resulting formulas differ from the approach and results presented here.\newline

Possible future directions of the development of the ideas and techniques presented in this dissertation include the generalization of the results to homogeneous spaces of arbitrary semisimple Lie groups, including the exceptional cases, as well as the adaptation of the push-forward residue type formulas to equivariant K-theory by applying the Riemann-Roch theorem to our formulas. \newline 

The organization of the dissertation is as follows. In Chapter \ref{ch:prelim} we introduce the necessary definitions and notation and describe the nonabelian localization theorems of Jeffrey--Kirwan, Guillemin--Kalkman and Martin. Chapters \ref{ch:equiv-jk} and \ref{ch:martin} contain proofs of the analogues in equivariant cohomology of the Jeffrey--Kirwan and Martin theorems respectively. In Chapter \ref{ch:formulas} we derive formulas for the Gysin homomorphism for partial flag varieties of types $A$, $B$, $C$ and $D$. Finally, Chapter \ref{ch:applications} provides and example of computation using the mentioned formulas. We reprove in a simple way the Pragacz--Ratajski theorem on push-forwards of Schur classes on the Lagrangian Grassmannian. Two technical steps in the derivation of the push-forward formulas are the content of the Appendices A and B.

\chapter{Preliminaries}\label{ch:prelim}

	\section{Equivariant cohomology}\label{sc:cohomology}

Let $\G$ be a topological group and let $\B\G$ denote the classifying space of $\G$ (\cite{may1999concise}). Let $\E\G \to \B\G$ be the universal principal $\G$-bundle, in particular $\E\G$ is a contractible space on which $\G$ acts freely.\footnote{For topological groups universal bundles exist under mild assumptions, it suffices to assume that the inclusion of the identity element in $\G$ is a cofibration (\cite{may1999concise}). For algebraic groups universal bundles do not exist in general, instead one constructs a directed system of bundles approximating them (see \cite{totaro2014}, \cite{edidin1996}).}
For a topological space $\X$ with a right $\G$-action, one defines the \emf{homotopy quotient} of $\X$ by $\G$, denoted by $\E\G \times^{\G} \X$, to be the quotient of $\E\G \times \X$ by the diagonal action of $\G$
\[g(e,x) = (e g^{-1}, g x ).\]
The homotopy quotient has a natural structure of a fiber bundle \linebreak $\E\G \times^{\G} \X \to \B\G$ with fiber $\X$, via the map induced by the projection on the first factor.   
The \emf{equivariant cohomology} of $\X$ with coefficients in a ring $R$ is, by definition, the singular cohomology of the homotopy quotient,
\[\h^*_{\G}(\X;R) := \h^*(\E\G \times^{\G} \X;R).\]
This definition is often referred to as the \emf{Borel construction} of equivariant cohomology as it was introduced by Borel in \cite{borel1960seminar}. Other important constructions of equivariant cohomology for smooth manifolds with an action of a Lie group are the \emf{Cartan construction}, described in Ch.~\ref{ch:prelim}, Sect.~\ref{subsc:cartan}, and the  \emf{Weil construction} (\cite{cartan1950}). Throughout this dissertation we consider cohomology theories with coefficients in a field of characteristic zero, usually $\mathbb{R}$, and abbreviate in notation $\h^*(-):=\h^*(-; \mathbb{R})$. The cup product of cohomology classes $\alpha, \beta$ is denoted by $\alpha \cdot \beta$. \newline

Equivariant cohomology is functorial for equivariant maps. If $\varphi: \G \to \G^{'}$ is a group homomorphism and $f: \X \to \X^{'}$ is an \emf{equivariant map}, i.e. $f(g x) = \varphi(g) f(x)$, then there is an induced map 
\[ \h^*_{\G^{'}}(\X^{'}) \to \h^*_{\G} (\X). \]

\begin{remark}\label{ex:coh_point}
If $\X$ is a point, then the equivariant cohomology is the one of the classifying space
\[\h^*_{\G}(pt) = \h^*(\B\G).\]
In particular, for cohomology with coefficients in a ring $R$ one has
\begin{enumerate}
	\item $\h^*_{S^1}(pt; R) = \h^*(\P^{\infty};R) = R[t],$ where $t = c_1(\mathcal{O}(-1))$ is the first Chern class of the tautological bundle on the infinite projective space $\P^{\infty}$. 
	\item If $\T=(S^1)^n$ is a torus, then
\[\h^*_{\T}(pt;R) = \h^*((\P^{\infty})^n;R) = R[t_1, \dots, t_n].\]
For $i=1,\dots,n$ the generator $t_i = c_1(\mathcal{O}_i(-1))$ is the first Chern class of the pullback of the tautological bundle $\mathcal{O}(-1)$ on $(\P^{\infty})^n$ along the projection on the $i$th factor.
	\item If $\G=GL_n$, then $\h^*_{\G}(pt)=R[e_1, \dots, e_n]$, where $e_i = c_i(\mathcal{R})$ is the $i$th Chern class of the tautological bundle over the Grassmannian of $n$-planes in $\C^{\infty}$, $\Grass{n}{\infty}$. 

\end{enumerate}
\end{remark}

The following theorem relates the $\G$-equivariant cohomology of a point with the $\T$-equivariant cohomology for a maximal torus $\T < \G$.

\begin{thm}[Chevalley] For a vector space $V$ denote by $\Sym V $ its symmetric algebra. Let $\lie{\g}$ be a Lie algebra and let $\W$ be the Weyl group of $\lie{\g}$. The restriction to a maximal commutative subalgebra $\lie{\t} \subseteq \lie{\g}$ induces an isomorphism of algebras
\[ (\Sym^*\g^*)^{\G} \to (\Sym^* \lie{t}^*)^{\W}.\]
In particular, for cohomology with coefficients in a ring $R$ in which the order of the Weyl group is invertible, one has $\h^*_{\G}(pt;R) = \h^*_{\T}(pt;R)^\W$. \label{thm:chevalley}
\end{thm}

		\subsection{Approximation spaces}\label{subsc:approx}

The total space of the universal bundle $\E\G$ is typically infinite dimensional. While working with algebraic varieties with an action of a linear algebraic group, one often uses finite dimensional \emf{approximation spaces} $\E_m$, whose connectivity diverges to infinity as $m$ goes to infinity. The spaces $\E_m$ allow to compute the equivariant cohomology of $\X$, in the sense of the following lemma. 

\begin{lemma}\label{lem:approx}
Suppose $\E_m$ is a connected space with a free $\G$-action, such that $\h^i(\E_m)=0$ for $0 < i < k(m)$ for some natural number $k(m)$. Then for any $\X$ there are natural isomorphisms 
\[ \h^i(\E_m \times^{\G} \X) \simeq \h^i_{\G}(\X) \textrm{ for } i < k(m).\]  
\end{lemma}

For example for $\G=\C^*$ one can take $\E_m = \C^{m}\setminus \{0\}$ as approximation spaces. These $\E_m$ satisfy the assumptions of the Lemma \ref{lem:approx} with $k(m) = 2m-1$. For $\G = GL(n,\C)$, one can choose the approximation spaces to be $\E_m = M^{0}_{m \times n}$, the set of full rank $m \times n$-matrices. These spaces satisfy the assumptions of the the Lemma \ref{lem:approx} with $k(m) = 2(m-n)$. From the approximation spaces for $GL(n,\C)$ one can construct approximation spaces for an arbitrary complex linear algebraic group $\G$, in the category of nonsingular algebraic varieties over $\C$, by embedding $\G$ into $GL(n,\C)$ for some $n$. \newline

The approximation spaces for the Chow ring of the classifying space for an arbitrary complex linear group $\G$ are described by Totaro in \cite{totaro2014}.

		\subsection{The Cartan model}\label{subsc:cartan}

In the category of compact smooth manifolds with an action of a compact Lie group, one can compute the equivariant cohomology with real or complex coefficients from the equivariant analogue of the de Rham complex, giving the strict meaning to the concept of an equivariant differential form. \newline

Let $\X$ be a compact smooth manifold with an action of a compact connected Lie group $\G$ with Lie algebra $\lie{\g}$. Let $k=\mathbb{R}$ or $\C$ and let $\Omega(\X)$ denote the $k$-valued differential forms on $\X$. We define the \emf{equivariant differential forms} of degree $q$ to be the elements of 
\[ \tilde{\Omega}^{q}_{\G}(\X) := \bigoplus_{i+2j=q}(\Omega^i(\X) \otimes \Sym^i(\lie{\g}^*))^{\G}, \]
where $\Sym^i(\lie{\g}^{*})$ is the symmetric algebra on the dual of the Lie algebra of $\G$. The $\G$-invariants are taken with respect to the natural action by translations on $\Omega^*(\X)$ and the coadjoint action on $\lie{\g}^*$. One can think of elements of $\tilde{\Omega}^{q}_{\G}(\X)$ as being $\G$-invariant, $\Omega^*(\X)$-valued polynomial functions on $\lie{\g}$.
For $\xi \in \lie{\g}$, the \emf{fundamental vector field} $v_{\xi}$ at the point $x \in \X$ is defined by the formula 
\[v_{\xi}(x):= \frac{d}{dt}\bigg{|}_{t=0} (\exp(t \xi)\cdot x).\]
The fundamental vector field is sometimes called the \emf{infinitesimal generator} of the action. \newline

Define the differential $d_{\G}: \tilde{\Omega}^{q}_{\G}(\X) \to \tilde{\Omega}^{q+1}_{\G}(\X)$ for a $\G$-invariant polynomial function $\alpha: \lie{\g} \to \Omega^{*}(\X)$ by the formula
\[ (d_{\G} \alpha)(\xi) := d(\alpha(\xi)) + \iota(v_{\xi})\alpha(\xi), \]
where $d$ is the differential on $\Omega(\X)$ and $\iota$ denotes contraction with a vector field \footnote{The differential $d_{\G}$ can also be written explicitly in coordinates, as follows. Let $\{ a \}$ be a basis of  $\lie{\g}$  and $\{ \mu^a \}$ the dual basis of $\lie{\g}^*$. Then $d_{\G}$ is given by the formula \[ d_{\G} = d \otimes 1 - \sum_{a}  \iota(v_{a}) \otimes \mu^a\] }.

\begin{thm}[Equivariant de Rham Theorem] Let $\G$ be a compact connected Lie group acting on a smooth manifold $\X$. Let $(\tilde{\Omega}_{\G}(\X), d_{\G})$ be the Cartan complex. Then
\[ \h^*_{\G}(\X; k) \simeq \frac{\ker d_{\G}}{ \image d_{\G}}.\]
\end{thm} 
The above theorem is due to Cartan (\cite{cartan1950}). \newline

A differential form $\omega \in \Omega(\X)$ is called \emf{horizontal} if 
\[ \iota(v_{\xi})\omega = 0 \textrm{ for all } \xi \in \lie{\g}.\]
The set of $\G$-invariant horizontal forms is closed under differentiation and hence forms a subcomplex of $\Omega(\X)$, the cohomology of which (with real coefficients) is isomorphic to the singular cohomology of $\X / \G$, provided the action of  $\G$ is proper (the result is due to Koszul (\cite{koszul1953}) for compact groups and Palais (\cite{palais1961}) for proper actions). If the action is locally free, the $\G$-invariant horizontal forms on  $\X$ are in bijection with differential forms on the orbifold $\X / \G$.

Finally, we introduce the notion of the connection form, which will be needed in Ch.~\ref{ch:equiv-jk}. For a $\G$-principal bundle $E \to \X$ and $\xi \in \lie{\g}$, let us denote by $E_{\xi}$ the fundamental vector field on $E$ associated to $\xi$. There exists then a $\lie{\g}$-valued $1$-form $\theta$ on $E$, called the \emf{connection form}, which is $\G$-invariant and satisfies
\[ \theta(E_{\xi}) = \xi \textrm{ for } \xi \in \lie{\g}.\]

		\subsection{Equivariant formality}\label{subsc:formality}

The equivariant cohomology $\h^*_{\G}(\X)$ is naturally a module over the cohomology of a point $\h^*_{\G}(pt)$, the module structure being induced by the map $\X \to  pt$. The space $\X$ is called \emf{equivariantly formal} with respect to the action of a connected group $\G$ if $\h^*_{\G}(\X)$ is a free $\h^*_{\G}(pt)$-module. In particular, there is an isomorphism of $\h^*_{\G}(pt)$-modules
\[ \h^*_{\G}(\X) \simeq \h^*(\X) \otimes \h^*_{\G}(pt).\]
This condition is equivalent to saying that the Serre spectral sequence of the fibration $\E\G \times^{\G} \X \to \B \G$ degenerates at the $E_2$-term. In the case of compact Lie group actions, it suffices to consider equivariant formality for torus actions, due to the following proposition (\cite{guillemin2002moment}), which is a consequence of Theorem \ref{thm:chevalley}.

\begin{proposition}\label{prop:formal}
Let $\G$ be a connected compact Lie group and let $\T$ be a maximal torus in $\G$. Let $\X$ be a compact $\G$-manifold. Then $\X$ is equivariantly formal with respect to the $\G$-action if and only if it is equivariantly formal with respect to the $\T$-action.
\end{proposition}

There is a wide class of $\T$-equivariant spaces (\cite{goresky1997}), including:
\begin{itemize}
	\item Any $\T$-space having a $\T$-invariant $CW$-decomposition.
	\item Any $\T$-space whose cohomology groups vanish in odd degrees.
	\item Compact symplectic manifolds with a Hamiltonian $\T$-action.
\end{itemize}

Equivariant formality ensures the nice behavior of the equivariant cohomology groups, with two key properties stated as Propositions \ref{prop:epi} and \ref{prop:mono} below. For the proofs see \cite{guillemin2002moment}, Appendix C, Chapter 4.

\begin{proposition}\label{prop:epi}
Let $\X$ be an equivariantly formal $\T$-space. Then the map 
\[\h^*_{\T}(\X;\QQ) \to \h^*(\X;\QQ),\]
given by restriction to a fiber in the fiber bundle $\E\G \times^{\G} \X \to \B\G$, is an epimorphism.
\end{proposition}

\begin{proposition}\label{prop:mono}
Let $\X$ be an equivariantly formal $\T$-space. Then the map induced by the inclusion of the fixed point set $i: \X^{\T} \hookrightarrow \X$
\[\h^*_{\T}(\X;\QQ) \to \h^*_{\T}(\X^{\T};\QQ)\]
is a monomorphism.
\end{proposition}

		\subsection{Gysin maps}\label{subsc:gysin}	

Let $f: X \to Y$ be a map of closed oriented manifolds. The \emf{Gysin map} $f_*: \h^*(X) \to \h^*(Y)$, also called the \emf{push-forward} in cohomology, is defined by the following diagram, in which the vertical maps are Poincar\'{e} duality isomorphisms and the bottom map is the push-forward in homology.

\[
\begin{tikzcd}
\h^*(X)  \arrow[d, "P.D.", leftrightarrow]  \arrow[r, "f_*"] & \h^{* - \dim X + \dim Y}(Y)  \arrow[d, "P.D", leftrightarrow] \\
\h_{\dim X - *}(X) \arrow[r, "f_*"] & \h_{\dim X - *}(Y)
\end{tikzcd}
\]
One analogously defines the push-forward of a proper map $f:X \to Y$ of noncompact manifolds, replacing the homology groups with the Borel-Moore homology groups. \newline

In equivariant cohomology, the Gysin maps are defined as the standard Gysin maps applied to the approximation spaces. For a proper morphism $f: \X \to Y$ one defines 
\[f_*^{\G} : \h^i_{\G}(\X) \to \h^{i+d}_{\G}(Y), \]
where $d = \dim Y - \dim \X$, as the classical Gysin homomorphism for the map 
\[ \E\G_m \times^{\G} \X \to \E\G_m \times^{\G} Y .\]
In the case when $f:\X \to pt$ is the constant map, the Gysin homomorphism $f_*^{\G}: \h^*_{\G}(\X) \to \h^*_{\G}(pt)$ associated to it is often denoted by $a \mapsto \int_\X a$. \newline

The Gysin homomorphism satisfies the following three fundamental properties\footnote{The notion of the Gysin push-forward can be extended to an arbitrary complex-oriented multiplicative and additive generalized cohomology theory $h^*(-)$ (\cite{switzer1975algebraic}), as described in \cite{levine2007algebraic}. Proposition \ref{prop:gysin1} holds in the general setting (for the proper notion of tranversality in this setting see \cite{levine2007algebraic}, Ch.~1.)}. 

\begin{proposition}\label{prop:gysin1}

\begin{enumerate}
	\item \textbf{Naturality:} For a composition $X \xrightarrow{f} Y \xrightarrow{g} Z$ one has $(g \circ f)_* = g_* \circ f_*$.
	\item \textbf{Projection formula:} For $x \in \h^*_\G(X)$ and $y \in \h^*_\G(Y)$ one has $f_*(f^*(y)\cdot x) = y \cdot f_*(x) $.
	\item \textbf{Base change:} For a commutative diagram in which $X \times_{Z} Y$ is the fiber product of $X$ and $Y$ over $Z$ and the maps $f$ and $g$ are transverse

\[
\begin{tikzcd}
X \times_{Z} Y  \arrow[d, "\overline{g}"]  \arrow[r, "\overline{f}"] & Y  \arrow[d, "g"] \\
X \arrow[r, "f"] & Z
\end{tikzcd},
\]
one has $f^* \circ g_* = \overline{g}_* \circ \overline{f}^*$.

\end{enumerate}
\end{proposition}

For the proof of Proposition \ref{prop:gysin1} see \cite{quillen1971elementary}. When studying push-forwards for manifolds we will need the following additional property of the Gysin map (\cite{bott2013differential}).

\begin{proposition}\label{prop:gysin2} Let $X$ be a compact oriented manifold. 

\begin{enumerate}
	\item \label{pushpull} \emf{Push-pull formula:} For an inclusion of a closed oriented submanifold $i: Z \to X$, the composition $i_* i^*$ is the multiplication by the Poincar\'e dual of the fundamental class of $Z$,
\[i_* i^* \alpha = \alpha \cdot [Z]^{PD}.\]
	\item \emf{Normal bundle to the zero locus:} \label{normalbundle} Let $p: E \to X$ be a vector bundle. A section $s: X \to E$ is called \emf{transverse} if it is transverse to the zero section. Let $Z$ be the zero locus of a transversal section. Then $Z$ is a submanifold of $X$ and its normal bundle in $X$ is the restriction of $E$ to $Z$, 
\[\nu_{Z/X} = E_{|Z}.\]
	\item \emf{Euler class:} \label{poincaredual}
Let $p: E \to X$ be an oriented vector bundle over an oriented manifold $X$. Then the Euler class $e(E)$ of $E$ is Poincar\'e dual to the zero locus of a transversal section. 
\end{enumerate}
\end{proposition}

		\subsection{Localization}\label{subsc:localization}

Localization theorems in equivariant cohomology enable one to recover some of the structure of $\h^*_{\G}(\X)$ from the fixed point set of the action alone. The first results pointing towards localization theorems are due to Borel.\footnote{See for example \cite{borel1960seminar}, Ch.~IV, Proposition 3.6.} The following localization theorem is due to Quillen (\cite{quillen1971}). 

\begin{thm}[Quillen]
Let $\X$ be a compact topological space equipped with an action of a torus $\T$. Assume the set of identity components of the isotropy groups of points of $\X$ is finite. Let $\X^{\T}$ be the fixed point set of this action. Then the restriction homomorphism 
\[\h^*_{\T}(\X) \to \h^*_{\T}(\X^{\T})\]
is an $\h^*_{\T}(pt)$-module isomorphism modulo torsion. 
\end{thm} 

If $\X$ is additionally a manifold, the following formula due to Atiyah and Bott (\cite{atiyah1984}) and independently to Berline and Vergne (\cite{berline1982}) describes the above isomorphism explicitly.\footnote{We state the theorem in the case when the action has a finite number of fixed points. However, the theorem works also for nonisolated fixed points, if one replaces the sum over fixed points with the sum over the fixed components, and the restriction to a point with the integral over the fixed component.}

\begin{thm}[Atiyah--Bott, Berline--Vergne]
Let $\X$ be a compact manifold with an action of a torus $\T$. Assume $\T$ acts with a finite number of fixed points. For a fixed point $p$, let $i_p: p \hookrightarrow \X$ denote the inclusion of $p$ into $\X$.
For an equivariant class $\alpha \in \h_{\T}^*(\X)$ one has
\[ \int_{\X} \alpha  = \sum_{p \in \X^\T} \frac{i_p^* \alpha}{e_p},\]
where $e_p$ is the equivariant Euler class of the tangent bundle at the fixed point $p$, and $i_p^*: \h^*_{\T}(\X) \to \h^*_{\T}(p)$ is induced by the inclusion  $i_p$. \label{abbv-formula} \label{thm:abbv}
\end{thm}

	\section{Homogeneous spaces of semisimple Lie groups}\label{sc:lie}

We introduce the necessary notation concerning Lie groups. For the definitions of the introduced notions we direct the reader to \cite{humphreys1994introduction}. \newline

For a compact \emf{semisimple Lie group} $\G$, we denote its \emf{Lie algebra} by $\lie{\g}$. The complexification $\lie{\g}_\C=\lie{\g} \otimes \C$ of $\lie{\g}$ determines a reductive linear algebraic group $\G_\C$. The group $\G$ is then called the \emf{compact real form} of $\G_\C$. We choose a \emf{Cartan subalgebra} $\lie{\t}$, which is a nilpotent, self-normalizing subalgebra of $\lie{\g}$.  Such a subalgebra is a Lie algebra of a maximal torus in $\G$,
which we denote by $\T$. The choice of $\lie{\t}$ determines the \emf{Cartan decomposition} of $\lie{\g}_\C$
\[ \lie{\g}_\C = \lie{\t}_\C \oplus \sum_{\gamma \in \Phi} \lie{\g}_{\gamma}, \]
where for $\gamma \in \lie{\t}_\C^*$, the subspace $\lie{\g}_{\gamma}$ is defined as
\[\lie{\g}_{\gamma} = \{ x \in \lie{\g}_\C: \forall y \in \lie{\t}_\C,  [y,x] = \gamma(y)x \}\]
 and the sum is indexed over the set $\Phi$ of \emf{roots} of $\lie{\g}_\C$, which is the set of those $\gamma \in \lie{\t}_\C^*$ for which the corresponding eigenspace $\lie{\g}_{\gamma}$ is nontrivial. The set of roots carries a natural action of the \emf{Weyl group} of $\G$, which is the quotient $\W := N(\T)/\T$ of the normalizer of the torus $\T$ by the torus itself. \newline

Choosing a maximal solvable Lie subalgebra $\lie{b}$ in $\lie{\g}_\C$, called the \emf{Borel subalgebra}, determines a decomposition of $\Phi$ as a sum of the sets $\Phi^{+}$ of \emf{positive roots} and $\Phi^{-}$ of \emf{negative roots}, such that
\[\lie{b} = \lie{t}_\C \oplus \sum_{\gamma \in \Phi^{-}} \lie{\g}_{\gamma}.\]
A Lie subalgebra $\lie{p} \subset \lie{\g}_\C$ is called \emf{parabolic} if it contains some Borel subalgebra. A Borel subalgebra corresponds to a maximal connected solvable subgroup $B \subseteq \G_\C$, called a \emf{Borel subgroup}. A subgroup $P$ satisfying $B \subset P \varsubsetneq \G_\C$ for some Borel subgroup $B$ is called a \emf{parabolic subgroup}. One of the distinguishing features of parabolic subgroups (for semisimple Lie groups) is that the quotients $\G_\C/P$ are compact complex manifolds.

	\section{Symplectic geometry}\label{sc:symplectic}

A smooth manifold $\M$ is called a \emf{symplectic manifold} if it is equipped with a \emf{symplectic form} $\omega$, i.e. a closed, nondegenerate, skew-symmetric differential $2$-form. Morphisms in the category of symplectic manifolds are the \emf{symplectomorphisms}: 
\[\psi \in Symp(M_1, M_2) \iff \psi^* \omega_2 = \omega_1.\]

		\subsection{Hamiltonian actions and moment maps}\label{subsc:hamiltonian}

Let $\Ka$ be a Lie group with Lie algebra $\lie{k}$ and denote by $\lie{\ka}^*$ the dual of $\lie{\ka}$, equipped with a natural pairing $\langle - , -\rangle: \lie{\ka}^* \times \lie{\ka} \to \RR$.

Assume $\Ka$ acts on $\M$ by \emf{Hamiltonian symplectomorphisms} meaning that that the fundamental vector fields $v_{\xi}$ associated to the action satisfy 
\[d \h_{\xi} = \iota(v_{\xi}) \omega, \textrm{ for some smooth function } \h_{\xi}.\]
Such an action is called \emf{Hamiltonian} if the choice of functions $\h_{\xi}$ is consistent in the sense that the association $\xi \mapsto \h_{\xi}$ is a homomorphism of Lie algebras\footnote{The Lie algebra structure on the set $C^{\infty}(\M)$ of the smooth functions on $M$ is given by the Poisson bracket.} $\tilde{\mu} : \lie{\ka} \to C^{\infty}(\M)$. This homomorphism determines a map $\mu: \M \to \lie{\ka}^*$, called the \emf{moment map} of the action, by dualization.\footnote{The formula for the moment map is $\mu(x) := \tilde{\mu}(-)(x)$}

The moment map of a Hamiltonian action satisfy the following properties:
	
\begin{enumerate}
	\item Let $H \to \Ka$ be a Lie group homomorphism and $p_*: \lie{k}^* \to \lie{h}^*$ the induced Lie algebra homomorphism. Then if $\Ka$ acts on $\M$ with moment map $\mu_\Ka$, then the induced action of $H$ is also Hamiltonian, with moment map $\mu_H = p^* \circ \mu_\Ka$.
	\item Let $\M_1$, $\M_2$ be two symplectic manifolds equipped with Hamiltonian actions of $\Ka$ with moment maps $\mu_1: \M_1 \to \lie{\ka}^*$ and $\mu_2: \M_2 \to \lie{\ka}^*$. Then the diagonal action of $\Ka$ on $\M_1 \times \M_2$ is Hamiltonian with moment map $\mu_1 + \mu_2$. 
	\item If $\Ka, H$ act on $\M$ with moment maps $\mu_\Ka, \mu_H$ and the actions commute, then $\Ka \times H$ acts on $\M$ with moment map $\mu_\Ka \oplus \mu_H: \M \to \lie{\ka}^* \oplus \lie{h}^*$.
\end{enumerate}

For a torus action on a compact manifold $\M$, the following theorem due to Atiyah (\cite{atiyah1982}) and independently to Guillemin and Sternberg (\cite{guillemin1982}) describes the image of the moment map. 

\begin{thm}[Atiyah, Guillemin--Sternberg] \label{thm:convexity}
Let $T=(S^1)^m$ act on a compact connected symplectic manifold $\M$ in a Hamiltonian way. Then the image of the moment map is a convex polyhedron (the convex hull of the images of the fixed points of the action).
\end{thm}

		\subsection{Symplectic reduction}\label{subsc:reduction}

Let $\M$ be a compact symplectic manifold equipped with a Hamiltonian action of a compact Lie group $\Ka$, with moment map $\mu: \M \to \mathfrak{\ka}^* $. Assume $0$ is a regular value of $\mu$. The \emf{symplectic reduction} of $\M$ with respect to $\Ka$ is defined as 
\[ \M \git \Ka := \mu^{-1}(0) / \Ka.\]
The assumption that $0$ is a regular value of the moment map is equivalent to each $m \in \mu^{-1}(0)$ having a finite stabilizer. It ensures that $\mu^{-1}(0)$ is a manifold and $\Ka$ acts locally freely on it. In particular, the  symplectic reduction $\M \git \Ka$ is (at least) an orbifold, and moreover has an induced symplectic structure.

The \emf{Kirwan map} $\kappa$ is a map relating the cohomology of the symplectic reduction with the equivariant cohomology of the unreduced manifold,
\[ \kappa: \h^*_{\Ka}(\M) \to \h^*(\mu^{-1}(0)/\Ka),\]
and is defined as the following composition:
\[ \kappa: \h^*_{\Ka}(\M) \xrightarrow{i^*} \h^*_{\Ka}(\mu^{-1}(0)) \xrightarrow{(\pi^*)^{-1}}  \h^*(\mu^{-1}(0)/\Ka),\]
where $i^*$ is the map of $\Ka$-equivariant cohomology induced by the inclusion $i:\mu^{-1}(0) \to \M$ and $\pi^*$ is the natural isomorphism induced by the quotient map $\pi: \mu^{-1}(0) \to \mu^{-1}(0)/\Ka$.\footnote{The map $q$ induces an isomorphism on the rational cohomology, because the action of $\Ka$ on $\mu^{-1}(0)$ is locally free.}

\begin{remark}\label{rem:git}
Symplectic reduction is a certain quotient construction for symplectic manifolds. It is related to the GIT quotient in algebraic geometry by the Kempf--Ness theorem (\cite{kempf1979}), which establishes a bijection between the underlying sets of the GIT quotient and the symplectic reduction in the following sense. Consider an action of a reductive complex linear group $\G$ on a complex projective variety $\M$. The assumptions on $\G$ imply that $\G$ is a complexification of a compact real subgroup $\Ka < \G$, i.e. $\lie{\g} = \lie{\ka} + i \lie{\ka}$. Assume that $\G$ acts on $\M$ via $SL$-transformations on the projective space and the compact real subgroup $\Ka$ acts via $SU$-transformation, i.e.
\[
\begin{tikzcd}[column sep=tiny]
\Ka \arrow[d]& < & \G \textrm{      } \arrow[d] &  \curvearrowright & \M \arrow[d, hook] \\
SU(n+1) & < & SL(n+1, \C)  & \curvearrowright & \P^n
\end{tikzcd}
\]
In this situation the action of $\Ka$ in $\P^n$ preserves the standard symplectic form $\omega$ (the Fubini--Study form\footnote{We choose the Fubini--Study form normalized in such a way, that $\vol{\P^n}=1$.}) and moreover $\Ka$ acts by symplectomorphisms on $\M$ (\cite{kirwan1984cohomology}).  The Kempf--Ness theorem identifies the polystable $\G$-orbits with the zeros of the moment map, up to the action of $\Ka$. Semistable orbits are the ones that have a zero of the moment map in the closure. 

\end{remark}

\begin{remark}
In the examples considered in this dissertation the compact real subgroup in the linear algebraic group $\Ka < \G$ will be:
\[(S^1)^m < (\C^*)^m, SU(m) < SL(m, \C), Sp(n) < Sp(2n,\C), SO(n) < SO(n,\C)\]
\end{remark}

	\section{The nonabelian localization theorems}\label{sc:nonabelian}

The nonabelian localization theorems presented in this section relate the cohomology of the symplectic reduction of a symplectic manifold $\M$, taken with respect to a given Hamiltonian action of a compact Lie group $\Ka$, with the cohomology of the symplectic reduction $\M \git \S$ for a maximal torus $\S < \Ka$. The Jeffrey--Kirwan and Guillemin--Kalkman theorems provide formulas for push-forwards in singular cohomology with complex coefficients of the symplectic reduction $\M \git \Ka$. The outcome of the push-forward is computed as an iterated residue at infinity, taken with respect to a suitably chosen set of variables. \newline

Throughout this dissertation, the iterated residue at infinity of a complex function $f$ is defined as follows. Let $\z = (z_1,\dots, z_n)$ and let $f(\z)$ be a meromorphic function on $\C^n$. Assume the poles of $f$ form a normal crossing divisor. We define the \emf{iterated residue at zero} of the function $f$, which we denote by $\Res_{\z=0} f(\z)$, to be the coefficient at $z_1^{-1}\dots z_n^{-1}$ in the multivariate Laurent series expansion of $f$ at $0$. The multivariate Laurent series expansion in general may depend on the order of variables with respect to which the expansion is taken. However, if the function being expanded has poles which are normal crossing the result does not depend on the order (\cite{griffiths1979}). We define the \emf{iterated residue at infinity} of the function $f$ by the formula
\[\Res_{\z =\infty} f(\z) = (-1)^n \Res_{\z=0} \frac{f(\z^{-1})}{\z^2},\]
where $\z^2 = z_1^2 \dots z_n^2$ and $f(\z^{-1}) = f(z_1^{-1}, \dots, z_n^{-1})$.

		\subsection{Jeffrey--Kirwan nonabelian localization theorem}\label{subsc:jk}

Let $\M$ be a compact symplectic manifold equipped with a Hamiltonian action of a compact Lie group $\Ka$, with  moment map $\mu_\Ka: \M \to \mathfrak{k}^* $. Assume $0$ is a regular value of $\mu_\Ka$. Let $\S$ be a maximal torus in $\Ka$ acting on $\M$ by restriction of the $\Ka$-action with moment map $\mu_\S$. Denote by $\M \git \Ka$ the symplectic reduction of $\M$ by the action of $\Ka$. \newline

The Jeffrey--Kirwan nonabelian localization theorem states that the Kirwan map for the action of $\Ka$ is an epimorphism, and its kernel can be explicitly described in terms of intersection pairings. From the point of view of the residue formulas in equivariant cohomology, the most interesting result is the following (\cite{jeffrey1995}, Theorem 8.1):

\begin{thm}[Jeffrey--Kirwan] Let $\omega$ be a symplectic form on $\M$ and $\omega_0$ the induced symplectic form on $\M \git \Ka$. 
Let $\eta \in \h^*_\Ka(\M)$. Let $[\M \git \Ka]$ be the fundamental class of $\M \git \Ka$ in $\h^*_\Ka(\M)$. Let $\Phi^{+}$ and $\W$ be, respectively, the set of positive roots and the Weyl group of $\Ka$. Denote by $\upvarpi$ the product of the roots of $\Ka$,\footnote{In the original formulation in \cite{jeffrey1995} $\upvarpi$ is the product of positive roots of $\Ka$, so that the in the formulas $\upvarpi$ is replaced by $(-1)^{|\Phi^{+}|} \upvarpi^2$ .} 
\[\upvarpi = \prod_{\gamma \in \Phi}\gamma.\] 
Then one can choose a subset $\mathcal{F}$ of the set of components of the fixed point set of the action of $\S$ such that the following formula holds.
\[\kappa(\eta) e^{i \omega_0}[\M \git \Ka] = \]
\[ = \frac{1}{(2 \pi )^{k-s} |\W| \vol(\S)} \Res\bigg{(}\upvarpi \sum_{F \in \mathcal{F}} e^{i \mu_\S(F)} \int_F \frac{i^*_F(\eta e^{i \omega})}{e_F} \bigg{)}.\]
For $F \in \mathcal{F}$, the map $i_F$ is the inclusion of $F$ into $\M$ and $e_F$ is the equivariant Euler class of the normal bundle to $F$ in $\M$. The constant $\vol(\S)$ is the volume of the torus $\S$,  and $k,s$ denote the dimensions of $\Ka$ and $\S$ respectively.\footnote{The residue in the Jeffrey--Kirwan theorem is defined as a certain contour integral (see def. 8.5 in \cite{jeffrey1995}).} \label{thm:jk}
\end{thm}

		\subsection{Guillemin--Kalkman nonabelian localization theorem}\label{subsc:dendrite}\label{subsc:gk} 

The nonabelian localization theorem of Jeffrey--Kirwan has a more compact reformulation by Guillemin and Kalkman (\cite{guillemin1996}), stated as Theorem \ref{thm:gk} below. The Guillemin--Kalkman theorem reduces the push-forward of the image under the Kirwan map of a cohomology class $\alpha \in \h^*_\Ka(\M)$ to a sum of local contributions at fixed points.\footnote{Or fixed components, if the fixed points are not isolated.} The indexing subset of fixed points is constructed using the moment map for the action, as we briefly recall below. For the detailed description see \cite{guillemin1996}, Sect.~3. \newline

Recall that by Theorem \ref{thm:convexity} the image of $\mu$ is a convex polytope. Let us denote this polytope by $\Delta \subseteq \lie{\s}^*$ and assume $0$ lies in its interior, for otherwise the symplectic reduction is trivial, meaning $\M \git \S=\varnothing$. The set $\Delta^0$ of the regular values of $\mu$ is a disjoint union of convex polytopes 
\[\Delta^0 = \Delta^0_1 \cup \dots \cup \Delta^0_k,\]
and by assumption $0$ lies inside one of the polytopes $\Delta^0_j$. For any chosen element $\theta$ in the weight lattice of $\lie{\s}^*$, one can consider the ray through the origin in the direction of $\theta$,
\[l = \{ t \theta: t \in [0, \infty) \}. \]
Let us choose $\theta$ in such a way that the ray $l$ does not intersect any of the walls of $\Delta_i^0$ of codimension greater than one and hence intersects the codimension one walls transversally. Then the Lie subalgebra $\lie{h} \subseteq \lie{\s}$ defined as
\[\lie{h} = \{ v \in \lie{\s}: \langle \theta, v \rangle = 0 \}\]
is the Lie algebra of a codimension one subtorus $H \subseteq \S$. The assumptions made on the ray $l$ and hence on the element $\theta$ imply that the moment map $\mu$ is transverse to $l$ and the action of $H$ on $\mu^{-1}(l)$ is locally free. 

Let us choose the ray $l$ as described above, and call $l$ a main branch. Next, for every intersection point $p_j$ of $l$ with the codimension one walls of $\Delta_i^0$ one chooses a ray $l_j$ starting at $p_j$ and not intersecting any codimension three walls of $\Delta_i^0$. The rays $l_j$ are called secondary branches. One continues the procedure by considering the intersection points of secondary branches $l_j$ with codimension $2$ walls of $\Delta_i^0$, and at each such point chooses ternary branches (rays not crossing the codimension 4 walls), and so on. Finally one arrives at a vertex of the moment polytope, which is an image of a fixed point (or fixed component) by the convexity theorem \ref{thm:convexity}. One obtains what is called a dendrite $\mathcal{D}$: a set of branches, consisting of sequences of rays $(l, l^{(1)}, \dots , l^{(n)})$, where $l^{(j)}$ is a branch constructed in step $j$, and a set of points $(0, p_1, \dots, p_n)$ on those branches, such that the branch $l^{(i)}$ 
starts at the point $p_i$ and intersects the codimension $i+2$ wall at $p_{i}$. \newline

We define the subset $D \subseteq \M^{\S}$ to be the set of those fixed points whose images under the moment map are the endpoints of all the branches of the dendrite $\mathcal{D}$.  \newline

Using the dendrite procedure described above one obtains a sequence of subtori in $\S$, determined by each branch in the dendrite
\[ \{ 0 \} = \S^{(0)} < \S^{(1)} < \cdots < \S^{(n)} = \S\]
with $\dim \S^{(i)} = i$. Each such sequence of tori determines a sequence of symplectic submanifolds of $\M$
\[ \M = \M^{(0)} \supset \M^{(1)} \supset \cdots \supset \M^{(n)} = \M^{\S}\]
such that $\M^{(i)}$ is a connected component of the fixed point set of the subtorus $\S^{(i)}$. One can choose a basis $z_1, \dots, z_n$ of $\lie{s}$ in such a way, that for each $i$ the elements $z_1^*, \dots, z_i^*$ form a basis of the integer lattice in the Lie algebra of $\S^{(i)}$. The following theorem describes the push-forward in the cohomology of $\M \git \Ka$.

\begin{thm}[Guillemin--Kalkman]\label{thm:gk}
Let $\M$ be a compact symplectic manifold equipped with a Hamiltonian action of a compact Lie group $\Ka$. Let $\S$ denote the maximal torus in $\Ka$ and assume the $\S$-fixed points are isolated. Let $\W = N_{\Ka}(\S)/\S$ denote the Weyl group of $\Ka$. Denote by $\z = (z_1,\dots,z_n)$ the basis of $\lie{s}^*$ chosen as above. For a fixed point $p$, let $i_p$ denote the inclusion of $p$ into $\M$ and let $\{ \lambda_i(p)\}_{i=1}^{\dim \S}$ be the weights of the isotropy representation of $\S$ at $p$. Finally let $\upvarpi \in \C[\z]$ be the product of the roots of $\Ka$.
Then for $\alpha \in \h^*_\Ka(\M)$ one has
\[ \int_{\M \git \Ka} \kappa(\alpha)  = \frac{1}{|\W|} \sum_{p \in D} \Res_{\z=\infty} \frac{\upvarpi \cdot i^*_p \alpha }{\prod \lambda_i (p)}.\]
\end{thm}
Note that the set $D$ depends only on the action of the torus $\S$ and the choice of the splittings of $\S$ as a product of one dimensional tori, and not on $\Ka$. The reduction of the push-forward from the quotient with respect to the group $\Ka$ to its maximal torus $\S$ is compensated by the $\upvarpi$ factor under the residue and the $\frac{1}{|\W|}$ constant factor. For a detailed explanation see \cite{martin2000} and Sect.~\ref{subsc:martin} below. \newline

The set $D$ depends on the choice of the dendrite, however the final sum over the points $p \in D$ in the above theorem does not. The residue at $\z \in \lie{s}^*$ is the iterated residue in the sense of Sect.~\ref{sc:nonabelian}, as the expression under the residue is a rational function in $\z$. The $\S$-equivariant cohomology of $\M$ can be computed from the Cartan complex
\[\tilde{\Omega} = \Omega^*(\M)^{\S} \otimes \C[\z]\] 
and the restriction of a form $\alpha \in \h^*_\Ka(\M)$ to a fixed point $p$ is an element of $\h^*_{\Ka}(p)=\C[\z]^\W$. The products of weights $\lambda_i(p)$ and $\upvarpi$ lie in $\lie{s}^*$ hence can also be written as polynomials in $\z$. 

		\subsection{Martin integration formula}\label{subsc:martin}

Let $\M$ be a compact symplectic manifold equipped with a Hamiltonian action of a compact Lie group $\Ka$. Let $\S$ be a maximal torus in $\Ka$, acting on $\M$ by restriction of the $\Ka$-action. Let $\mu_\Ka: \M \to \lie{k}^*$ be the moment map for the $\Ka$-action, and $\mu_\S$ its restriction to $\lie{\s}^*$. The following theorem of Martin (\cite{martin2000}) relates the cohomology rings (with rational coefficients) of the symplectic reductions $\M \git \Ka$ and $\M \git \S$.

\begin{thm}[Martin]\label{thm:martin}
Let $\W$ denote the Weyl group of $\Ka$, acting naturally on $\M \git \S$.  To any weight $\alpha$ of $\S$ we associate the equivariant line bundle $L_{\alpha}$ over $\M \git \S$ and define $e(\alpha):=e(L_{\alpha})$ to be the Euler class of this bundle. Denote $e = \prod_{\alpha \in \Phi} e(L_{\alpha})$ and let $\ann(e)$ be the ideal in $\h^*(\M \git \S;\mathbb{Q})^\W$ consisting of cohomology classes whose cup product with $e$ vanishes. There is a natural ring isomorphism 
\[ \h^*(\M \git \Ka; \mathbb{Q}) \simeq \frac{\h^*(\M \git \S;\mathbb{Q})^\W}{\ann(e)}.\]
\end{thm}

In particular, with the notation as in theorem above, one has the following formula relating the push-forwards to the point in the rational cohomology of $\M \git \Ka$ and $\M \git \S$.  

\begin{thm}[Martin Integration Formula]
Let $i:\mu_{\Ka}^{-1}(0)/\S \to \M \git \S$ be the map induced by inclusion $\mu_{\Ka}^{-1}(0) \hookrightarrow \mu_{\S}^{-1}(0)$ and consider the map $\pi: \mu_{\Ka}^{-1}(0)/\S \to \M \git \Ka $, which is a fibration with fiber $\Ka/\S$. The class $\tilde{a} \in \h^*( \M \git \S)$ is called a \emf{lift} of the class $a \in \h^*( \M \git \Ka)$ if $\pi^* a = i^* \tilde{a}$. Then if $\tilde{a}$ is a lift of $a$, the following formula holds
\[\int_{\M \git \Ka} a = \frac{1}{|\W|} \int_{\M \git \S} \tilde{a} \cdot e.\]
\label{thm:mit}
\end{thm} 

\chapter{Equivariant Jeffrey--Kirwan localization theorem}\label{ch:equiv-jk}

The Guillemin--Kalkman reformulation \ref{thm:gk} (\cite{guillemin1996}) of the Jeffrey--Kirwan nonabelian localization theorem \ref{thm:jk} (\cite{jeffrey1995}) provides a residue type formula for the non-equivariant push-forwards, as described in Ch.~\ref{ch:prelim}, Sect.~\ref{sc:nonabelian}. We generalize the result to equivariant push-forwards for torus actions by using approximation spaces for the Borel model of the equivariant cohomology and reducing the equivariant push-forwards to non-equivariant ones. We begin with reviewing the statement of the Guillemin--Kalkman formula \ref{thm:gk}, adapting it to the equivariant setting. We also weaken some of the original assumptions on the compactness of the spaces in question as the proof of the residue-type formula can be improved to work without them. We restrict our attention to the case of torus actions, and later in Ch.~\ref{ch:martin} we prove a theorem providing a transition between the action of a compact Lie group $\Ka$ and a maximal torus $\S < \Ka$, hence obtaining the general result. \newline

The surjectivity of the Kirwan map relies on the assumption of $\M$ being compact, the residue formula however can be proved under weaker assumptions, as described in Sect.~\ref{sc:convexity} at the end of this chapter. For simplicity and compactness of the formulation of the theorems we first prove the equivariant version of the Guillemin--Kalkman formula with additional assumption of $\M$ being compact, and later describe how the compactness assumption can be weaken.

\section{The equivariant Guillemin--Kalkman residue formula for torus actions}\label{sc:equiv-gk}
\sectionmark{Equivariant Guillemin--Kalkman formula}

Let $\M$ be a compact symplectic manifold equipped with Hamiltonian actions of two tori $\T$ and $\S$ and assume the two actions commute. Denote by $\mu_\S$ the moment map for the action of $\S$ and assume $0$ is a regular value of $\mu_{\S}$.
Assume additionally that the set $\mu_{\S}^{-1}(0)$ is $\T$-invariant.\footnote{This is the case for example if the moment map $\mu_
S$ is $\T$-equivariant, meaning that it comes from a moment map $\mu: \M \to (\lie{\t} +  \lie{\s})^*$.} We define a $\T$-equivariant analogue of the Kirwan map for the $\S$-action:
\[ \kappa_{\T}: \h^*_{\T \times \S}(\M) \xrightarrow{i^*} \h^*_{\T \times \S}(\mu_\S^{-1}(0)) \xrightarrow{(\q^*)^{-1}} \h^*_{\T}(\mu_\S^{-1}(0)/\S)=\h^*_{\T}(\M \git \S) ,\]
which is defined as the composition of the map induced on equivariant cohomology by the inclusion $i: \mu_\S^{-1}(0) \hookrightarrow \M$ (or, equivalently, as the map induced on singular cohomology by the inclusion $\E\T \times^{\T} \mu_\S^{-1}(0) \hookrightarrow \E\T \times^{\T} \M$) with the inverse of the natural isomorphism $\q^*: \h^*_{\T}(\mu_\S^{-1}(0)/\S) \to \h^*_{\T \times \S}(\mu_\S^{-1}(0))$.\footnote{ By the assumption that $\mu_\S^{-1}(0)$ is $\T$-invariant and that the two actions commute, the projection on the second factor $\tilde{q}: \E\T \times \mu_\S^{-1}(0)  \to \E\T \times \mu_\S^{-1}(0) / \S$ descends to quotients, yielding a projection  $q: \E\T \times^\T \mu_\S^{-1}(0)  \to \E\T \times^\T \mu_\S^{-1}(0) / \S$. The map $q^*$ is an isomorphism because the action of $\S$ on $\mu_\S^{-1}(0)$ is locally free by the assumption that $0$ is a regular value of $\mu_{\S}$.}

\begin{remark} Since the actions of $\S$ and $\T$ are assumed to commute, the equivariant Kirwan map defined above is just the standard Kirwan map on the space $\E\T \times^{\T} \M$. By Lemma \ref{lem:approx} one can approximate $\h^*_{\T \times \S}(\M) $ by the $\S$-equivariant cohomology of the approximation spaces $\E_m \times^{\T} \M$. Since $\T$ is a torus we can take $\E_m = (S^{2m+1})^{\dim \T}$. In particular the map $\kappa_{\T}$ can be approximated by the non-equivariant Kirwan maps $\kappa_m: \h^{*}(\E_m \times^{\T} \M) \to \h^*(\E_m \times^{\T} \M \git \S)$, meaning that for every $\alpha \in \h^k_{\T \times \S}(\M)$ one can find an $m \in \mathbb{N}$ such that $\kappa_{\T}(\alpha) = \kappa_m(\alpha)$, where on the right-hand side we identify $\alpha \in \h^k_{\T \times \S}(\M)$ with the corresponding element in $\h^{k}(\E_m \times^{\T} \M)$.
\end{remark}

\begin{remark} The term "equivariant Kirwan map" has been used by several authors in various contexts, different from the one presented here. For example in \cite{goldin2002} Goldin introduces under the same name the map
\[\kappa_{\T}: \h^*_{\Ka}(\M) \to \h^*_{\Ka/\S}(\mu^{-1}(0)/\S),\]
where $\S < \Ka$ is a normal subtorus of a compact Lie group $\Ka$. 
\end{remark}

Let $\M_{critical}$ be the set of critical points of the moment map $\mu_\S$. By the results of Guillemin and Sternberg \cite{guillemin1982} it admits a decomposition into a finite union 
\[\M_{critical}= \bigcup_{j} \M_j,\]
 where each $\M_j$ is a fixed point set of a one-dimensional subgroup $\S_j$ of $\S$. Consider the equivariant Kirwan map for the action of $H_j=\S/\S_j$ on $\M_j$
\[\kappa_{\T}^{j}: \h^*_{\T \times H_j}(\M_j) \to \h^*_{\T}(\M_j \git H_j),\]
and let $\kappa_\T$ be the $\T$-equivariant Kirwan map for the action of $\S$ on $\M$
\[\kappa_{\T}: \h^*_{\T \times \S}(\M) \to \h^*_{\T}(\M \git \S).\] 
Let $i_j$ be the inclusion $\M_j \to \M$ and let $e^{\T \times \S}(\nu_j)$ be the equivariant Euler class of the normal bundle $\nu_j$ to $\M_j$ in $\M$. We choose a generator $x_j$ of $\h^*_{\S_j}(pt)$. Define 
\[res_j(\alpha):= res_{x_j=\infty}\frac{i^*_j \alpha}{e^{\T \times \S}(\nu_j)}.\]
 A direct consequence of the Guillemin--Kalkman theorem \ref{thm:gk} is the following result.

\begin{thm}
Let $\alpha \in \h^{*}_{\T}(\M \git \S)$. The $\T$-equivariant push-forward to a point of $\kappa_{\T}(\alpha)$ is given by the following formula:
\[\int_{\M \git \S} \kappa_{\T}(\alpha) = \sum_{j \in D} \int_{\M_j \git H_j} \kappa_{\T}^{j} (res_j(\alpha)).\]
The summation is taken with respect to the subset of fixed components $D$, as described in Ch.~\ref{ch:prelim}, Sect.~\ref{subsc:dendrite}.
\label{thm:equiv-gk}
\end{thm}

Note that the above Theorem \ref{thm:equiv-gk} and Theorem \ref{thm:gk} have almost identical formulations, differing only by replacing all the cohomology rings  $\h^*(-)$ by their $\T$-equivariant counterparts $\h^*_{\T}(-)$ and consequently replacing all subordinate notions like characteristic classes, Kirwan maps etc. by their equivariant analogues. Hence Theorem \ref{thm:equiv-gk} is a straightforward consequence of the Theorem \ref{thm:gk}, applied to the approximation spaces $\E_m \times^{\T} \M$. More precisely, we prove the following proposition describing the $\T$-equivariant push-forward  in terms of the standard push-forwards in singular cohomology.

\begin{proposition}\label{fact:integrals}
Let $\M$ be a compact manifold with Hamiltonian actions of a torus $\T$ and a torus $\S$, and assume that the two actions commute and $\mu^{-1}_{\S}(0)$ is $\T$-invariant. We fix a decomposition $\T = (S^1)^n$. Assume $\alpha \in \h^k_{\T \times \S}(\M)$ and let $p: \M \to pt$. Then the equivariant push-forward $p_*: \h^*_\T(\M \git \S) \to \h^*_\T(pt) = \C[t_1,\dots, t_n]$ of the image of Kirwan map on $\alpha$ is given by the formula
\[p_* \kappa_{\T}(\alpha) = \sum_I \beta_I t^I,\]
where the summation is over the multi-index $I = (i_1, \dots, i_n)$ of length $n= \dim \T$, such that $|I| = \frac{1}{2}(k- \dim \M \git \S)$. The coefficients $\beta_I$ are given by
\[\beta_I = \int_{\M_I \git \S} j^* \kappa_{\T}( \alpha),\]
where $\M_I = ( S^{2i_1 + 1} \times S^{2 i_2 + 1} \times \dots \times S^{2 i_l + 1} ) \times^\T \M$ and $j: \M_I \hookrightarrow \E\T \times^{\T} \M$ denotes inclusion. 
\end{proposition}

\begin{proof}[Proof of Proposition \ref{fact:integrals}]
More generally, if $\M$ is a compact manifold with a $\T$-action and $p: \M \to pt$, then the push-forward $p_*: \h^*_\T(\M) \to \h^*_\T(pt)$ is given by the formula
\[p_* \alpha = \sum_I \beta_I t^I,\]
where $I = (i_1, \dots, i_n)$ is the multi-index satisfying $|I| = \frac{1}{2}(\deg \alpha- \dim \M)$. The coefficients $\beta_I$ are given by 
\[\beta_I = \int_{\M_I} j^* \alpha,\]
with $\M_I = ( S^{2i_1 + 1} \times S^{2 i_2 + 1} \times \dots \times S^{2 i_l + 1} ) \times^\T \M \xrightarrow{j} \E\T \times^\T \M$. It can easily be seen using de Rham cohomology in which the push-forward is given by integration along fibers. We can use de Rham cohomology since $\M$ is a manifold and by Theorem \ref{lem:approx} we can perform computations in the  approximation spaces for the chosen model of $\E\T = S^{\infty} \times \dots \times S^{\infty}$ which are compact manifolds as well, for example $\E_m = S^{2m+1} \times \dots \times S^{2m+1}$. \newline

The push-forward $p_*$ in equivariant cohomology is approximated by the push-forwards for maps $p_m: \E_m \times^\T \M \to \B_m$. Writing $(p_m)_* \alpha = \sum_I \beta_I^m t^I$, the coefficients $\beta_I^m$ are given by
\[\beta_I^m = \int_{S^{2i_1 + 1}\times \dots \times S^{2i_n + 1}} (p_m)_* \alpha  = \int_{p_m^{-1}(S^{2i_1 + 1}\times \dots \times S^{2i_n + 1})} j^* \alpha = \int_{\M_I} j^* \alpha. \]

Applying this to $\M = \M \git \S$ and using the commutativity condition for the two actions proves Prop.~\ref{fact:integrals}. 
\end{proof}

Theorem \ref{thm:equiv-gk} reduces the push-forward over $\M \git \S$ to a sum of push-forwards over $\M \git H_j$, where $H_j$ are subtori of codimension $1$ in $\S$, obtained as quotients of $\S$ by certain one-dimensional tori $\S_i$. In order to proceed inductively, one needs to ensure that one can apply the same procedure to the varieties $\M \git H_j$, which is achieved by carefully choosing the one-dimensional tori $\S_i$. We check the needed details in Sect.~\ref{subsc:equiv-induction}. Once the technical problems are overcome, one arrives at the following $\T$-equivariant generalization of the Guillemin--Kalkman Theorem \ref{thm:gk}.

\begin{thm}[Equivariant Guillemin--Kalkman Theorem]\label{thm:equiv-gk-main}
Let $\alpha$ be a cohomology class in $\h^{*}_{\T \times \S}(\M)$. One can choose a subset $\mathcal{F}$ of connected components of the fixed point set $\M^{\S}$ and for each $F \in \mathcal{F}$ a sequence of subtori
\[ \S_F^{(1)} \subseteq \S_F^{(2)} \subseteq \dots \subseteq \S_F^{(N)} = \S,\] 
with $\dim \S_F^{(i)} = i$ and a basis $x_{F,1},\dots,x_{F,n}$ of $\lie{\s}^*$ such that for each $i$ the dual elements $x_{F,1}^*, \dots, x_{F,i}^*$ form a basis of the integer lattice in the Lie algebra $\lie{s}_F^{(i)}$ of $S_F^{(i)}$, such that the $\T$-equivariant push-forward to a point of $\kappa_{\T}(\alpha)$ is given by the following formula:
\[\int_{\M \git \S} \kappa_{\T}(\alpha) = \sum_{F \in \mathcal{F}} \int_F res_{x_{F,n}=\infty} \dots res_{x_{F,1}=\infty} \frac{i_F^*\alpha}{e^{\T \times \S}(\nu)}.\]
In the above formula $e^{\T \times \S}(\nu)$ denotes the $\T \times \S$-equivariant Euler class of the normal bundle to $F$ in $\M$. 
\end{thm}

In the next two sections we describe the details of the proofs. Sect.~\ref{subsc:s1} contains a detailed proof of the case of the $S^1$-action, which is a modification of the original Guillemin--Kalkman proof in the equivariant case. Sect.~\ref{subsc:equiv-induction} provides details on the inductive procedure.
		
		\subsection{The case of the $S^1$-actions via approximation spaces}\label{subsc:s1}		

For one-dimensional torus actions one can prove a more general statement than claimed in Theorem \ref{thm:equiv-gk}. All symplecticity assumptions are in fact superfluous, one only needs a compact orientable manifold with boundary $(\M, \partial \M)$ together with an $S^1$-action which is locally free on the boundary. The quotient  $\partial \M / S^1$ plays the role of the symplectic reduction. The symplectic structure and the properties of the moment map are essential for higher dimensional tori actions. \newline

We consider a compact oriented $S^1$-manifold with boundary $(\M, \partial \M)$ and assume $S^1$ acts locally freely on $\partial \M$. The inclusion $ i: \partial{\M} \to \M$ induces the map on $S^1$-equivariant cohomology \[i^*: \h^*_{S^1}(\M) \to  \h^*_{S^1}(\partial{\M}).\]
 Recall that if the action on $\partial{\M}$ is locally free, there is a canonical isomorphism
\[  q^*: \h^*(\partial{\M} / S^1 ) \to \h^*_{S^1}(\partial{\M}), \]
induced by the map $q: \E S^1 \times^{S^1} \partial{\M} \to \partial{\M} / S^1 $. We can therefore define the ``Kirwan map'' for this action as 
\[ \kappa = (q^*)^{-1} \circ i^* : \h^*_{S^1}(\M) \to \h^*(\partial{\M} / S^1 ). \]
Analogously as in Sect.~\ref{sc:equiv-gk}, if $\partial{\M}$ is $\T$-invariant we define the $\T$-equivariant Kirwan map in the $\T$-equivariant cohomology
\[ \kappa_\T = (q^*)^{-1} \circ i^* : \h^*_{\T \times S^1}(\M) \to \h^*_{\T}(\partial{\M} / S^1 ). \]

Note that the name ``Kirwan map'' in the literature is reserved for the Kirwan map associated with symplectic reduction as defined in Ch.~\ref{ch:prelim}, Sect.~\ref{subsc:reduction}. To recover the original setup one can consider a Hamiltonian action of $S^1$ with the moment map $\mu$ on a manifold $\M^{'}$. Then the subspace $\M := \{ p \in \M^{'} : \mu(p) \geq 0 \}$ is a manifold with boundary $\partial{\M} = \mu^{-1}(0)$ and the ``Kirwan map'' defined as above coincides with the standard Kirwan map. \newline

\begin{thm}\label{thm:equiv-s1}
Let $(\M, \partial \M)$ be a compact oriented $\T \times S^1$-manifold with boundary and assume the $S^1$-action to be locally free on $\partial \M$ and $\partial \M$ is $\T$-invariant.  Let $\int\limits_{\partial \M/ S^1} \!\!\!-$ denote the $\T$-equivariant push-forward to a point and let $\kappa_\T$ be the Kirwan map defined above. Then the following formula holds:
\[\int\limits_{\partial \M/ S^1} \kappa_{\T} (\alpha)  = \sum_{k=1}^N \int_{\F_k} res_{x=\infty} \frac{i^*_k \alpha}{e(\nu_k)}.\]
In the above formula $x$ is a generator of the $S^1$-equivariant cohomology of a point, $\h^*_{S^1}(pt) \simeq \C[x]$ and the varieties $\F_k$ are connected components of the fixed point set of the $S^1$-action, with normal bundles in  $\M$ denoted by $\nu_k $.
\end{thm}

\begin{proof}[Proof of Theorem \ref{thm:equiv-s1}]

We begin by describing the Kirwan map $\kappa_\T$ in terms of the Cartan complex. It is defined as a composition of the restriction map and the inverse of the isomorphism induced by the projection $q: \partial{\M} \to \partial{\M} / S^1$, so we need to describe the map $(q^*)^{-1}$ in detail. To do this we use the Cartan model of the $S^1$-equivariant cohomology, following the calculation in \cite{guillemin1996}.
The Cartan complex for an $S^1$-action on a manifold $\M$ is 
\[ \tilde{\Omega} = \Omega^*(\M)^{S^1} \otimes \C[x] \]
with differential $\tilde{d} = d \otimes 1 + \iota(v) \otimes x$,  where $\iota(v)$ denotes the contraction with the vector field generated by the unit vector $v$ tangent to $S^1$ at $1$. Since we want to describe the map $(q^*)^{-1}$, whose domain is $\h^*_{\T \times S^1}(\partial{\M})= \h^*_{S^1} (\E\T \times^{\T}\partial{\M})$ we need to describe the $S^1$-cohomology of the space $\E\T \times^{\T}\partial{\M}$. However, for technical reasons we use approximation spaces $\E_m \times^{\T} \partial{\M}$, which we denote by $\partial \M_m$, which are the boundaries of the approximation spaces $\M_m$ of $\M$. We work with the Cartan complexes
\[ \tilde{\Omega}^*_m = \Omega^*(\M_m)^{S^1} \otimes \C[x] \]
\[ \tilde{\Omega}^*_m (\partial \M) = \Omega^*(\partial \M_m)^{S^1} \otimes \C[x]. \]
In particular, for $k \in \mathbb{Z}$ we choose an $m \in \mathbb{Z}$  such that  $\h^i_{\T \times S^1}(\M) \simeq \h^i(\M_m)$ for $i \leq k$ and identify an element $\alpha \in  \h^k_{\T \times S^1}(\M)$ with the class of an element in the Cartan complex $\tilde{\Omega}_m^*$, which we also denote by $\alpha$. By abuse of notation we will also denote by $\alpha$ its restrictions to $\partial \M$, $\partial \M_m$ and its representants in the corresponding Cartan complexes. \newline

Let $\alpha \in \tilde{\Omega}^k_m(\partial \M)$ be an equivariantly closed form. Our aim is to describe the inverse of the isomorphism
\[q^*: \h^*_{\T}(\partial{\M} / S^1 ) \to \h^*_{\T \times S^1}(\partial \M). \]
The above isomorphism on cohomology implies that the form $\alpha$ can be written as
\[\alpha = \tilde{d} \nu + q^* \gamma,\]
for some closed forms $\nu \in \tilde{\Omega}^{k-1}_{m}(\partial \M)$ and $\gamma \in \Omega^k(\partial \M_m / S^1)$. In the non-equivariant case the explicit forms $\nu$ and $\gamma$ are constructed in the proof of the Berline--Vergne localization formula in \cite{berline1983} for $\deg \alpha \geq \dim \M_m$ and improved to work for $\deg \alpha \geq \dim \M_m-1$ by Guillemin and Kalkman in \cite{guillemin1996}. Since the original proof depends on the degree of the form $\alpha$ with relation to the dimension of the manifold, we repeat the proof to make sure choosing the approximation spaces does not impact any step of the proof.  \newline

Let $\theta$ be the connection form on $\partial \M_m$ as defined in Sect.~\ref{subsc:cartan}, in particular $\theta$ is an $S^1$-invariant $1$-form such that $\iota(v)\theta = 1$.
Consider the element defined by the following power series expansion at infinity, with coefficients in $\Omega^*(\partial \M_m)^{S^1}$,
 \[ \nu = \frac{\theta \alpha }{x + d \theta} = \frac{\theta \alpha}{x} \sum_{n \geq 0} \big{(}\frac{-d \theta}{x}\big{)}^n. \]
Note that $\nu$ lies in $\tilde{\Omega}^*_m(\partial \M)[x^{-1}] \subseteq \Omega^*(\partial \M_m)^{S^1} \otimes \C[x] \llbracket x^{-1} \rrbracket $. Formally, 
\[\tilde{d}\nu = \alpha.\]
Since we study push-forwards in the cohomology of $\partial \M_m \git S^{1}$, we can assume $\alpha$ has degree $k \geq \dim \partial \M_m -1$, since otherwise the push-forward is zero.

\begin{remark}\label{rem:degree}
Note that one cannot make arbitrary assumptions on the degree of $\alpha$. By working with approximation spaces we automatically impose some restrictions on the degree of $\alpha$, because the isomorphism
\[\h^i_{S^{1}}(\E\T \times \M) \simeq \h^i_{S^{1}}(\E_m \times \M)\]
is only valid for gradations $i$ smaller than a certain integer $k(m)$ as stated in the Lemma \ref{lem:approx}. For an action of an $n$-dimensional torus we can take $\E_m = S^{2m_1-1} \times \dots \times S^{2m_n-1}$, in which case $k(m) = \min_{i} \{ 2m_i-1 \}$. In particular, it usually won't be possible to have both $\deg \alpha = k = \dim \partial \M_m -1$ and $k < k(m)$. To avoid this problem we first choose $m$ such that $\deg \alpha = k = \dim \partial \M_m -1$ and then choose $m^{'}$ (possibly a lot bigger than $m$) such that $\h^k_{S^{1}}(\E\T \times^\T \M) \simeq \h^k_{S^{1}}(\E_{m^{'}} \times^\T \M)$ and perform all the computations in $\h^k_{S^{1}}(\E_{m^{'}} \times^\T \M)$, restricting them to $\h^k_{S^{1}}(\E_{m} \times^\T \M)$ via the inclusion $\E_{m} \times^\T \M \hookrightarrow \E_{m^{'}} \times^\T \M$.

\end{remark}

Assuming $\deg \alpha = \dim \partial \M_m -1$ and writing $\alpha$ as a polynomial in $x$ with coefficients in $\Omega^*(\partial \M_m)^{S^1}$ we have 
\[\alpha = \sum_{i} \alpha_i x^{i} \]
and substituting it into the definition of $\nu$ gives 
\[ \nu = \sum_{n,i} \theta \alpha_i (-d \theta)^n x^{i-n-1}.\]
The coefficients of this series lie in $\Omega^*(\partial \M_m)^{S^1}$, hence they vanish in degrees higher than $\dim \partial \M_m$. In particular, the coefficient of $x^{i-n-1}$ in the series defining $\nu$ can only be nonzero provided that
\[\deg (\theta \alpha_i (-d \theta)^n) = 1 + \deg(\alpha_i)  + 2n \leq \dim \partial \M_m. \]
Since $\deg \alpha = \dim \partial \M_m -1$, the only non-zero coefficients appear when $n \leq i$, so the only powers of $x$ that can occur in the Laurent series of $\nu$ are the positive ones and $x^{-1}$, so we can write
\[\nu = \nu_0 + \beta x^{-1}. \]
Recall that the residue at infinity is the coefficient at $x^{-1}$ in the Laurent series expansion at infinity of a function, hence 
\[ \beta = res_{x=\infty} \nu \in \Omega^*(\partial \M_m)^{S^1}. \]
Note that even though $\nu$ was formally defined in the localized complex $\Omega^*(\partial \M_m)^{S^1} \otimes \C[x] \llbracket x^{-1}\rrbracket$, both $\nu_0$ and $\beta$ are well-defined forms in the nonlocalized complexes. Since $\tilde{d}\nu =\alpha$, if one writes $\nu$ as $\nu = \nu_0 + \beta x^{-1}$ one gets (from the definition of the differential in the Cartan complex)
\[\alpha = \tilde{d} \nu_0 + \iota(v)\beta,\]
and the form $\iota(v)\beta$ is $S^1$-invariant and horizontal, so it comes from a form $\gamma \in \Omega^k(\partial \M_m / S^1)$,
\[\iota(v)\beta = q^* \gamma.\]
This shows that the map $(q^*)^{-1}$ is given by the formula  
\begin{equation}\label{eq:q-inverse}
(q^*)^{-1}(\alpha) = res_{x=\infty} \iota(v) \frac{\theta \alpha}{x + d \theta}
\end{equation}
Since $v$ is the generating field for the $S^1$-action, contraction with $v$ coincides with the push-forward $q_*$, which is given by integration along the fiber of $q: \partial \M_m \to \partial \M_m \git S^{1}$, hence
\[(q^*)^{-1}(\alpha)= res_{x=\infty} \pi_{*}  \frac{\theta \alpha}{x + d \theta}.\]

Let $\{ (\M_m)_i \}_{i=1, \dots, N}$ be the connected components of the fixed point set of the action of $S^1$, and let $\{ U_i \}_{i=1, \dots, N}$ be pairwise disjoint tubular neighbourhoods of the sets $(\M_m)_i$ such that $U_i \cap \partial \M_m = \varnothing$. Let $\alpha \in \h^*_{S^1}(\M)$ and let $\nu,\theta$ be as above. Assume that $\deg \alpha = \dim \partial \M_m -1$. Let us extend the form $\theta$ to $\M_m \setminus  \M_m^{S^1}$ ($\theta$ is the connection form, hence is well defined outside $\M_m^{S^1}$). The form $\nu$ can therefore also be extended to $\M_m \setminus  \M_m^{S^1}$. Applying the Stokes theorem to $0 = \int_{\M_m} \alpha =  \int_{\M_m} \tilde{d} \nu$ one gets
\[ \sum_{k=1}^N \int_{U_k} \frac{\theta \alpha}{x + d \theta} = \int_{\partial \M_m} \frac{\theta \alpha}{x + d \theta} = \int_{\partial \M_m / S^1} \pi_{*}\big( \frac{\theta \alpha}{x + d \theta} \big).\]
It was shown in \cite{berline1983}\footnote{The computation is a step of the proof of a theorem announced in \cite{berline1982} and proven in \cite{berline1983}. The most detailed computation can be found in \cite{guillemin1999supersymmetry}.} that by shrinking the radii of $U_i$ to zero the left-hand side converges to
\[ \sum_{k=1}^N \int_{(\M_m)_k} \frac{i^*_k \alpha}{e(\nu_k)},\]
where $i_k : (\M_m)_k \to \M_m$ is the inclusion map and $e(\nu_k)$ denotes the Euler class of the normal bundle to $(\M_m)_k$ in $\M_m$. Taking residues at $x=\infty$ of both sides of the above expression one obtains
\[ \int_{\partial \M_m / S^1} res_{x=\infty} \pi_{*}\big( \frac{\theta \alpha}{x + d \theta} \big) = \sum_{k=1}^N \int_{(\M_m)_k} res_{x=\infty} \frac{i^*_k \alpha}{e(\nu_k)} \]
and the left-hand side equals $\int_{\partial \M_m / S^1} \kappa_\T (\alpha)$ by equation \eqref{eq:q-inverse}. \newline

Finally, we have shown that for $\alpha \in \h_{S^1}^{\dim \partial \M_m - 1}(\M_m)$ one has
\[\int\limits_{\partial \M_m / S^1} \kappa_{\T} (\alpha)  = \sum_{k=1}^N \int_{(\M_m)_k} res_{x=\infty} \frac{i^*_k \alpha}{e(\nu_k)},\]
so for a class $\alpha \in \h^{k}_{S^{1}}(\E\T \times^{\T} \M)$ one chooses $m$ such that $k=\dim \M_m -1$ and considers the embedding of $\M_m$ into  $\M_{m^{'}}$ as described in Remark \ref{rem:degree}. It follows that for  $\alpha \in \h^k_{\T \times S^1}(\M)$ the equivariant push-forward satisfies 
\[\int\limits_{\partial \M/ S^1} \kappa_{\T} (\alpha)  = \sum_{k=1}^N \int_{\F_i} res_{x=\infty} \frac{i^*_k \alpha}{e(\nu_k)}.\]
In the last equality, $F_i$ denote the connected components of the fixed point set of the action of $S^1$ on $\M$. All the integrals in the formula above denote the push-forwards in $\T$-equivariant cohomology. 



%
\end{proof}

\begin{cor}[Symplectic case]\label{cor:symplectic} 
Consider a symplectic manifold $\M$ with a Hamiltonian action of $S^1 \times \T$ with moment map for the $S^1$-action $\mu: \M \to \lie{s}^*$. If $0$ is a regular value of $\mu$, then the manifold with boundary $\M_+ = \{ x \in \M : \mu(x) \geq 0 \}$ is a compact $S^1$-manifold with a locally free $S^1$-action on the boundary $\partial \M_+ = \mu^{-1}(0)$. Assume $\partial \M_+$ is $\T$-invariant. Let $\M \git S^1$ denote the quotient $\mu^{-1}(0)/ S^1$ and let $\M_k$ denote the connected components of the fixed point set of the action on $\M \git S^1$. By the considerations above applied to $\M_+$ one gets:

\[\int_{\M \git S^1} \kappa_{\T}(\alpha) = \sum_{\{ k: \mu_{|\M_k} \geq 0 \}} \int_{\M_k} res_{x=\infty} \frac{i^*_k \alpha}{e(\nu_k)}. \] 
\end{cor}

\begin{cor}[Jeffrey--Kirwan formula]\label{cor:jk}
In \cite{jeffrey1995} the authors consider the case when $\M$ is a symplectic manifold with an action of a compact Lie group $K$ with the moment map $\mu: M \to \lie{k}^*$. The differential form which is being integrated is
\[\kappa(\alpha) = \eta_0 e^{i \omega_0},\]
where $\eta_0 = \kappa(\eta)$  is the image of the Kirwan map for some form $\eta \in \h^*_{K}(\M)$ and $\omega_0$ is the symplectic form on the symplectic reduction $\M \git K$. In the special case when $K = \T = S^1$, one obtains the following formula for the push-forward:
\begin{align*}
\int_{\M \git S^1} \eta_0 e^{i \omega_0} &= \sum_{F \subseteq F_+} \int_{F} res_{x=\infty} \frac{i^*_F (\eta e^{i (\omega + \mu)})}{e(\nu_F)} \\
&= res_{x=\infty} \sum_{F \subseteq F_+} \int_{F} e^{i \mu(F) }\frac{i^*_F (\eta e^{i \omega})}{e(\nu_F)},
\end{align*}
where $F \subseteq F_+$ are the components of the subset $F_+$ of the fixed-point set on which $\mu$ is positive. \newline
 
After a slight change of notation, this is almost the formula from Theorem 8.1 in \cite{jeffrey1995} for torus actions, corrected by eliminating the $\frac{1}{2}$-factor and the $\psi^2$-factor under the residue. 
\end{cor}

Note that our notation for the residues is different than the one of Guillemin and Kalkman, who denote the residues by $res_{x=0}$, not by $res_{x=\infty}$. However, this is a matter of notation only. The residue in \cite{guillemin1996}, like here, is defined to be the coefficient at $x^{-1}$ in the series expansion at infinity, which we choose to call the residue at infinity to remain consistent with the classical notation from calculus. 

		\subsection{The inductive procedure equivarantly}\label{subsc:equiv-induction}

The proof of the equivariant Guillemin--Kalkman theorem is based on induction. For one-dimensional torus actions the argument presented in the preceding section is more general, i.e. it does not require symplecticity assumptions, one only needs a compact orientable manifold with boundary $(\M, \partial \M)$ together with an $S^1$-action which is locally free on the boundary. Instead of the symplectic reduction one can then take $\partial \M / S^1$. However, to proceed with the induction one needs to choose subsequent one-dimensional tori in $\S$ in such a way that in every step the assumptions are satisfied, i.e. one needs to choose a splitting $\S=\S^1 \times H$, where $\S^1$ is a one-dimensional torus, acting locally freely on $\partial \M / H$. If we knew that $\partial \M / H$ is the boundary of a compact manifold $\partial \M / H = \partial M$, then applying the theorem in the one-dimensional case of $\S=S^1$, we could express the integral 
\[\int_{\partial \M / S \times H} \kappa(\alpha) = \int_{\partial M / S} \kappa(\alpha) = \sum_k \int_{M_k} res_{x=\infty} \frac{i^*_k \alpha}{e(\nu_k)},\]
where the $M_k$ are the connected components of the fixed point set of $S^1$. In general one cannot claim that $\partial \M / H$ is the boundary of some compact manifold. This is where we use the symplectic structure and the moment map of the action. \newline

The assumptions that $\M$ is symplectic and the action is Hamiltonian enable one to use the moment map for the action for choosing a sequence of subtori as described in Sect.~\ref{subsc:gk}. Recall that for an element $\theta$ in the weight lattice of $\lie{\s}^*$, one considers a ray through the origin in the direction of $\theta$
\[l = \{ t \theta: t \in [0, \infty) \},  \]
choosing $\theta$ in such a way that the ray $l$ does not intersect any of the walls of $\Delta_i^0$ of codimension greater than one and hence intersects the codimension one walls transversely. Then the Lie subalgebra $\lie{h} \subseteq \lie{\s}$ defined as
\[\lie{h} = \{ v \in \lie{\s}: \langle \theta, v \rangle = 0 \}\]
is the Lie algebra of a codimension one subtorus $H \subseteq \S$. The assumptions made on the ray $l$ and hence on the element $\theta$ imply that the moment map $\mu$ is transverse to $l$ and the action of $H$ on $\mu^{-1}(l)$ is locally free. Moreover, the action of $\S/H \simeq S^1$ on $\mu^{-1}(l)/H$ is locally free on the boundary, which makes it possible to proceed with induction. Note that $\mu^{-1}(l)/H$ might not be a symplectic manifold, but rather a symplectic orbifold. However, the proof can easily be adapted to the case of orbifolds, because most of the analysis is done locally, and the only global component of the proof is Stokes Theorem, which holds for orbifolds (\cite{satake1957}).

	\section{Reducing the compactness assumptions}\label{sc:convexity}

The Atiyah--Guillemin--Sternberg convexity theorem \ref{thm:convexity} plays a crucial role in the proof of Theorem \ref{thm:equiv-gk-main}. This is the fundamental reason why one assumes the compactness of the manifold $\M$, and not just the symplectic reduction $\M \git \S$. In the proof of Theorem \ref{thm:equiv-gk-main} one needs the image of the moment map for the torus action to be convex. A result due to Prato (\cite{prato1994}) assures that the convexity still holds if the assumption on $\M$ being compact is replaced by a weaker assumption on the moment map itself.  \newline

Assume $\M$ is a symplectic manifold equipped with a Hamiltonian action of a torus $\S$. Denote the moment map for the action by $\mu$ and for $s \in \lie{\s}$ denote by $\mu_{s}$ the function defined as
\[ \mu_s(x):= \langle \mu_{x}, s \rangle \textrm{ for } x \in \M.\]

In \cite{prato1994} Prato proves the following theorem.
\begin{thm}\label{thm:prato}
Assume there exists an integral element $s_0 \in \lie{s}$ such that $\mu_{s_0}$ is a proper map which has a minimum at its unique critical value. Then the image of the moment map $\mu(\M)$ is a convex hull of a finite number of affine rays in $\lie{s}^*$, each of the rays stemming from an image of the $\S$-fixed point. 
\end{thm} 

Another condition assuring convexity of the image of the moment map has been given by Lerman, Meinreken, Tolman and Woodward in \cite{lerman1998}. We apply the symplectic reduction construction to vector spaces, in which case either set of assumptions is satisfied.

\chapter{Equivariant Martin integration formula}\label{ch:martin}

This chapter generalizes the theorem by Martin (\cite{martin2000}), describing the relation between the cohomology rings of symplectic reductions of a smooth symplectic manifold by a compact group $\Ka$ and its maximal torus $\S < \Ka$. Martin's theorem relates the push-forwards to a point in nonequivariant cohomology of $\M \git \Ka$ and $\M \git \S$ by the formula
\[\int_{\M \git \Ka} \alpha = \frac{1}{|\W|} \int_{\M \git \S} \tilde{\alpha} \cdot e,\]
where $\tilde{\alpha}$ is the lift of $\alpha \in \h^*(\M \git \Ka)$ to $\h^*(\M \git \S)$, in the sense explained below, and $e$ is the Euler class of a certain vector bundle. We prove an analogous formula for push-forwards in $\T$-equivariant cohomology. Throughout the proof we use the minimal set of assumptions on the actions of $\Ka$ and $\T$, as stated below. In applications in Chapter \ref{ch:formulas}, the moment map $\mu_{\Ka}$ is in fact $\T$-invariant, in which case the proof can be simplified.  \newline

Let $\M$ be a smooth symplectic manifold equipped with two commuting actions: a Hamiltonian action of a compact Lie group $\Ka$ and an action of a torus $\T$. Let $\S$ be a maximal torus in $\Ka$, acting by restriction of the action of $\Ka$,\footnote{In particular the moment map $\mu_\S$ for the $\S$-action is the composition of the moment map $\mu_\Ka: \M \to \lie{k}^*$ with a projection $\lie{\ka}^* \to \lie{\s}^*$.} and let $0$ be a regular value of the moment map $\mu_\Ka$. Assume that the symplectic reductions $\M \git \Ka$, $\M \git \S$ are compact. Assume that the sets $\mu_{\Ka}^{-1}(0)$, $\mu_{\S}^{-1}(0)$ are $\T$-invariant. Consider the following two maps:

\[
\begin{tikzcd}
\mu_\Ka^{-1}(0)/\S \arrow{r}{i} \arrow{d}{\pi} & \M \git \S \ni \tilde{\alpha} \\
 \alpha \in \M \git \Ka \ \ \ \  & 
\end{tikzcd}.
\]
The map
\[i: \mu_K^{-1}(0)/\S  \hookrightarrow \M \git \S = \mu_\S^{-1}(0)/\S \]
is an inclusion induced by the inclusion $\mu_\Ka^{-1}(0) \hookrightarrow \mu_\S^{-1}(0)$ and

\[\pi: \mu_\Ka^{-1}(0)/\S \to \M \git \Ka \] is a fibration with fiber $\Ka/\S$.
A cohomology class $\tilde{\alpha} \in \h^*_{\T}(\M \git \S)$ is called a \emf{lift} of $\alpha \in \h^*_{\T}(\M \git \Ka)$, if $\pi^* \alpha = i^* \tilde{\alpha}$.

\begin{thm}[Equivariant Martin Integration Formula]\label{thm:equiv-martin}
For a weight $\gamma$ of the $\S$-action let $\C_{\gamma}$ denote the vector space $\C$ with the action of $\S$ given by $\gamma$. Let $e^{\T} = \prod_{\gamma \in \Phi} e^{\T}(\gamma) \in \h^*_\T( \M \git \S)$ be the product of $\T$-equivariant Euler classes associated to the roots $\Phi$ of $\Ka$, denoted $e^{ \T}(\gamma) := e(L^{ \T}_{\gamma})$, where $L^{\T}_{\gamma} = \E\T \times^\T \mu_{\S}^{-1}(0) \times^\S \C_{\gamma} \to \M \git \S$. Then
\[\int_{\M \git \Ka} \alpha = \frac{1}{|\W|} \int_{\M \git \S} \tilde{\alpha} \cdot e^{\T},\]
where the integrals denote push-forwards to a point in $\T$-equivariant cohomology i.e.the equivariant Gysin map for a projection to a point.
\end{thm}

We begin with proving two lemmas.

\begin{lemma}\label{fact:normal}
The $\T$-equivariant normal bundle to $\mju{\Ka}{0} / \S$ in $\mu_{\S}^{-1}(0) / \S = \M \git \S$ can be decomposed as
\[\nu^{\T}(\mju{\Ka}{0} / \S  ) \simeq \bigoplus_{\gamma \in \Phi^{-}} L^{\T}_{\gamma | \mju{\Ka}{0} / \S}\]
\end{lemma}

\begin{proof}[Proof of Lemma \ref{fact:normal}]

By Proposition \ref{prop:gysin2} (2) it is sufficient to find a transversal section of the bundle $\bigoplus_{\gamma \in \Phi^{-}} L^{\T}_{\gamma} \to \M \git \S$ whose zero locus is $\mju{\Ka}{0} / \S$. \newline

The moment map for the $\S$-action equals the composition of the map for the $\Ka$-action and the natural projection 

\[\mu_{\S}: \M \xrightarrow{\ \mu_\Ka \ } \lie{\ka}^* \twoheadrightarrow \lie{\s}^*. \]
We claim $\mju{\Ka}{0}$ is a zero locus of a transversal section of the bundle 
\[ \bigoplus_{\gamma \in \Phi^{-}} L^\T_{\gamma} \to \M \git \S .\]
To define the section, let us first consider the preimage of $0$ under the natural projection $\lie{\ka}^* \twoheadrightarrow \lie{\s}^*$, $V \subset \lie{\ka}^*$, so that $\mju{\S}{0}= \mju{\Ka}{V}$. In terms of the roots of $\Ka$ the set $V$ is just the sum of weight spaces corresponding to the negative roots, $V = \bigoplus_{\gamma \in \Phi^{-}} \C_{\gamma}$. The moment map $\mu_{\Ka}$ restricts to a map $\tilde{s} = \mu_{\Ka|\mju{\S}{0}}$,
\[ \tilde{s}: \mju{\S}{0} \to V. \]
Because this map is obviously $\S$-equivariant as a restriction of a $\Ka$-equivariant map, it descends to a map 
\[s: \M \git \S = \mju{\S}{0}/\S \to \mju{\S}{0} \times^\S V. \]
From this we produce a map:
\[s^{\T}:  \E\T \times^{\T} \M \git \S \to \E\T \times^{\T} \mju{\S}{0} \times^\S V = \bigoplus_{\gamma \in \Phi} L^{\T}_{\gamma},\]
which does not change the zero locus, as follows. \newline

Recall that $\mu_\S^{-1}(0)$ is $\T$-invariant, by assumption. This implies that the section $s:  \M \git \S \to \mju{\S}{0} \times^\S V$ descends to the quotients\footnote{The action of $\T$ on $V$ is defined via $t \cdot \mu_{\Ka}(m) := \mu_{\Ka}(tm)$. In examples considered later this action is in fact trivial.} by the $\T$-action $s:  (\M \git \S) / T \to (\mju{\S}{0} \times^\S V)/T$ and thus defines a section 
\[s^{\T}:  \E\T \times^{\T} \M \git \S \xrightarrow{\ id \times s\ } \E\T \times^{\T} \mju{\S}{0} \times^\S V, \]
as defined in the following diagram:

\[
\begin{tikzcd}
\E\T \times^{\T}  \mju{\S}{0} \times^\S V \arrow{r} \arrow[d, swap,"id \times p"] &(\mju{\S}{0} \times^\S V)/T \arrow[d, swap, "p"] \\
 \E\T \times^{\T} \M \git \S \arrow{r} \arrow[u, bend right,swap,"id \times s"] & (\M \git \S)/\T \arrow[u, bend right, swap,"s"]
\end{tikzcd}.
\]
The zero locus of the section $s^{\T}$ is obviously $\E\T \times^{\T} Z$, where $Z$ denotes the zero locus of the section $s$. By construction, the zero locus of $s$ is $\mju{\Ka}{0} \git \S$. The equivariant normal bundle to $\mju{\Ka}{0} \git \S$ in $\M \git \S$ is by definition the normal bundle to  $\E\T \times^{\T} \mju{\Ka}{0} \git \S$ in  $\E\T \times^{\T} \M \git \S$, which by Proposition \ref{prop:gysin2} (2) equals the restriction of $\bigoplus_{\gamma \in \Phi} L^{\T}_{\gamma}$ to $\mju{\Ka}{0} \git \S$.
\end{proof}

\begin{lemma}\label{lem:beta}
Let $\beta^\T = \prod_{\gamma \in \Phi^{+}} e^\T(L^{\T}_{\gamma}) \in \h_\T^*(\M \git \S)$, where $e^\T(-)$ denotes the $\T$-equivariant Euler class. Then 
\[\pi_* i^* \beta^\T = \int_{\Ka/\S} e^\T(T(\Ka/\S)) = |\W|.\] \label{fact1eq}
\end{lemma}

\begin{proof}[Proof of Lemma \ref{lem:beta}]
The cohomology class $\beta^\T = \prod_{\gamma \in \Phi^{+}} e^\T(L^{\T}_{\gamma})$ lies in gradation $2 | \Phi^+|$, hence also $i^* \beta^\T \in \h_\T^{2 |\Phi^+|}(\mu_\Ka^{-1}(0)/\S)$. The map $\pi$ is a fibration with fiber $\Ka / \S$, so the push-forward $\pi_*$ lowers the gradation by the codimension of $\M \git \Ka$ in $\mu_\Ka^{-1}(0)/\S$, which equals the real dimension of $\Ka / \S$, $2 |\Phi^+|$. Hence $\pi_* i^* \beta^\T \in \h^0_\T(\M \git \Ka)$. We can therefore restrict further the result to the cohomology of a point, without altering the result.\footnote{A map $i_{pt}: pt \to \M \git \Ka$ induces identity on the cohomology in gradation $0$.} \newline

By the base change property of Gysin maps (Proposition \ref{prop:gysin1}, (3)) applied to the diagram

\[
\begin{tikzcd}
\Ka / \S   \arrow[r,"i_{\Ka/\S}"]  \arrow[d] & \mu_{\Ka}^{-1}(0)/\S \arrow[d, "\pi"] \\
 pt \arrow[r,"i_{pt}"] & \M \git \Ka
\end{tikzcd},
\]
one has $i^*_{pt} \pi_* = \int_{\Ka/\S} i_{\Ka/\S}^*$, hence

\[\pi_* i^* \beta^\T = i^*_{pt} \pi_* i^* \beta^\T = \int_{\Ka/\S} i_{\Ka/\S}^* i^* \beta^\T =\]
\[= \int_{\Ka/\S} \prod_{\gamma \in \Phi^{+}} e^\T(L^{\T}_{\gamma|\Ka/\S})  =\int_{\Ka/\S} e^\T(T(\Ka/\S)) = |\W|,\]
where the last equality follows from the identification of $T(\Ka/\S)$ with $\lie{\ka}/\lie{\s}$ and its Euler class with the product of classes of line bundles corresponding to the positive roots. 

\end{proof}

\begin{proof}[Proof of Theorem \ref{thm:equiv-martin}]
Throughout this proof, $\int_{Y}\!-$ denotes the push-forward to a point in $\T$-equivariant cohomology and all push-forwards are considered in $\T$-equivariant cohomology. Let $\alpha \in \h^*_{\T}(\M \git \Ka)$ and let $\tilde{\alpha} \in \h^*_{\T}(\M \git \S)$ be the lift of $\alpha$. Let us denote the $\T$-equivariant analogs of the maps $i$, $\pi$ as in the following diagram

\[
\begin{tikzcd}
\E\T \times^\T \mu_\Ka^{-1}(0)/\S \arrow{r}{i_\T} \arrow{d}{\pi^\T} & \E\T \times^\T \M \git \S \\
 \E\T \times^\T \M \git \Ka & 
\end{tikzcd},
\]
Finally, let $\beta^\T = \prod_{\gamma \in \Phi^{+}} e^\T(L^{\T}_{\gamma})$, as in the Lemma \ref{fact1eq} above.
Applying the equivariant push-forward composition formula of Proposition \ref{prop:gysin1} (3) to \linebreak $\pi^\T: \E\T \times^\T \mu_\Ka^{-1}(0)/\S \to \E\T \times^\T \M \git \Ka$ and $x = (\pi^\T)^* \alpha \cdot i_\T^* \beta^\T$, one gets
\[\int_{\mu_\Ka^{-1}(0)/\S} (\pi^\T)^* \alpha \cdot i_\T^* \beta^\T = \int_{\M \git \Ka} \pi^\T_* ((\pi^\T)^* \alpha \cdot i_\T^* \beta^\T).\]
By the projection formula \ref{prop:gysin1} (2), $\pi^\T_* ((\pi^\T)^* \alpha \cdot i_\T^* \beta^\T) = \alpha \cdot \pi^\T_* i_\T^* \beta^\T$, and Lemma \ref{lem:beta} implies $\pi^\T_* i_\T^* \beta^\T = |\W|$, so one gets:

\[ \int_{\mu_\Ka^{-1}(0)/\S} (\pi^\T)^* \alpha \cdot i_\T^* \beta^\T = \int_{\M \git \Ka} \alpha |\W_\Ka|.\]
Finally, using $(\pi^\T)^* \alpha = i_\T^* \tilde{\alpha}$, one can write

\begin{equation}\label{eq:integral1}
\begin{split}
\int_{\mu_\Ka^{-1}(0)/\S} (\pi^\T)^* \alpha \cdot i_\T^* \beta^\T &= \int_{\mu_\Ka^{-1}(0)/\S} i_\T^* \tilde{\alpha} \cdot i_\T^* \beta^\T \\
&= \int_{\M \git \S} (i_\T)_* i_\T^*(\tilde{\alpha} \cdot \beta^\T),
\end{split}
\end{equation}
again by the composition of push-forwards rule. To compute $(i_\T)_* i_\T^*$ we need to use the approximation spaces:

\[
\begin{tikzcd}
\cdots \arrow[r] &\E\T_m \times^\T \mu_\Ka^{-1}(0)/\S \arrow[d,"i_\T^m"] \arrow[r, "j_m"] & \E\T_{m+1} \times^\T \mu_\Ka^{-1}(0)/\S \arrow[d,"i_{\T}^{m+1}"] \arrow[r,"j_{m+1}"]& \cdots \\
\cdots \arrow[r] & \E\T_m \times^\T \M \git \S \arrow[r, "j_m"]&  \E\T_{m+1} \times^\T \M \git \S \arrow[r,"j_{m+1}"]& \cdots
\end{tikzcd}
\]
For each $m$, $(i_\T^m)_* (i^m_\T)^*$ is given by multiplication by the Euler class of the normal bundle to $\E\T_m \times^\T \mu_\Ka^{-1}(0)/\S$ in  $\E\T_m \times^\T \M \git \S$ as a consequence of Proposition \ref{prop:gysin2} (1) and (2). Note that Proposition  \ref{prop:gysin2} (1) is valid only for inclusions of closed submanifolds, which is why we needed to consider the "approximated" inclusion $i_\T^m$ instead of working directly with $i_\T$. The inclusions $i_{\T}^m$ are compatible, in particular the above diagram is commutative and the Euler classes of the normal bundles are mapped onto one another under the maps $j^*_{m}$. The Euler class of the normal bundle to $\E\T \times^\T \mu_\Ka^{-1}(0)/\S$ in $\E\T_m \times^\T \M \git \S$ is the limit of the Euler classes of the "approximating" normal bundles (by definition). Therefore $(i_\T)_* i_\T^*$ is given by multiplication by the Euler class of the normal bundle, which is described in Lemma \ref{fact:normal}, hence by equation \eqref{eq:integral1} one gets
\[ \int_{\M \git \Ka} \alpha |\W| = \int_{\M \git \S} (i_\T)_* i_\T^*(\tilde{\alpha} \cdot \beta^\T) = \int_{\M \git \S} \tilde{\alpha} \cdot \beta^\T \cdot \prod_{\gamma \in \Phi^{-}} e^\T(L_{\gamma}) = \int_{\M \git \S} \tilde{\alpha} \cdot e^\T,\]
where the last equality is just the definition of $e^\T$.

\end{proof}

\chapter{Push-forward formulas}\label{ch:formulas}

	\section{Classical Grassmannian}\label{sc:grassmannian}

Consider the action of the group of unitary matrices $U(k)$ on the space $\HomC{k}{n}$ of linear maps $\C^k \to \C^n$, given by matrix multiplication on the right. The space $\HomC{k}{n}$ has a symplectic structure given by the symplectic form $\omega(A,B) = 2 \Imm(\tr A B^*)$. The moment map for this action \linebreak $\mu: \HomC{k}{n} \to \mathfrak{u}(k)^* \simeq \mathfrak{u}(k)$ is given by\footnote{This moment map is computed in \cite{kirwan1984cohomology}.}
\[ \mu(A) = A^* A- \Id. \]
Hence $\mu^{-1}(0) = \{ A^* A = \Id\}$ and the column vectors of a matrix $A \in \mu^{-1}(0)$ form a unitary $k$-tuple in $\C^n$. This implies that the symplectic reduction with respect to the $\U(k)$-action is the Grassmannian of $k$-planes in $\C^n$,
\[\HomC{k}{n} \git \U(k) = \grr_{k}(\C^n).\]
To make the notation more compact---especially while drawing the diagrams---we set $\Hom{k}{n}:=\HomC{k}{n}$ and $\Grass{k}{n}:=\grr_k(\C^n)$.

Let $\S$ denote the maximal torus in $\U(k)$, acting on $\Grass{k}{n}$ via restriction of the action of $\U(k)$. Let $z_1,\dots,z_k$ denote the characters of the action of $\S$. Consider additionally the action of the maximal torus $\T$ in $\U(n)$ on $\Grass{k}{n}$ with characters $t_1,\dots,t_n$ and consider the push-forward to a point in the $\T$-equivariant cohomology of $\Grass{k}{n}$. The following lemma is a consequence of the Martin integration formula \ref{thm:mit} and reduces the integral over $\Grass{k}{n}$ to the integral over the symplectic reduction $\Hom{k}{n} \git \S$.

\begin{lemma}\label{lem:grass-kappa}
Let $\kappa^\S_\T$ be the $\T$-equivariant Kirwan map for the action of $\S$ on $\Hom{k}{n}$ and let $\W$ denote the Weyl group of $\U(k)$. Let $e = \prod_{\gamma \in \Phi}\gamma = \prod_{i \neq j} (z_i - z_j)$ be the product of the roots of $\U(k)$. Then the push-forwards to a point in $\T$-equivariant cohomology of $\Grass{k}{n}$ and $\Hom{k}{n} \git \S$ are related by the following formula.
\[ \int_{\Grass{k}{n}} \kappa_\T(\alpha) = \int_{\Hom{k}{n} \git \S} \kappa^\S_\T(\alpha) \cdot e . \]
\label{lemMIT}
\end{lemma}

\begin{proof}[Proof of the Lemma \ref{lemMIT}]
Denote by $i^*$ the map induced on cohomology by the inclusion 
\[i: \mu_{\U(k)}^{-1}(0)/ \S \hookrightarrow \mu_{\S}^{-1}(0)/ \S = \Hom{k}{n} \git S\]
and by $\pi^*$ the map induced by the map 
\[\pi: \mu_{\U(k)}^{-1}(0)/ \S \hookrightarrow \mu_{\U(k)}^{-1}(0)/ \U(k),\] 
which is a fibration with fiber $\U(k)/\S$. By the Martin integration formula \ref{thm:mit} we have
\[ \int_{\Grass{k}{n}} \kappa_\T(\alpha) = \int_{\Hom{k}{n} \git \S} \widetilde{\kappa_\T(\alpha)} \cdot e, \]
where $\widetilde{\kappa_\T(\alpha)}$ is the lift of $\kappa_\T(\alpha)$ to $\h^*_{\T}(\Hom{k}{n} \git \S)$ as defined in Ch.~\ref{ch:prelim}, Sect.~\ref{subsc:martin}. It remains to be checked that the lift of $\kappa_\T(\alpha)$ is $\kappa^\S_\T(\alpha)$ i.e. that $i^*\widetilde{\kappa_\T(\alpha)} = \pi^*\kappa_\T(\alpha)$. This requires the following diagram to commute:

\begin{equation}\label{diagramik}
\begin{tikzcd}
\h^*_{\T \times \U(k)}(\Hom{k}{n}) \arrow[r, "\kappa_{\T}"] \arrow[dd, hook ] & \h^*_\T(\Grass{k}{n}) \arrow[rd, "\pi^*"] \\
{} & {} & \h^*_\T(\mu_{\U(k)}^{-1}(0) / \S) \\
  \h^*_{\T \times \S}(\Hom{k}{n})\arrow[r, "\kappa^\S_{\T}"] & \h^*_\T(\Hom{k}{n} \git S) \arrow[ru, "i^*"]
\end{tikzcd}.
\end{equation}

Since $\Hom{k}{n}$ is contractible, one has
\[\h^*_{\T \times \S}(\Hom{k}{n}) = \h^*_{\T \times \S}(pt) = \C[t_1\dots,t_n, z_1,\dots,z_k]. \]
By Theorem \ref{thm:chevalley}, for a Lie group $\Ka$ with a maximal torus $\S$ and Weyl group $\W$, one has $\h^*_{\Ka}(pt) = \h^*_{\S}(pt)^\W$. Hence
\[\h^*_{\T \times \U(k)}(\Hom{k}{n}) = \h^*_{\T \times \U(k)}(pt) = \C[t_1\dots,t_n, z_1,\dots,z_k]^{\Sigma_k},\]
where the permutation group $\Sigma_k$ acts by permuting variables $z_1,\dots,z_k$. From the construction and the naturality of the Kirwan map it follows that the Kirwan map $\kappa^{\S}_{\T}$ maps $\W$-invariants to $\W$-invariants. Moreover, the map $\pi$ induces an isomorphism 
\[\h^*_\T(\Grass{k}{n}) \xrightarrow{\pi^*} \h^*_\T(\mu_{\U(k)}^{-1}(0) / \S)^\W, \]
by the following theorem of Borel \cite{borel1953}, applied to $\S < \Ka = \U(k)$ and $\pi$ as above. 

\begin{thm}\label{thm:borel-fibration}
Let $X \xrightarrow{\pi} Y$ be a $\Ka$-principal bundle, let $\S < \Ka$ be a maximal torus and let $\W$ be the Weyl group of $\Ka$. Then the projection $\pi$ induces an isomorphism
\[ \h^*(Y; \mathbb{Q}) \simeq \h^*(X/\S; \mathbb{Q})^\W .\]
\end{thm} 

Let $r : \h^*_{\T \times \U(k)}(-) \to \h^*_{\T \times \S}(-)$ denote the restriction map in $\T$-equivariant cohomology induced by the inclusion $\S \hookrightarrow \U(k)$. From the Theorem \ref{thm:borel-fibration} (and the naturality of the Kirwan maps) it follows that the following diagram is commutative:

\[
\begin{tikzcd}
\kappa_\T: \h^*_{\T \times \U(k)}(pt) \arrow[r, two heads] \arrow[d, "r", "\simeq"' ] & \h^*_{\T \times \U(k)}(\mu^{-1}_{\U(k)}(0)) \arrow[r] \arrow[d,"r", "\simeq"']& \h^*_\T(\Grass{k}{n}) \arrow[d, "\pi^*", "\simeq"'] \\
\kappa_\T \circ r: \h^*_{\T \times \S}(pt)^\W \arrow[d, equal]  \arrow[r, two heads] & \h^*_{\T \times \S}(\mu^{-1}_{\U(k)}(0))^\W \arrow[r]& \h^*_\T(\mu_{\U(k)}^{-1}(0) / \S)^\W \\
 \kappa^\S_{\T}: \h^*_{\T \times \S}(pt)^\W  \arrow[r, two heads] & \h^*_{\T \times \S}(\mu^{-1}_{\S}(0))^\W \arrow[r] \arrow[u,"i^*"'] &\h^*_\T(\Hom{k}{n} \git S)^\W \arrow[u, "i^*"']
\end{tikzcd}
\]
The maps denoted with double-head arrows are epimorphisms by the Jeffrey--Kirwan theorem. The up-going arrows are restriction maps in cohomology induced by the inclusion $i:\mu_{\U(k)}^{-1}(0) \hookrightarrow \mu_{\S}^{-1}(0)$. In particular, for any class $\alpha \in \h^*_{\T \times \U(k)}(pt) = \C[t_1\dots,t_n, z_1,\dots,z_k]^{\Sigma_k}$ we have $\pi^* \kappa_{\T}(\alpha) = i^* \kappa_{\T}^{\S}(\alpha)$, which proves the commutativity of the diagram \eqref{diagramik} and completes the proof of Lemma \ref{lemMIT}. 

\end{proof}

Let us now describe the symplectic reduction $\Hom{k}{n} \git \S$. The maximal torus $\S$ in $\U(k)$ acts diagonally on $\Hom{k}{n} = \C^n \oplus \dots \oplus \C^n$, with the action on each component $\C^n$ given by multiplication. The moment map $\mu_\S : \Hom{k}{n} \to \mathfrak{s}^*$ for this action is given by projection of the moment map for the action of $\U(k)$, hence
\[\mu_\S(A) = (||v_1||-1,\dots, ||v_k||-1),\]
where $v_1, \dots, v_k$ are the column vectors of $A$. The symplectic reduction for the $\S$-action is therefore
\[\Hom{k}{n} \git \S = \mu_\S^{-1}(0)/\S = (\C\P^{n-1})^k,\]
since $\mu_\S^{-1}(0)$ consists of $k$-tuples of vectors of length $1$ in $\C^n$. The image of the moment map $\mu_\S$ is a product of half-lines $(\RR_{\geq 0})^k$. It is not a convex polytope, because the space $\Hom{k}{n}$ is noncompact. However, we can still use the dendrite algorithm as in \cite{guillemin1996}, only this time we have to make sure that in each step we choose rays $l$ in such a way, that they intersect a codimension one wall (and not diverge to infinity). In our case this can be easily achieved - if the chosen ray $l$ does not intersect a codimension one face of $(\RR_{\geq 0})^k$, then the ray $-l$ does. Choosing the rays in this way always leads to a branch ending at the only fixed point of the action, the origin. Therefore, we get the following expression for the push-forward:

\begin{equation}
 \int_{\Grass{k}{n}} \kappa_\T(\alpha) = \frac{1}{|\W|}res_{z_1,\dots z_k=\infty} \frac{e \cdot i_0^*(\alpha)}{e^{\T \times \S}(0)}, \label{eq:grass}
\end{equation}
where $e = \prod_{\gamma \in \Phi}\gamma = \prod_{i \neq j} (z_i - z_j)$ is the product of the roots of $\U(k)$, $i_0^*(\alpha)$ denotes the restriction of $\alpha$ to $0 \in \Hom{k}{n}$ and $e^{\T \times \S}(0)$ is the $\T \times \S$-equivariant Euler class at zero, which equals $e^{\T \times \S}(0) = \prod_{i,j}(z_i - t_j)$, where $z_1,\dots,z_k$ are the characters of the $\S$-action and $t_1,\dots,t_k$ are the characters of the $\T$-action.

\begin{remark}\label{rem:restriction}
Since $\Hom{k}{n}$ is contractible, by abuse on notation for a class $\alpha \in \h^*_{\T \times \U(k)}(\Hom{k}{n})$ we will denote the its restriction to $0 \in \Hom{k}{n}$ also by $\alpha$. The notation $\alpha(z_1,\dots,z_k)$ means that we view $\alpha$ as a polynomial in $z_1,\dots,z_k$ with coefficients in $\C[t_1, \dots, t_k]$.
\end{remark}

Finally, we obtain the following formula for the push-forward.

\begin{thm}\label{thm:formula-grass}
Let $\alpha \in \h_{\T \times \U(k)}^*(\Hom{k}{n})$. Then the push-forward of its image under the Kirwan map $\kappa_\T$ equals the following residue taken with respect to the characters of the maximal torus in $\U(k)$:  
\[ \int_{\Grass{k}{n}} \kappa_\T(\alpha) = \frac{1}{|\W|} res_{z_1,\dots,z_k=\infty} \fr{ \prod_{i \neq j} (z_i - z_j) \alpha(z_1,\dots,
z_k)}{\prod_{i,j}(z_i - t_j)}
\]
\end{thm}
Note that the fact that we push forward only forms which lie in the image of the the equivariant Kirwan map $\kappa_{\T}$ is not a very restrictive condition, it means that we push-forward forms which at the fixed points of the action are given by $\W$-symmetric polynomials. By the surjectivity of the Kirwan map, every form can be presented in such a way. As mentioned in the introduction, this formula can also be obtained combinatorically, using the Atiyah--Bott--Berline--Vergne formula and introducing the residue at infinity ad hoc, as described in \cite{zielenkiewicz2014}. \newline

The description of the Grassmannian as the symplectic reduction for the natural $\U(k)$-action has been used by Martin to give a description of the nonequivariant cohomology of the classical Grassmannian. For details, see \cite{martin2000}, Chapter 7.

	\section{Lagrangian Grassmannian}\label{sc:lagrangian}

Let $\omega$ be the standard complex symplectic form on $\C^{2n}$ and consider the Lagrangian Grassmannian $LG(n)$ parametrizing maximal isotropic subspaces of $\C^{2n}$ i.e. subspaces $V$ of dimension $n$ satisfying
\[V = V^{\perp} = \{ w: \omega(w,v) = 0 \textrm{ for all } v \in V \}.\]
The Lagrangian Grassmannian canonically embeds in $\Grass{n}{2n}$. 
The maximal torus in $Sp(n)$ consists of elements of the form 
\[ \diag(t_1,t_2,\dots,t_n,t_n^{-1},\dots,t_{2}^{-1},t_1^{-1}),\]
 so one has a natural inclusion of the maximal tori $\T_{Sp(n)} \hookrightarrow \T_{\U(2n)}$:
\[\diag(t_1,t_2,\dots,t_n,t_n^{-1},\dots,t_{2}^{-1},t_1^{-1}) \mapsto \diag(t_1,t_2,\dots, t_{2n}).\]
We will derive the residue-type formula for the push-forward in the $\T_{Sp(n)}$-cohomology of the Lagrangian Grassmannian from the formulas for the classical Grassmannian by considering the inclusion
\[ LG(n) \hookrightarrow \Grass{n}{2n}.\]

		\subsection{Push-forward formula for $\Grass{n}{2n}$}\label{subsc:gr2n}

Assume $\U(n)$ acts on the space $\Hom{n}{2n}$ by matrix multiplication on the right, as in the previously considered example of the classical Grassmannian. Recall that the moment map for this action is 
\[ \mu(A) = A^* A - \Id. \]
Let $\S$ be a maximal torus in $\U(n)$ and let $\T:=\T_{Sp(n)}$ be the maximal torus in $Sp(n)$, canonically included in the maximal torus of $\U(2n)$. As before, we consider the symplectic reduction with respect to $\S$ and push-forwards to a point in $\T$-equivariant cohomology.
The push-forward $\h^*_{\T}(\Grass{n}{2n}) \to \h^*_{\T}(pt)$ is given by the formula of Theorem \ref{thm:formula-grass},

\[ \int_{\Grass{n}{2n}} \kappa_\T(\alpha) = \frac{1}{|\W|}res_{z_1,\dots z_n=\infty} \frac{e \cdot \alpha(z_1,\dots, z_n)}{e^{\T \times \S}(0)}, \]
where $e= \prod_{\gamma \in \Phi^+}\gamma = \prod_{i \neq j} (z_i - z_j)$ is the product of the roots of $\U(n)$, and $e^{\T \times \S}(0)$ is the $\T \times \S$-equivariant Euler class at zero, which equals 
\[e^{\T \times \S}(0) = \prod_{i,j=1}^n (z_i - t_j)(z_i + t_j),\]
 because $z_1,\dots,z_n$ are the characters of the $\S$-action and $t_1,\dots,t_n, -t_1, \dots -t_n$ are the characters of the $\T$-action. One gets

\begin{equation} 
\int_{\Grass{n}{2n}} \kappa_\T(\alpha) = \frac{1}{|\W|} res_{z=\infty} \fr{ \prod_{i \neq j} (z_i - z_j) \alpha(z_1,\dots,
z_n)}{\prod_{i,j=1}^n(z_i^2 - t_j^2)}.
 \label{eq:gr2n}
\end{equation}

		\subsection{Relating push-forwards for $\Grass{n}{2n}$ and $LG(n)$}\label{subsc:relating}

Since the Lagrangian Grassmannian $LG(n)$ is a subvariety of $\Grass{n}{2n}$, the push-forwards in equivariant cohomology are related by the following relation:

\[\int_{LG(n)} \alpha = \int_{\Grass{n}{2n}} \tilde{\alpha} \cdot [LG(n)], \]
where $[LG(n)]$ denotes the fundamental class of $LG(n)$ in $\h_{\T}^* (\Grass{n}{2n})$ and $i^*\tilde{\alpha}  = \alpha$. Consider the following diagram:

\[
\begin{tikzcd}
C_L \arrow[r, hook, "j"] \arrow[d, "q_{|C_L}", ] &  \arrow[d, "q"] \Hom{n}{2n} \\
 LG(n) \arrow[r, hook] & \Grass{n}{2n}
\end{tikzcd}
\]
where $C_L$ is the preimage if $LG(n)$ under the symplectic reduction map $q$. The subset $C_L \subseteq \Hom{n}{2n}$, called the Lagrangian cone, consists on those maps $\C^n \to \C^{2n}$ whose image is an isotropic subspace of $\C^{2n}$. Note that the Lagrangian cone $C_L \subseteq  \Hom{n}{2n}$ is not smooth, so we cannot directly use the approach we used for the classical Grassmannian and construct $LG(n)$ as a symplectic reduction of a smooth symplectic manifold. One can solve this inconvenience either by considering a slightly more general singular symplectic reduction (c.f. for example \cite{lerman1993}), or by restricting the result obtained for $\Grass{n}{2n}$ and checking that the push-forwards are compatible with this restrictions. We will follow the later approach.

		\subsection{Kirwan maps and push-forwards}\label{subsc:kirwan-lg}

Consider the $\T$-equivariant Kirwan maps for the symplectic reductions $\Hom{n}{2n} \git \S$ and $C_L \git \S$:

\[\kappa^{\S}_\T: \h^*_{\T \times \S}(\Hom{n}{2n}) \to  \h^*_{\T}(\Hom{n}{2n} \git \S),\]
\[\kappa^{\S}_{\T, C_L}:\h^*_{\T \times \S}(C_L) \to \h^*_{\T}(C_L \git \S)= \h^*_{\T}(C_L \cap \mu^{-1}_{\S}(0) / \S).\]
Similarly, we have the Kirwan map for the quotients by the $\U(n)$ action:
\begin{align*}
\kappa_\T&: \h^*_{\T \times \U(n)}(\Hom{n}{2n}) \to  \h^*_{\T}(\Hom{n}{2n} \git \U(n)) =  \h^*_{\T}(\Grass{n}{2n}), \\
\kappa_{\T,C_L}&:\h^*_{\T \times \U(n)}(C_L) \to \h^*_{\T}(C_L \git \U(n))= \h^*_{\T}(C_L \cap \mu^{-1}_{\U(n)}(0) / \U(n)).
\end{align*}
The results obtained in Sect.~\ref{sc:grassmannian} can be summarized in the commutativity of the following diagram:

\[
\begin{tikzcd}[column sep=tiny]
\h^*_{\T \times U(n)}(\Hom{n}{2n}) \arrow[rr, "\kappa_{\T}"] \arrow[dd, leftrightarrow,"\simeq" ] &{} &\h^*_\T(\Grass{n}{2n}) \arrow[dd, leftrightarrow,"\simeq"' ] \arrow[rrd, sloped, near start, "\int -"] & {} & {} \\
{} & {} &{}&{}& \h^*_\T(pt) \\
  \h^*_{\T \times \S}(\Hom{n}{2n})^\W \arrow[rr, "\kappa^\S_{\T}"] & {}& \h^*_\T(\Hom{n}{2n} \git \S)^\W \arrow[rru, sloped, near end, "\int - e"] & {} & {}
\end{tikzcd}.
\]
Restricting all maps via the map induced by $j: C_L \hookrightarrow \Hom{n}{2n}$ yields

\[
\begin{tikzcd}[row sep=3em, column sep=1.5em]
& \h^*_{\T \times \U(n)}(pt) \arrow[dl,swap,"j^*"] \arrow[rr,"\kappa_{\T}"] \arrow[dd, swap, near end, "\simeq"] 
  & & \h^*_\T(\Grass{n}{2n}) \arrow[dd, swap, near end, "\simeq"] \arrow[dl,swap,sloped,near start, "j^*"] \\
\h^*_{\T \times \U(n)}(C_L)  \arrow[rr, crossing over, near end, "\kappa_{\T, C_L}"] \arrow[dd, swap, near end, "\simeq"] 
  & & \h^*_\T(LG(n)) \\
& \h^*_{\T \times \S}(pt)^\W \arrow[rr, near start, "\kappa_{\T}^{\S}"] \arrow[dl, "j^*"] 
  & & \h^*_\T(\Hom{n}{2n} \git \S)^\W \arrow[dl, "j^*"] \\
 \h^*_{\T \times \S}(C_L)^\W \arrow[rr, "\kappa_{\T, C_L}^{\S}"] & & \h^*_\T(C_L \git S)^\W \arrow[uu, crossing over, leftarrow, near start, "\simeq"]
\end{tikzcd}
\]

We intend to describe the Gysin homomorphisms for the cohomology groups on the right-hand face of the above commutative cube. These Gysin maps are related by the equivariant Martin theorem \ref{thm:equiv-martin}, applied to each triangle in the diagram below:

\[
\begin{tikzcd}[column sep=huge, row sep=6em]
\h^*_{\T}(LG(n)) \arrow[rd, sloped, "\int\limits_{LG(n)} j^*(-) "] \arrow[dd, leftrightarrow, "\simeq" ] & {} & \h^*_\T(\Grass{n}{2n}) \arrow[dd, leftrightarrow,"\simeq"' ] \arrow[ld, sloped, near end, "\int\limits_{\Grass{n}{2n}} -\cdot {[LG(n)]}" ] \arrow[ll, "j^*"] \\
{} & \h^*_\T(pt) & {} \\
  \h^*_{\T}(C_L \git \S)^\W \arrow[ru, sloped, near end, "\frac{1}{|\W|}\int\limits_{C_L \git \S} j^*(- \cdot e)"] & {}& \h^*_\T(\Hom{n}{2n} \git \S)^\W  \arrow[ll, "j^*"] \arrow[lu, sloped, near start,  "\frac{1}{|\W|} \!\!\!\! \int\limits_{\Hom{n}{2n} \git \S} - \cdot e \cdot{[C_L \git \S]}"] 
\end{tikzcd}
\]

		\subsection{Derivation of the push-forward formula for $LG(n)$}\label{subsc:formulas-lg}

Bearing in mind the relation between the push-forwards in the cohomology of $\Grass{n}{2n}$ and $LG(n)$ one can write

\[\int_{LG(n)} j^* \kappa_{\T}(\alpha) = \int_{\Grass{n}{2n}} \kappa_\T(\alpha) \cdot [LG(n)] = \int_{\Grass{n}{2n}} \kappa_\T(\alpha \cdot \widetilde{LG(n)}),\]
where $\widetilde{[LG(n)]}$ is an element of $\h^*_{\T \times \U(n)}(\Hom{n}{2n})$ such that $\kappa_{\T}(\widetilde{[LG(n)]})=[LG(n)]$, and use equation \eqref{eq:gr2n} to obtain the result in a form of an iterated residue. We have
\[\int_{\Grass{n}{2n}} \kappa_\T(\alpha \cdot \widetilde{LG(n)}) = \frac{1}{|\W|} \Res_{z=\infty} \fr{ \prod_{i \neq j} (z_i - z_j) i^*_{0}\kappa^{\S}_{\T}(\alpha \cdot \widetilde{LG(n)})}{\prod_{i,j=1}^n(z_i - t_j)(z_i + t_j)}.\]
The final step in obtaining the residue-type push-forward formula for $LG(n)$ is therefore identifying the polynomial in $\C[z_1,\dots,z_n]$ representing the class $i^*_{0}\kappa^{\S}_{\T}(\alpha \cdot \widetilde{LG(n)})$, which is the content of the Remarks \ref{rem:fund-class-lg}, \ref{rem:fund-class-lg2} and \ref{rem:fund-class-lg3}.

\begin{remark} \label{rem:fund-class-lg}\label{remark}
The Kirwan map $\kappa_\T^\S$ maps the fundamental class \linebreak $[C_L] \in \h^*_{\T \times \S}(\Hom{n}{2n})$ to the class $[C_L \git \S] \in  \h^*_{\T}(\Hom{n}{2n} \git \S)$:
\[\kappa_{\T}^\S([C_L]) = [C_L \git \S].\]
This is obvious because Kirwan map is defined as the map induced by restriction followed by the natural isomorphism.

The fundamental class $[C_L] \in \h^*_{\T \times \S}(\Hom{n}{2n})$ can be computed by describing $C_L$ by equations. Let $A \in \Hom{n}{2n}$ be represented by the matrix with column vectors $v_1,\dots,v_n$,

\[ A =
\left[
	\begin{array}{l|l|l}
		& & \\
		& & \\
		v_1 & \ldots & v_n \\
		& & \\ 
		& & \\
	\end{array}
\right]
\]
The maximal torus $\S$ in $\U(n)$ acts on $A$ by multiplication of each column vector by a corresponding character.  For $A$ to be contained in the Lagrangian cone $C_L$, the column vectors of $A$ must satisfy:
\[\omega(v_i, v_j) = 0 \textrm{ for } i < j.\]
Under the action of $\S$ the function $\omega(v_i, v_j)$ is multiplied by $z_i z_j$, hence the weight of the $\S$-action on this equation is $z_i + z_j$.\footnote{We use the same notation $z_i$ to denote both the coordinate and the character of the torus action.} The torus $T$ acts by multiplying rows of the matrix $A$ by $t_1,\dots,t_n,t_n^{-1},\dots,t_1^{-1}$, hence acts on $\omega(v_i, v_j)$ trivially. The equations are clearly independent. The are $\frac{n(n-1)}{2}$ equations, which equals the codimension of $C_L$ in $\Hom{n}{2n}$. This assures that $C_L$ is a complete intersection in $\Hom{n}{2n}$ and the fundamental class equals the product of the weights of the equations:
\[ [C_L] = \prod_{\substack{i,j = 1 \\ i<j}}^n (z_i + z_j).\]
\end{remark}

\begin{remark}\label{remark12}\label{rem:fund-class-lg2}
Since the class $[C_L] \in \h^*_{\T \times \S}(\Hom{n}{2n})$ is obviously $\W$-invariant, an analogous statement is true for $[C_L] \in \h^*_{\T \times \U(n)}(\Hom{n}{2n})$:

\[\kappa_{\T}([C_L]) = [C_L \git \S] \in \h^*_{\T}(\Grass{n}{2n}).\]
On the other hand, from definition of the Kirwan map $\kappa_{\T}$ is is clear that $\kappa_{\T}$ maps $[C_L]$ to $[LG(n)] \in \h^*_{\T}(\Grass{n}{2n})$, hence 
\[ [LG(n)] = \kappa_{\T} \big( \prod_{\substack{i,j = 1 \\ i<j}}^n (z_i + z_j)\big).\]
\end{remark}

\begin{remark} \label{remark2}\label{rem:fund-class-lg3}
The map $\kappa_{\T}^{\S}$ maps the class $\prod_{i \neq j} (z_i - z_j) \in \h^*_{\T \times \S}(\Hom{n}{2n})$ to the class $e = \prod_{\gamma \in \Phi} e(L_{\gamma}) \in  \h^*_{\T}(\Hom{n}{2n} \git \S)$:
\[\kappa_{\T}^{\S} \big( \prod_{i \neq j} (z_i - z_j) \big) = e.\]
It follows directly from the identification of the roots of $\U(n)$---which are exactly $\{z_i - z_j: i \neq j\}$---with cohomology classes. Again, the class is $\W$-invariant, so the same statement is true in $\h^*_{\T \times \U(n)}(\Hom{n}{2n})$.
\end{remark}

We are now ready to prove the following theorem, expressing the push-forward to a point in $\T$-equivariant cohomology of $LG(n)$ as a residue at infinity.

\begin{thm}\label{thm:formula-lg}
Let $j: C_L \hookrightarrow \Hom{n}{2n}$ denote the natural inclusion.
Let $\alpha \in \h^*_{\T \times \U(n)}(\Hom{n}{2n})$. Then

\[ \int\limits_{LG(n)} j^* \kappa_{\T}(\alpha) = \frac{1}{|\W|} \Res_{z_1, \dots, z_n = \infty}\fr{\alpha(z_1,\dots,
z_n) \prod_{i \neq j} (z_i - z_j) \prod_{i < j} (z_i + z_j)}{\prod_{i,j=1}^n(z_i - t_j)(z_i + t_j)}.\]

\end{thm}

%
%

\begin{proof}[Proof of Theorem \ref{thm:formula-lg}]
One has
\begin{align*}
\int\limits_{LG(n)} j^* \kappa_{\T}(\alpha) &\eq{1} \int\limits_{\Grass{n}{2n}} \kappa_{\T}(\alpha) \cdot [LG(n)] \\
 &\eq{2} \int\limits_{\Grass{n}{2n}} \kappa_{\T}(\alpha) \cdot \kappa_{\T} \big( \prod_{\substack{i,j = 1 \\ i<j}}^n (z_i + z_j) \big) \\
 &\eq{3} \int\limits_{\Grass{n}{2n}} \kappa_{\T} \big( \alpha \cdot \prod_{\substack{i,j = 1 \\ i<j}}^n (z_i + z_j) \big) \\ 
 &\eq{4} \frac{1}{|\W|} \int\limits_{\Hom{n}{2n} \git \S} \kappa_{\T}^\S \big( \alpha \cdot \prod_{\substack{i,j = 1 \\ i<j}}^n (z_i + z_j) \big) \cdot e \\
 &\eq{5} \frac{1}{|\W|} \int\limits_{\Hom{n}{2n} \git \S} \kappa_{\T}^\S \big( \alpha \cdot \prod_{\substack{i,j = 1 \\ i<j}}^n (z_i + z_j) \cdot \prod_{i \neq j} (z_i - z_j) \big) \\
&\eq{6} \frac{1}{|\W|} \Res_{z_1, \dots, z_n = \infty}\fr{ \alpha(z_1,\dots,
z_n) \prod_{i < j} (z_i + z_j)\prod_{i \neq j} (z_i - z_j)}{\prod_{i,j=1}^n(z_i - t_j)(z_i + t_j)}, 
\end{align*}
where $\alpha(z_1,\dots,z_n) \in \C[t_1,\dots,t_n, z_1,\dots,z_n]^{\Sigma_n}$ is a $\Sigma_n$-symmetric polynomial representing the restriction of $\kappa_{\T}^\S(\alpha)$ to $0 \in \Hom{n}{2n}$. \newline

The above sequence of equalities follows from the following:
\begin{itemize}
	\item Equation $ \textcircled{\tiny{1}}$ is the relation between push-forwards for $\Grass{n}{2n}$ and $LG(n)$, as described in Sect.~\ref{subsc:relating}.
	\item Equation $ \textcircled{\tiny{2}}$ follows from the computation of $[LG(n)]$, as in Remark \ref{remark12}.
	\item Equation $ \textcircled{\tiny{3}}$ follows from the fact that the Kirwan map is a ring homomorphism.
	\item Equation $ \textcircled{\tiny{4}}$ is the Theorem \ref{thm:equiv-martin} applied to $\Hom{n}{2n}$ and its symplectic reductions with respect to $\S < \U(n)$.
	\item Equation $ \textcircled{\tiny{5}}$ follows from the computation of the class $e$, as in Remark \ref{remark2}.
	\item Equation $ \textcircled{\tiny{6}}$ is a consequence of Theorem \ref{thm:equiv-gk-main}.
\end{itemize}
\end{proof}

\begin{corollary}\label{cor:lg-2}
Let $\alpha \in \h^*_{\T \times \U(n)}(\Hom{n}{2n})$. Then
\[ \int\limits_{LG(n)} j^* \kappa_{\T}(\alpha) = \frac{1}{|\W|} \Res_{z_1, \dots, z_n = \infty}\fr{ \prod_{i < j} (z_i - z_j) \alpha(z_1,\dots,z_n)}{\prod_{i,j=1}^n(z_i^2 - t_j^2)\prod_{i<j}(t_i^2 - t_j^2)}.\]
\end{corollary}

\begin{proof}[Proof of corollary \ref{cor:lg-2}] The proof is a straightforward computation.

\end{proof}

	\section{Orthogonal Grassmannians}\label{sc:orthogonal}

Let $\Omega$ be a nondegenerate symmetric form on $\C^{2n}$ (or $\C^{2n+1}$ respectively).\footnote{In the standard basis the form $\Omega$ is given by a matrix with ones on the antidiagonal.} The orthogonal Grassmannian $OG(n,2n)$ ($OG(n,2n+1)$ respectively) parametrizes the maximal isotropic subspaces. The residue-type formulas for the Gysin homomorphism for the orthogonal Grassmannians $OG(n,2n)$ and $OG(n,2n+1)$ can be derived analogously to the ones for the Lagrangian Grassmannian by using the embeddings $OG(n, 2n) \hookrightarrow \Grass{n}{2n}$ and $OG(n, 2n+1) \hookrightarrow \Grass{n}{2n+1}$. The cases of odd- and even-dimensional orthogonal Grassmannian are slightly different, so we consider them separately. 

		\subsection{The even-dimensional orthogonal Grassmannian }\label{subsc:og-even}
The orthogonal Grassmannian $OG(n,2n)$ canonically embeds in $\Grass{n}{2n}$. The maximal torus in $SO(2n)$ consists of diagonal matrices of the form 
\[\diag(t_1,t_2,\dots,t_n,t_n^{-1},\dots,t_2^{-1},t_1^{-1}),\]
 so we have a natural inclusion of the maximal tori $\T_{SO(2n)} \hookrightarrow \T_{\U(2n)}$:
\[\diag(t_1,t_2,\dots,t_n,t_n^{-1},\dots,t_2^{-1},t_1^{-1}) \mapsto \diag(t_1,t_2,\dots, t_{2n}).\]
We consider Gysin homomorphisms in $\T$-equivariant cohomology of the classical and orthogonal Grassmannian, for $\T = \T_{SO(2n)}$, constructed as symplectic reductions with respect to an $\U(n)$-action. As before, $\S$ denotes the maximal torus on $\U(n)$ and its characters are denoted by $\mathbf{z} = \{z_1,\dots,z_n\}$ and $\W$ denotes the Weyl group of $\U(n)$. \newline

Since $OG(n,2n)$ is a subvariety of $\Grass{n}{2n}$, the push-forwards are related as follows:

\[\int_{OG(n,2n)} \alpha = \int_{\Grass{n}{2n}} \tilde{\alpha} \cdot [OG(n,2n)], \]
where $[OG(n,2n)]$ denotes the fundamental class of the orthogonal Grassmannian in $\h_{\T}^* (\Grass{n}{2n})$ and $\tilde{\alpha}_{|OG(n,2n)} = \alpha$. \newline

Let $C_O$ be the preimage of $OG(n,2n)$ under the symplectic reduction map $q$.

\[
\begin{tikzcd}
C_O \arrow[r, hook, "j"] \arrow[d, "q_{|C_O}", ] &  \arrow[d, "q"] \Hom{n}{2n} \\
 OG(n,2n) \arrow[r, hook] & \Grass{n}{2n}
\end{tikzcd}\]
The subset $C_O \subseteq \Hom{n}{2n}$ consists of linear maps $\C^n \to \C^{2n}$ whose image is an isotropic subspace of maximal dimension in $\C^{2n}$. We checked in Sect.~\ref{subsc:kirwan-lg} that the Kirwan maps commute with restrictions and we established the relation between corresponding push-forwards. The proof can be repeated for the embedding $C_O \hookrightarrow \Hom{n}{2n}$ instead of $C_L \hookrightarrow \Hom{n}{2n}$ (and, in fact, for any subvariety $Z  \hookrightarrow \Grass{n}{2n}$ and its preimage in $\Hom{n}{2n}$). We obtain the following commutative diagram:

\[
\begin{tikzcd}[column sep=huge, row sep=6em]
\h^*_{\T}(OG(n,2n)) \arrow[rd, sloped, "\int\limits_{OG(n,2n)} j^*(-) "] \arrow[dd, leftrightarrow, "\simeq" ] & {} & \h^*_\T(\Grass{n}{2n}) \arrow[dd, leftrightarrow,"\simeq"' ] \arrow[ld, sloped, near end, "\int\limits_{\Grass{n}{2n}} -\cdot {[OG(n,2n)]}" ] \arrow[ll, "j^*"] \\
{} & \h^*_\T(pt) & {} \\
  \h^*_{\T}(C_O \git \S)^\W \arrow[ru, sloped, near end, "\frac{1}{|\W|}\int\limits_{C_O \git \S} j^*(- \cdot e)"] & {}& \h^*_\T(\Hom{n}{2n} \git \S)^\W  \arrow[ll, "j^*"] \arrow[lu, sloped, near start,  "\frac{1}{|\W|} \!\!\!\! \int\limits_{\Hom{n}{2n} \git \S} - \cdot e \cdot{[C_O \git \S]}"] 
\end{tikzcd}
\]
It follows, that the push-forward is given by the formula:

\begin{align*}
 \int\limits_{OG(n,2n)}\kappa_\T(\alpha) &=  \int\limits_{\Grass{n}{2n}} \kappa_\T(\alpha \cdot \widetilde{OG(n,2n)}) \\
&=  \frac{1}{|\W|}\Res_{\mathbf{z}=\infty} \fr{ \prod_{i \neq j} (z_i - z_j) i^*_{0}\kappa^{\S}_{\T}(\alpha \cdot \widetilde{OG(n,2n)})}{e^{\T \times \S}(0)},
\end{align*}
where $\widetilde{OG(n,2n)} \in \h^*_{\T \times \S}(\Hom{n}{2n})^\W$ is a class whose image under $\kappa_\T$ is the fundamental class of $C_O \git \S$. As before, $e^{\T \times \S}(0)$ denotes the $\T \times \S$-equivariant Euler class of the normal bundle at $0$.

\begin{lemma}\label{lem:og} The lift of the fundamental class $[OG(n,2n)]$ to $\h^*_{\T \times \S}(\Hom{n}{2n})$ equals
\[ [\widetilde{OG(n,2n)}]  = \pr{i,j=1}{i\leq j}^n (z_i + z_j).\]
\end{lemma} 

\begin{proof}[Proof of Lemma \ref{lem:og}]
The fundamental class $[C_O] \in \h^*_{\T \times \S}(\Hom{n}{2n})$ can be computed by describing $C_O$ by equations, similarly as in Remark \ref{rem:fund-class-lg}. Let $A \in \Hom{n}{2n}$ be represented by the matrix with column vectors $v_1,\dots,v_n$, then the maximal torus $\S$ in $\U(n)$ acts on $A$ by multiplication of each column vector by a corresponding character. The matrix $A$ is contained in $C_O$ if the column vectors of $A$ satisfy
\[\Omega(v_i, v_j) = 0 \textrm{ for } i \leq j.\]
Under the action of $\S$ the function $\Omega(v_i, v_j)$ is multiplied by $z_i z_j$, while $\T$ acts on $\Omega(v_i, v_j)$ trivially, hence the weight of the $\S$-action on this equation is $z_i + z_j$.  The equations are clearly independent. The are $\frac{n(n+1)}{2}$ equations, which equals the codimension of $C_O$ in $\Hom{n}{2n}$. It follows that $C_O$ is a complete intersection in $\Hom{n}{2n}$ and the fundamental class equals the product of the weights of the equations:
\[ [C_O] =2^n  \prod_{\substack{i,j = 1 \\ i<j}}^n (z_i + z_j) \prod_{i=1}^n z_i.\]
This class is obviously $\W$-invariant, hence belongs to $\h^*_{\T \times \U(n)}(\Hom{n}{2n})$ and the Kirwan map clearly maps it to $[OG(n,2n)]$.
\end{proof}

Finally, since $e^{\T \times \S}(0) = \prod_{i,j=1}^n(z_i - t_j)(z_i + t_j)$ we get the final formula for the Gysin map:

\begin{thm}\label{thm:og-even}
Let $j: C_O \hookrightarrow \Hom{n}{2n}$ denote the natural inclusion.
Let $\alpha \in \h^*_{\T \times \U(n)}(\Hom{n}{2n})$. Then
\[ \int\limits_{OG(n,2n)}\kappa_\T(\alpha) =  \frac{2^n}{|\W|}\Res_{z=\infty} \fr{ \alpha(z_1,\dots,z_n) \prod_{i \neq j} (z_i - z_j)\prod_{i<j} (z_i + z_j) \prod_{i=1}^n  z_i}{\prod_{i,j=1}^n(z_i - t_j)(z_i + t_j)}.\]
\end{thm}

		\subsection{The odd-dimensional orthogonal Grassmannian}\label{subsc:og-odd}

The case of $OG(n,2n+1)$ is very similar to the case of $OG(n,2n)$, but there is one essential difference that should be pointed out. 
The Grassmannian $OG(n,2n+1)$ canonically embeds in $\Grass{n}{2n+1}$, and the corresponding embedding of the maximal tori $\T_{SO(2n+1)} \hookrightarrow \T_{\U(2n+1)}$ is given by:

\[\diag(t_1,t_2,\dots,t_n,t_n^{-1},\dots,t_2^{-1},t_1^{-1},1) \mapsto \diag(t_1,t_2,\dots, t_{2n},1).\]
Hence the weights of the action of $\T = T_{SO(2n+1)}$ are $\{ \pm t_i \} \cup \{ 0 \}$ and the $\T \times \S$-equivariant Euler class in this case equals:
\[e^{\T \times \S}(0) = \prod_{i,j=1}^n(z_i - t_j)(z_i + t_j) \prod_{i=1}^{n} z_i.\]

All other computations from Sect.~\ref{subsc:og-even} remain valid. The fundamental class of $C_O$ in $\Hom{n}{2n+1}$ is given by the same polynomial  

\[ [C_O] = 2^n \prod_{\substack{i,j = 1 \\ i<j}}^n (z_i + z_j)\prod_{i=1}^n  z_i ,\]
because the conditions on isotropy of the column vectors of a matrix representing an element of $\Hom{n}{2n+1}$ are the same. \newline

The formula for the push-forward in $\T$-equivariant cohomology is therefore the following.

\begin{thm}\label{thm:og-odd}
Let $j: C_O \hookrightarrow \Hom{n}{2n+1}$ denote the natural inclusion.
Let $\alpha \in \h^*_{\T \times \U(n)}(\Hom{n}{2n})$. Then
\[ \int\limits_{OG(n,2n)}\kappa_\T(\alpha) =  \frac{2^n}{|\W|}\Res_{\mathbf{z}=\infty} \fr{\alpha(z_1,\dots,z_n) \prod_{i \neq j} (z_i - z_j)\prod_{i<j} (z_i + z_j)}{\prod_{i,j=1}^n(z_i - t_j)(z_i + t_j)}.\]
\end{thm}

	\section{Partial flag varieties}\label{sc:partial}

The results obtained for Grassmannians can be generalized to an arbitrary partial flag variety (of type A, B, C or D). Partial flag varieties of type $A$ can be realized as symplectic reductions as described in Sect.~\ref{sc:seriesA} below (the result is due to Kamnitzer \cite{kamnitzer}). The push-forward formulas for partial flag varieties of types $B, C$ and $D$ can be deduced from the type $A$ case by considering their embeddings into type A partial flag varieties, analogously as in the Grassmannian case.

		\subsection{Series A partial flag varieties}\label{subsc:seriesA}\label{seriesA}\label{sc:seriesA}

Consider the partial flag variety of type $d=(d_1,\dots,d_k)$ in $W \simeq \C^n$
\[\Fl_d(W) = \{ V_1 \subset V_2 \subset \dots \subset V_k \subset W :  V_i \textrm{ linear subsapce of } W, \dim V_i = d_i \}.\]
One can show (see \cite{kamnitzer}) that $\Fl_d(W)$ can be constructed as a symplectic reduction, as follows. Let $\{ V_1, V_2, \dots, V_k \}$ be a collection of vector spaces of dimensions $\dim V_i = d_i$. Define
\[\homm(V,W):= \bigoplus_{i=1}^{k-1} \homm(V_i,V_{i+1}) \oplus \homm(V_k, W).\]
This is a symplectic manifold, since each of the components $\homm(V_i,V_{i+1})$ is symplectic, with the symplectic form given by $\omega(A,B) = 2 \Imm(\tr A B^*)$.
Let $\U(V):=\U(V_1) \times \dots \times \U(V_k)$ and consider the action of $\U(V)$ on $\homm(V,W)$ given by
\[(g_1, \dots, g_k)(A_1, \dots, A_{k-1}, B) = (g_2 A_1 g_1^{-1},\dots, g_k A_{k-1} g_{k-1}^{-1}, B g_k^{-1}). \]
This action is Hamiltonian with moment map $\mu: \homm(V,W) \to \lie{u}(V)$
\[\mu(A_1, \dots, A_{k-1}, B) = (A_1^* A_1, A_2^* A_2 - A_1 A_1^*, \dots, B^* B - A_{k-1} A_{k-1}^*)\]
If we choose $\lambda=(\lambda_1 \Id_{d_1}, \dots, \lambda_k \Id_{d_k}) \in \lie{u}(V)$ such that the real numbers $(\lambda_1, \dots,\lambda_k )$ are linearly independent over $\QQ$, then there is an isomorphism

\[\homm(V,W) \git_\lambda \  \U(V) \simeq \Fl_d(W).\]

Consider two torus actions on $\homm(V,W)$: 
\begin{itemize}
	\item The action of the maximal torus $\S=\S_1 \times \dots \times \S_k$ contained in $\U(V) = \U(V_1) \times \dots \times \U(V_k)$ acting by the restriction of the above action.
	\item The action of the $n$-dimensional torus $\T$ acting in $\homm(V,W)$ by matrix multiplication on the left on the last component $\homm(V_{k}, W)$.
\end{itemize}
Let us denote the characters of $\T$ by $\mathbf{t}=\{t_1,\dots,t_n\}$ and the characters of $\S$ by $\mathbf{z}$. The set $\mathbf{z}$ is the union of the sets of characters of the tori $\S_i$, which we denote by by $z_{i,1}, \dots, z_{i,d_i}$ for $i=1,\dots,k$. Following the same procedure as for the classical Grassmannian we reduce the Gysin homomorphism for $Fl_d(W)$ to the Gysin map for the symplectic reduction of $\homm(V,W) \git \S$ using Theorem \ref{thm:equiv-martin} and use Theorem \ref{thm:equiv-gk-main} to convert it into a residue. Let $\kappa_{\T}$ be the $\T$-equivariant Kirwan map for the action of $\U(V)$ on $\homm(V,W)$ and let $\kappa_{\T}^{\S}$ be the $\T$-equivariant Kirwan map for the action of $\S$ on $\homm(V,W)$. Denote by $e$ the $\T$-equivariant characteristic class corresponding to the product of roots of $\U(V)$ as in Ch.~\ref{ch:martin}. Then  

\begin{equation}\label{eq:preeq-flag}
\begin{split}
\int\limits_{\Fl_d(W)} \kappa_{\T}(\alpha) &= \frac{1}{|\W|} \int\limits_{\homm(V,W) \git \S} \kappa_{\T}^{\S}(\alpha)\cdot e  \\
&= \frac{1}{|\W|} \Res_{\mathbf{w}=\infty} \frac{i^*_{0}\kappa_{\T}^\S(\alpha \cdot \tilde{e})}{e^{\T \times \S}(0)},
\end{split}
\end{equation}
for a certain class $\tilde{e}$ such that $\kappa_{\T}^\S(\tilde{e})=e$. The residue is taken with respect to the weights of the action of $\S$. \newline

The last equality requires a comment on the point $0$ appearing in the formula even though the original reduction was taken with respect to $\lambda$. One can always modify the moment map by adding a constant, providing this constant is invariant with respect to the coadjoint action. In particular, after restricting to the action of the torus $\S$ we can modify the moment map so that the reduction $\homm(V,W) \git \S$ is taken at $0$. \newline

The torus $\S=\S_1 \times \dots \times \S_k$ acts on $\homm(V,W)$ via the restriction of the action of $\U(V)$,
\[(g_1, \dots, g_k)(A_1, \dots, A_{k-1}, B) = (g_2 A_1 g_1^{-1},\dots, g_k A_{k-1} g_{k-1}^{-1}, B g_k^{-1}).\]
In particular $\S$ acts on $A_1 \in \homm(V_1,V_2)$ by multiplying the rows of the $d_2 \times d_1$-matrix $A_1$ by $z_{2,1}, z_{2,2},\dots, z_{2,d_2}$ and by multiplying the columns by $z_{1,1}^{-1}, \dots, z_{1,d_1}^{-1}$ etc. The torus $\T$ acts on  $\homm(V,W)$ via
\[t \cdot (A_1, \dots, A_{k-1}, B) = (A_1, \dots, A_{k-1}, t \cdot B).\]

\begin{remark}\label{rem:substitution}
If we eliminate the matrix notation and choose the natural order of variables on $\homm(V,W)$, identifying it with a linear space of dimension $d_1 d_2 + d_2 d_3 + \dots d_{k-1}d_k + d_k n$, then the torus $\S$ acts on the $\homm(V,W)$ with the set of characters:
\[\{ z^{-1}_{1,i} z_{2,j}\}_{\substack{i=1,\dots,d_1 \\ j=1,\dots,d_2}} , \{ z^{-1}_{2,i} z_{3,j}\}_{\substack{i=1,\dots,d_2 \\ j=1,\dots,d_3}} , \dots, \{ z^{-1}_{k,i}\}_{i=1,\dots,d_k},\]
with each set of characters indexed by all possible $i,j$ (for example $\{ z^{-1}_{1,j}z_{2,i} \}$ is indexed by $i=1,\dots,d_1$ and $j=1\dots,d_2$ and so on). \newline

For computational reasons it is better to work with coordinated "adjusted" to the action, which is why we consider the following substitution:

\[
\begin{cases}
	v_{i,j} = z_{i+1,j} - z_{i,j} & \mbox{ for } i=1,\dots,k-1 \mbox{ and } j=1,\dots,d_i \\
	u_i = z_{k,i} & \mbox{ for } i=1,\dots,d_k
\end{cases}
\]
Denote by $\mathbf{v}$ the set $\{v_{i,j} \}_{\substack{i=1,\dots,k-1 \\ j=1,\dots,d_i }}$ and by $\mathbf{u}$ the set $\{ u_i \}_{i=1,\dots,d_k}$. With this substitution the weights of the $\S$ action are $u_i$, $v_{i,j}$ as above and their linear combinations. For details on this substitution see Appendix \ref{appendixA}. However, we express all final results in the original coordinates $\mathbf{z}$, referring to the substitution only while making simplifications in the computations. 
\end{remark}

\begin{proposition}\label{prop:etildaflag}
The class $\tilde{e}$ such that $\kappa_{\T}^{\S}(\tilde{e}) = e$ equals
\[\tilde{e}= \prod_{\substack{i,j = 1 \\ i \neq j}}^{d_1} (z_{1,i} - z_{1,j})  \prod_{\substack{i,j = 1 \\ i \neq j}}^{d_2} (z_{2,i} - z_{2,j})) \dots  \prod_{\substack{i,j = 1 \\ i \neq j}}^{d_k} (z_{k,i} - z_{k,j})\]
Changing the variables to  $\mathbf{u}, \mathbf{v}$ as in Remark \ref{rem:substitution} results in a polynomial $f_{\tilde{e}} \in \C[\mathbf{u}, \mathbf{v}]$ representing $\tilde{e}$. We would not need the exact form of this polynomial but the reader may find it in the Appendix \ref{appendixA}.
\end{proposition}

By the considerations identical to the ones for the classical Grassmannian in Sect.~\ref{sc:grassmannian} we get the following two consequences:

\begin{proposition}\label{prop:eulerflag} The $\T \times \S$-equivariant Euler class at zero equals
\[e^{\T \times \S}(0) = \prod_{i=1}^{k} \prod_{l=1}^{d_i} \prod_{m=1}^{d_{i+1}} (-z_{i,l} + z_{i+1, m})\cdot \prod_{j=1}^{d_k} \prod_{i=1}^{n}(-z_{k,j}+t_{i})\] 
In the variables $\mathbf{u}, \mathbf{v}$ this class equals
\[e^{\T \times \S}(0) = \prod_{m=1}^{k-1} \prod_{i=1}^{d_{m+1}} \prod_{j=1}^{d_m} (u_i - u_j - \sum_{n=m+1}^{k-1} v_{n,i} + \sum_{n=m}^{k-1} v_{n,j}) \prod_{i=1}^{n} \prod_{j=1}^{d_k}(t_{i}-u_j)\] 
\end{proposition}

\begin{proof}[Proof of propositions \ref{prop:etildaflag} and \ref{prop:eulerflag}]
The proof of Proposition \ref{prop:etildaflag} is identical to the proofs of Proposition \ref{remark2} for classical Grassmannian. Proposition \ref{prop:eulerflag} is a direct consequence of Remark \ref{rem:substitution}. \qedhere
\end{proof}

\begin{proposition}\label{prop:preformula-flag}
Let $\alpha \in \h_{\T \times \U(V)}^*(\homm(V,W))^{\W}$, where $\W = \Sigma_{d_1} \times \dots \times \Sigma_{d_k}$  is the Weyl group of $\U(V)$. Hence $\alpha$ is a polynomial in variables $\mathbf{z}, \mathbf{t}$, which is $\W$-symmetric in the variables $\mathbf{z}$. 
\[
\int\limits_{\Fl_d(W)} \kappa_{\T}(\alpha) 
 = \frac{1}{|\W|} \Res_{\mathbf{w}=\infty} \fr{ \alpha \cdot \prod_{\substack{i,j = 1 \\ i \neq j}}^{d_1} (z_{1,i} - z_{1,j})  \prod_{\substack{i,j = 1 \\ i \neq j}}^{d_2} (z_{2,i} - z_{2,j}) \dots \prod_{\substack{i,j = 1 \\ i \neq j}}^{d_k} (z_{k,i} - z_{k,j})}{\prod_{i=1}^{k} \prod_{l=1}^{d_i} \prod_{m=1}^{d_{i+1}} (-z_{i,l} + z_{i+1, m})\cdot \prod_{l=1}^{d_k} \prod_{m=1}^{n}(-z_{k,l}+t_{m})}
\]
and the residue is taken with respect to $z_{2,j}-z_{1,i} = \infty$, $z_{3,j}-z_{2,i} = \infty$, \dots, $z_{k,i} = \infty$, because these are the weights of the action. 
\end{proposition}

\begin{proof}[Proof of Proposition \ref{prop:preformula-flag}]
Using the substitution as in Remark \ref{rem:substitution}, the class $\alpha$ can be written as a polynomial in variables $\mathbf{u}, \mathbf{v}$, invariant under the induced action of the Weyl group. By equation \eqref{eq:preeq-flag} and Propositions \ref{prop:etildaflag} and \ref{prop:eulerflag} we get the following expression for the push-forward of a class $\kappa(\alpha)$.

\[
\int\limits_{\Fl_d(W)} \kappa_{\T}(\alpha) = \frac{1}{|\W|} \int\limits_{\homm(V,W) \git \S} \kappa_{\T}^{\S}(\alpha \cdot \tilde{e})  = \frac{1}{|\W|} \Res_{\mathbf{v,u}=\infty} \alpha \cdot \bigstar,\]
where the factor under the residue, $\bigstar$, equals
\[ \fr{ \prod_{\substack{i,j = 1 \\ i \neq j}}^{d_1} (u_i - u_j - \! \sum_{n=1}^{k-1} (v_{n,i} - v_{n,j}) )  \prod_{\substack{i,j = 1 \\ i \neq j}}^{d_2} ( u_i - u_j - \! \sum_{n=2}^{k-1}(v_{n,i} - v_{n,j}) )  \dots \!\! \prod_{\substack{i,j = 1 \\ i \neq j}}^{d_k} (u_i - u_j)}{\prod_{m=1}^{k-1} \prod_{i=1}^{d_{m+1}} \prod_{j=1}^{d_m} (u_i - u_j - \sum_{n=m+1}^{k-1} v_{n,i} + \sum_{n=m}^{k-1} v_{n,j})\prod_{i=1}^{n} \prod_{j=1}^{d_k}(t_{i}-u_j)}\]
and the residue is taken with respect to all $v_{i,j} = \infty$, $u_i = \infty$. Coming back to the original variables $\mathbf{z}$ proves Proposition \ref{prop:preformula-flag}.

\end{proof}
One can further simplify this expression to a residue only with respect to the variables $u_1=z_{k,1}, \dots, u_k=z_{k,d_k}$, using the following computational lemma:

\begin{lemma}\label{lem:residue}

Let $f \in \C[\mathbf{v}, \mathbf{u}]$ be a polynomial. Let $|\mathbf{v}| = d_1 + \dots + d_{k-1}$ denote the number of the variables $v_{i,j}$. Then the following equality holds:

\[ \Res_{\mathbf{v}=\infty} \fr{f(\mathbf{v}, \mathbf{u})}{\prod_{m=1}^{k-1} \prod_{i=1}^{d_{m+1}} \prod_{j=1}^{d_m} (u_i - u_j - \sum_{n=m+1}^{k-1} v_{n,i} + \sum_{n=m}^{k-1} v_{n,j}) \prod_{i=1}^{n} \prod_{j=1}^{d_k}(t_{i}-u_j) } = \]
\[= (-1)^{|\mathbf{v}|}\fr{f(\mathbf{0}, \mathbf{u})}{\prod_{m=1}^{k-1} \prod_{i=1}^{d_{m+1}} \prod_{\substack{j=1 \\ j \neq i}}^{d_m} (u_i - u_j)  \prod_{i=1}^{n} \prod_{j=1}^{d_k}(t_{i}-u_j)}.\]
Note that unlike in the push-forward formula we intend to prove, in this lemma we only take the residue with respect to the $v_{i,j}$'s, ale leave the variables $u_i$ untouched.  
\end{lemma}

\begin{proof}[Proof of Lemma \ref{lem:residue}]

We start by rewriting the denominator by extracting the product $\prod_{m=1}^{k-1}\prod_{j=1}^{d_m} v_{m,j}$:

\[\prod_{m=1}^{k-1} \prod_{i=1}^{d_{m+1}} \prod_{j=1}^{d_m} (u_i - u_j - \sum_{n=m+1}^{k-1} v_{n,i} + \sum_{n=m}^{k-1} v_{n,j}) \prod_{i=1}^{n} \prod_{j=1}^{d_k}(t_{i}-u_j) = \]
\[=\prod_{m=1}^{k-1}\prod_{j=1}^{d_m} v_{m,j} \prod_{m=1}^{k-1} \prod_{i=1}^{d_{m+1}} \prod_{\substack{j=1 \\ j \neq i}}^{d_m} (u_i - u_j - \sum_{n=m+1}^{k-1} v_{n,i} + \sum_{n=m}^{k-1} v_{n,j}) \prod_{i=1}^{n} \prod_{j=1}^{d_k}(t_{i}-u_j) = g(\mathbf{v},\mathbf{u}) \]
The expression under the residue
\[\fr{f(\mathbf{v}, \mathbf{u})}{g(\mathbf{v}, \mathbf{u})}\]
has poles along the union of divisors defined by the factors of $g(\mathbf{v}, \mathbf{u})$, which form a normal crossing. Hence the iterated residue can be computed by taking the residues in whichever order.\footnote{See the definition of the iterated residue in Sect.~\ref{sc:nonabelian} for explanation.} The residue at $v_{m,j}=\infty$ equals minus the residue at $v_{m,j} = 0$ by the Residue Theorem, and taking the residue at the simple pole $v_{i,j} = 0$ removes the factor $v_{i,j}$ from the denominator and substitutes $v_{i,j} = 0$ whenever $v_{i,j}$ appears in the remaining expression. Hence taking the residues with respect to all the variables $v_{i,j}$ results in removing the factor $\prod_{m=1}^{k-1}\prod_{j=1}^{d_m} v_{m,j}$ from the denominator, multiplying the formula by $(-1)^{N}$, where $N$ is the number of residues taken i.e. $N = |\mathbf{v}|$ and substituting all $v_{i,j} = 0$ in the remaining formula.

\end{proof}

Using Lemma \ref{lem:residue} we can get the following simplification of the push-forward formula for partial flag manifolds.

\begin{thm}\label{thm:formula-partial-vu} 
Let $\alpha \in \h^*_{\T \times \U(V)}(\homm(V,W))$. Identify $\h^*_{\T \times \U(V)}(\homm(V,W))$ with the ring of $\W$-symmetric polynomials in the variables $\mathbf{v}, \mathbf{u}, \mathbf{t}$. The $\T$-equivariant push-forward to a point is given by the following formula. 

\[\int\limits_{\Fl_d(W)} \kappa_{\T}(\alpha) = \frac{1}{|\W|} \Res_{\mathbf{u} =\infty}\fr{ \alpha \cdot \prod_{\substack{i \neq j  \\ (i,j) \in I_{Fl}}} (u_i - u_j) }{\prod_{i=1}^{n} \prod_{j=1}^{d_k}(t_{i}-u_j)},\]
where the indexing set $I_{Fl}$ is depicted on the following diagram (the indices coloured in grey belong to $I_{Fl}$, the white ones don't). For an explicit description of the set $I_{Fl}$ see Appendix \ref{appendixB}.
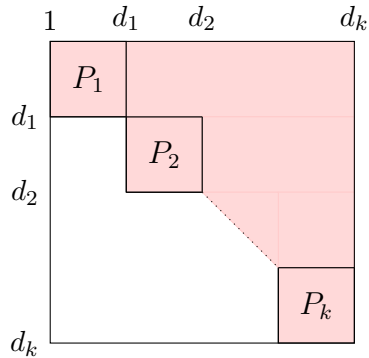
\begin{center}
\begin{tikzpicture}

\filldraw[fill=red!15!white] (0,4) rectangle (1,3);
\filldraw[fill=red!15!white] (1,3) rectangle (2,2);
\filldraw[fill=red!15!white] (3,1) rectangle (4,0);

\filldraw[fill=red!15!white, draw=red!20!white] (1,4) rectangle (4,3);
\filldraw[fill=red!15!white, draw=red!20!white] (2,3) rectangle (4,2);
\filldraw[fill=red!15!white, draw=red!20!white] (3,2) rectangle (4,1);
\filldraw[fill=red!15!white, draw=red!20!white]  (2,2) --  (3,2) --  (3,1) -- cycle;

\draw (0,0) rectangle (4,4);
\draw (1,4) -- (1,2);
\draw (2,3) -- (2,2);
\draw (3,1) -- (3,0);
\draw (0,3) -- (2,3);
\draw (1,2) -- (2,2);
\draw (3,1) -- (4,1);
\draw[dotted] (2,2) --(3,1);

\node at (0.5, 3.5){$P_1$};
\node at (1.5, 2.5){$P_2$};
\node at (3.5, 0.5){$P_k$};
\node[above] at (0, 4){\small{$1$}};
\node[above] at (1, 4){\small{$d_1$}};
\node[above] at (2, 4){\small{$d_2$}};
\node[above] at (4, 4){\small{$d_k$}};
\node[left] at (0, 3){\small{$d_1$}};
\node[left] at (0, 2){\small{$d_2$}};
\node[left] at (0, 0){\small{$d_k$}};

\end{tikzpicture}
\end{center}

\end{thm}

\begin{proof}[Proof of Theorem \ref{thm:formula-partial-vu}]
Applying the Lemma \ref{lem:residue} to the numerator of $\bigstar$, i.e. to the polynomial
\[num(\mathbf{v}, \mathbf{u}) :=  \alpha \cdot \prod_{\substack{i,j = 1 \\ i \neq j}}^{d_1} (u_i - u_j - \sum_{n=1}^{k-1} (v_{n,i} - v_{n,j}) )  \dots \prod_{\substack{i,j = 1 \\ i \neq j}}^{d_k} (u_i - u_j), \]
one gets

\begin{align*}
 \Res_{\mathbf{v}=\infty} & \fr{num(\mathbf{v}, \mathbf{u})}{\prod_{m=1}^{k-1} \prod_{i=1}^{d_{m+1}} \prod_{j=1}^{d_m} (u_i - u_j - \sum_{n=m+1}^{k-1} v_{n,i} + \sum_{n=m}^{k-1} v_{n,j}) \prod_{i=1}^{n} \prod_{j=1}^{d_k}(t_{i}-u_j) } \\
&= (-1)^{ |\mathbf{v}|}\fr{num(\mathbf{0}, \mathbf{u})}{\prod_{m=1}^{k-1} \prod_{i=1}^{d_{m+1}} \prod_{\substack{j=1 \\ j \neq i}}^{d_m} (u_i - u_j)  \prod_{i=1}^{n} \prod_{j=1}^{d_k}(t_{i}-u_j)} \\
&= (-1)^{|\mathbf{v}|}\fr{ \alpha \cdot \prod_{\substack{i,j = 1 \\ i \neq j}}^{d_1} (u_i - u_j)  \prod_{\substack{i,j = 1 \\ i \neq j}}^{d_2} ( u_i - u_j)  \dots \prod_{\substack{i,j = 1 \\ i \neq j}}^{d_k} (u_i - u_j)}{\prod_{m=1}^{k-1} \prod_{i=1}^{d_{m+1}} \prod_{\substack{j=1 \\ j \neq i}}^{d_m} (u_i - u_j)  \prod_{i=1}^{n} \prod_{j=1}^{d_k}(t_{i}-u_j)} \\
&=\fr{ \alpha \cdot \prod_{\substack{i \neq j  \\ (i,j) \in I_{Fl}}}(u_i - u_j) }{\prod_{i=1}^{n} \prod_{j=1}^{d_k}(t_{i}-u_j)}, 
\end{align*}
where the last equality is a straightforward (but computationally involved) simplification of the numerator and the denominator, which are both products of $(u_i-u_j)$ only indexed by different sets of $i,j$'s. The detail of this computations can be found in the Appendix \ref{appendixB}. The minus signs in front of the fraction compensate the changes of orders of variables. \newline

The final formula is obtained by taking the residue both with respect to the variables $v_{i,j}$ and $u_i$, hence
\begin{align*}
\int\limits_{\Fl_d(W)} \kappa_{\T}(\alpha) &= \frac{1}{|\W|} \Res_{\mathbf{v}, \mathbf{u} = \infty} \bigstar \\
&= \frac{1}{|\W|} \Res_{\mathbf{u} = \infty} \fr{ \alpha \cdot \prod_{\substack{i \neq j  \\ (i,j) \in I_{Fl}} }(u_i - u_j) }{\prod_{i=1}^{n} \prod_{j=1}^{d_k}(t_{i}-u_j)}.
\end{align*}
\end{proof}

Coming back to the original variables $\z$ we get the following formula for the push-forward.

\begin{thm}\label{thm:formula-partialA}\label{thm:formula-A}
Let $\alpha \in \h^*_{\T \times \U(V)}(\homm(V,W))$. Identify $\h^*_{\T \times \U(V)}(\homm(V,W))$ with the ring of $\W$-symmetric polynomials in the variables $\mathbf{z}, \mathbf{t}$. The $\T$-equivariant push-forward to a point is given by the following formula. 
\[
\int\limits_{\Fl_d(W)} \kappa_{\T}(\alpha) 
= \frac{1}{|\W|} \Res_{z_{k,1}, \dots, z_{k,d_k}=\infty} \fr{\alpha \cdot \prod_{\substack{i \neq j  \\ (i,j) \in I_{Fl}}} (z_{k,i} - z_{k,j})}{ \prod_{l=1}^{d_k} \prod_{m=1}^{n}(-z_{k,l}+t_{m})}. 
\]
The residue is taken with respect to the variables $\{z_{k,1}, \dots, z_{k,d_k}=\infty\}$ which correspond to the characters of the last component of the torus $\S = \S_1 \times \dots \times \S_k$, not with respect to all the characters.
\end{thm}

		\subsection{Series C partial flag varieties}\label{seriesC}\label{subsc:seriesC}

Let $\omega$ denote the symplectic form on $W \simeq \C^{2n}$. A subspace $V \subseteq W$ is called an \emf{isotropic subspace} if $\omega$ restricts to zero on $V$. Consider the variety of partial isotropic flags of type $d=(d_1,\dots,d_k)$ in $W \simeq \C^{2n}$
\[\Fl^{C}_d(W) = \{ V_1 \subset \dots \subset V_k \subset W :  V_i \textrm{ isotropic subspaces  of } W, \dim V_i = d_i \}.\]
The variety $\Fl^{C}_d(W)$ canonically embeds in $\Fl_d(W)$. Moreover, the maximal torus in $Sp(n)$ embeds in the maximal torus in $\U(2n)$, so there is a natural restriction map form $T_{\U(2n)}$-equivariant cohomology to $T_{Sp(n)}$-equivariant cohomology, which at a point is given by
\[ \C[t_1,\dots,t_{2n}] \to \C[t_1,\dots,t_n,t_n^{-1},\dots,t_1^{-1}].\]
We can thus obtain residue-type push-forward formulas in $T_{Sp(n)}$-equivariant cohomology of $\Fl^{C}_d(W)$ by embedding $\Fl^{C}_d(W) \hookrightarrow \Fl_d(W)$ and restricting to $\T =\T_{Sp(n)}$-equivariant cohomology and then following the same procedure we used to deduce the push-forward formulas for $LG(n)$ from the formula for $\Grass{n}{2n}$. Consider the following diagram

\[
\begin{tikzcd}
C_I \arrow[r, hook, "j"] \arrow[d, "q_{|V}", ] &  \arrow[d, "q"] \homm(V, W) \\
 \Fl^{C}_d(W) \arrow[r, hook, "j"] & \Fl_d(W)
\end{tikzcd}\]
where $q$ is the symplectic reduction map for the action of $\U(V) = \U(V_1)\times \dots \times \U(V_k)$ on $\homm(V,W) = \bigoplus_{i=1}^{k} \homm(V_i, V_{i+1}) \oplus \homm(V_k, W)$ as defined in Sect.~\ref{subsc:seriesA}. Let $\S$ denote the maximal torus in $\U(V)$, acting on $\homm(V,W)$ by restriction of the action of $\U(V)$ and let 
\[\mathbf{z}=\{ z_{1,1}, \dots, z_{1,d_1} \} \cup \{z_{2,1}, \dots, z_{2,d_2} \} \cup \dots \cup \{ z_{k,1}, \dots,  z_{k,d_k} \}\]
 denote the characters of the $\S$-action. \newline

The $\T$-equivariant push-forwards to a point for $\Fl^{C}_d(W)$ and $\Fl_d(W)$ are related as follows
\[\int\limits_{\Fl^{C}_d(W)} \alpha =\int\limits_{\Fl_d(W)} \tilde{\alpha} \cdot [\Fl^{C}_d(W)],  \]
where $[\Fl^{C}_d(W)]$ denotes the fundamental class of $\Fl^{C}_d(W)$ in $\h^*_{\T}(\Fl_d(W))$ and $\tilde{\alpha}$ is the lift of $\alpha \in \h^*_{\T}(\Fl^{C}_d(W))$ to $\h^*_{\T}(\Fl_d(W))$.

Let $\kappa_{\T}$ denote the Kirwan map associated to the action of $\U(V)$ on $\homm(V,W)$. One has
\[\int\limits_{\Fl^{C}_d(W)} j^* \kappa_{\T}(\alpha) = \int\limits_{\Fl_d(W)} \kappa_\T(\alpha) \cdot [\Fl^{C}_d(W)] = \int\limits_{\Fl_d(W)} \kappa_\T(\alpha \cdot \widetilde{\Fl^{C}_d(W)}),\]
where $\widetilde{[\Fl^{C}_d(W)]}$ is an element of $\h^*_{\T \times \U(V)}(\homm(V,W))$ such that $\kappa_{\T}(\widetilde{[\Fl^{C}_d(W)]})=[\Fl^{C}_d(W)]$. Once we compute the class $\widetilde{[\Fl^{C}_d(W)]}$ we would be able to write down the residue-type formula for the Gysin map.

\begin{remark}[The fundamental class of $C_I$] \label{rem:fund-class-C}
The fundamental class $[C_I] \in \h^*_{\T \times \S}(\homm(V,W))$ can be computed analogously as the class of the Lagrangian cone in the classical Grassmannian in Remark \ref{rem:fund-class-lg}. An element of $\homm(V,W)$ is represented by an $k$-tuple of matrices $(A_1, \dots, A_{k-1},B)$. \newline   

The case of matrix $B$ is completely analogous to the case of the Lagrangian Grassmannian. The subspace represented by $B \in \homm(V_k,W)$ is the span of the column vectors $b_1,\dots,b_{d_k}$ of $B$. It is isotropic if the symplectic form in $W$ is zero on each pair of them:
\begin{equation} \label{omega_flag} 
\omega(b_i, b_j) = 0 \textrm{ for } i<j
\end{equation}
The action of the torus $\S= \S_1 \times \dots \times \S_k$ on the matrix $B$ is given by $(g_1,\dots, g_k) B = B g_k^{-1}$, hence it multiplies the column vector $b_i$ by $z^{-1}_{k,i}$. In particular, the weight of the $\S$-action on $\omega(b_i, b_j)$ is $z^{-1}_{k,i} + z^{-1}_{k,j}$, and $\T$ acts trivially on $\omega(b_i, b_j)$.
Let us now consider the matrices $A_i$, $i=1,\dots, k-1$. The matrix $A_i \in Hom(V_i, V_{i+1})$ determines a subspace of $W$ via the composition 
\[ V_i \xrightarrow{A_i} V_{i+1} \xrightarrow{A_{i+1}} V_{i+2} \xrightarrow{A_{i+2}} \dots \xrightarrow{A_{k-1}} B \to W,\]
i.e. the subspace of $W$ represented by this composition is spanned by the column vectors $c_1, \dots, c_{d_i}$ of the matrix $C = B A_{k-1} \dots A_{i+1} A_i$. A $k$-tuple $(A_1,\dots,A_{k-1}, B)$ lies in $C_I$ if and only if each $A_i$ and $B$ represent isotropic subspaces of $W$, which means that the symplectic form $\omega$ restricts to zero on these subspaces. Note, that no matter what the linear map 
\[ V_i \xrightarrow{A_i} V_{i+1} \xrightarrow{A_{i+1}} V_{i+2} \xrightarrow{A_{i+2}} \dots \xrightarrow{A_{k-1}} B\]
looks like, as long as the image of the last map $B \to W$ is an isotropic subspace of $W$, the image of the whole composition $V_i \xrightarrow{C} W$ is an isotropic subspace of $W$. In particular, we do not get any more independent equations for $C_I$. Hence the fundamental class of $C_I$ in $\homm(V,W)$ equals
\[ [C_I] = \prod_{\substack{i,j=1 \\ i<j }}^{d_k} (z_{k,i} + z_{k,j}).\]

\end{remark}

\begin{remark}\label{remark12flag}\label{rem:fund-class-C2}
The class $[C_I] \in \h^*_{\T \times \S}(\homm(V,W))$ is obviously $\W$-invariant, hence an analogous statement is true for the class $[C_I] \in \h^*_{\T \times \U(V)}(\homm(V,W))$ (which is why we also don't distinguish them in the notation). From the definition of the Kirwan map one sees that the class $[C_I]$ is mapped by the Kirwan map $\kappa_{\T}$ to $[\Fl^C_d(W)] \in \h^*_{\T}(\Fl_d(W))$, hence 
\[ [\Fl_d^C(W)] = \kappa_{\T} (\prod_{\substack{i,j=1 \\ i<j }}^{d_k} (z_{k,i} + z_{k,j})).\]
\end{remark}

Finally, we get the following expression for the push-forward of the image of a class $\alpha \in \h_{\T \times \U(V)}^*(\homm(V,W))^{\W}$.

\begin{thm} \label{thm:formula-C}
Let $\alpha \in \h_{\T \times \U(V)}^*(\homm(V,W))^{\W}$, where $\W = \Sigma_{d_1} \times \dots \times \Sigma_{d_k}$  is the Weyl group of $\U(V)$. We get the following expression for the push-forward of a class $\kappa(\alpha)$.
\begin{align*}
\int\limits_{\Fl^C_d(W)} \!\!\!\!\! \kappa_{\T}(\alpha) & = \frac{1}{|\W|} \int\limits_{\homm(V,W) \git \S} \kappa_{\T}^{\S}(\alpha \cdot \tilde{e}) \\
&= \frac{1}{|\W|} \Res_{\mathbf{z}=\infty}\fr{ \alpha \cdot \prod_{\substack{i,j = 1 \\ i \neq j}}^{d_1} (z_{1,i} - z_{1,j}) \dots \prod_{\substack{i,j = 1 \\ i \neq j}}^{d_k} (z_{k,i} - z_{k,j})\prod_{\substack{i,j = 1 \\ i < j}}^{d_k} (z_{k,i} + z_{k,j})}{\prod_{i=1}^{k} \prod_{l=1}^{d_i} \prod_{m=1}^{d_{i+1}} (z_{i+1, m}-z_{i,l}) \prod_{l=1}^{d_k} \prod_{m=1}^{n}(t_m -z_{k,l})(t_m + z_{k,l})},
\end{align*}
where the residue is taken with respect to $z_{2,j}-z_{1,i} = \infty$, $z_{3,j}-z_{2,i} = \infty$, \dots, $z_{k,i} = \infty$.
\end{thm}

Similarly as in the case of series $A$ partial flag variety in Sect.~\ref{subsc:seriesA}, this expression can be simplified to involve only the residues with respect to the variables $z_{k,1}, \dots, z_{k,d_k}$. Using the same substitution described in the Appendix \ref{appendixA} and taking the residue first with respect to the variables $\mathbf{v}$ one obtains (after coming beck to the original variables $\mathbf{z}$) the following simplified formula.

\begin{thm} Under the assumptions of Theorem \ref{thm:formula-C} the following holds. 
\begin{align*}
\int\limits_{\Fl^C_d(W)} \kappa_{\T}(\alpha) &= \frac{1}{|\W|} \int\limits_{\homm(V,W) \git \S} \kappa_{\T}^{\S}(\alpha \cdot \tilde{e}) \\
&= \frac{1}{|\W|} \Res_{\mathbf{z}=\infty}\fr{ \alpha \cdot \prod_{\substack{ i \neq j \\  (i,j) \in I_{Fl}}} (z_{k,i} - z_{k,j}) \prod_{\substack{i,j = 1 \\ i < j}}^{d_k} (z_{k,i} + z_{k,j})}{ \prod_{l=1}^{d_k} \prod_{m=1}^{n}(t_m -z_{k,l})(t_m + z_{k,l})},
\end{align*}
where the residue is taken with respect to $z_{k,1} = \infty, \dots, z_{k,d_k} = \infty$. The set $I_{Fl}$ is the indexing set of Theorem \ref{thm:formula-A}.
\label{thm:formula-C-simplified}
\end{thm}

		\subsection{Series B and D partial flag varieties}\label{seriesBD}\label{subsc:seriesBD}

			\subsubsection{Series B}\label{seriesB}\label{subsc:seriesB}

Let $\Omega$ be a symmetric bilinear form on $W \simeq \C^{2n}$ as in Sect.~\ref{sc:orthogonal}. Consider the variety of partial orthogonal flags of type $d=(d_1,\dots,d_k)$ in $W$,
\[\Fl^{B}_d(W) = \{ V_1 \subset \dots \subset V_k \subset W :  V_i \textrm{ isotropic subspaces  of } W, \dim V_i = d_i \}.\]
The variety $\Fl^{B}_d(W)$ canonically embeds in $\Fl_d(W)$ and the maximal torus in $SO(2n)$ embeds in the maximal torus in $\U(2n)$, which gives a natural restriction map form $T_{\U(2n)}$-equivariant cohomology to $T_{SO(2n)}$-equivariant cohomology, which at a point is given by
\[ \C[t_1,\dots,t_{2n}] \to \C[t_1,\dots,t_n,t^{-1}_n,\dots,t^{-1}_1].\]
We can therefore obtain push-forward formulas in $T_{SO(2n)}$-equivariant cohomology of $\Fl^{B}_d(W)$ by embedding $\Fl^{B}_d(W) \hookrightarrow \Fl_d(W)$ and restricting to $\T = T_{SO(2n)}$-equivariant cohomology. As in the case of the isotropic flags, the only thing we need to compute in order to restrict the push-forward formulas for $\Fl_d(W)$ to $\Fl^{B}_d(W)$ is  the element in $\h_{\T \times \U(V)}^*(\homm(V,W))^{\W}$ 
mapping to the fundamental class of $\Fl^{B}_d(W)$ in $\h^*_{\T}(\Fl_d(W))$ via the Kirwan map. The computation is completely analogous to the one for type $C$ partial flag manifolds, yielding the following result.

\begin{remark}\label{remark12flagB}\label{remark:fund-class-B}
The fundamental class $[\Fl_d^B(W)] \in \h^*_{\T}(\Fl_d(W))$ is the image of the Kirwan map of the following element:
\[ [\Fl_d^B(W)] = \kappa_{\T} (2^{d_k}\prod_{\substack{i,j=1 \\ i \leq j }}^{d_k} (z_{k,i} + z_{k,j})).\]
\end{remark}

The residue-type formula (already simplified to include only the residues with respect to the variables $z_{k,i}$) is the following.

\begin{thm}\label{thm:formula-B} Let $\alpha \in \h^*_{\T \times \U(V)}(\homm(V,W))^\W$. Then
\begin{align*}
\int\limits_{\Fl^B_d(W)} \kappa_{\T}(\alpha) &= \frac{1}{|\W|} \int\limits_{\homm(V,W) \git \S} \kappa_{\T}^{\S}(\alpha \cdot \tilde{e}) \\
&= \frac{1}{|\W|} \Res_{\mathbf{z}=\infty}\fr{2^{d_k} \alpha \cdot \prod_{\substack{ i \neq j \\  (i,j) \in I_{Fl}}} (z_{k,i} - z_{k,j}) \prod_{\substack{i,j = 1 \\ i \leq j}}^{d_k} (z_{k,i} + z_{k,j})}{ \prod_{l=1}^{d_k} \prod_{m=1}^{n}(t_m -z_{k,l})(t_m + z_{k,l})},
\end{align*}
where the residue is taken with respect to $z_{k,1} = \infty, \dots, z_{k,d_k} = \infty$. The set $I_{Fl}$ is the indexing set of Theorem \ref{thm:formula-A}.
\end{thm}

			\subsubsection{Series D}\label{seriesD}\label{subsc:seriesD}

The only difference from the computations for type B partial flag variety is that now the inclusion is into the partial flags in a vector space $W$ of dimension $2n+1$,
\[\Fl^{D}_d(W) \hookrightarrow \Fl_d(W),\]
accompanied by the restriction from $T_{SO(2n+1)}$-equivariant to $T_{\U(2n+1)}$-equivariant cohomology, at the point given by
\[\C[t_1,\dots,t_{2n+1}] \to \C[t_1,\dots,t_n,t_n^{-1}, \dots, t_1^{-1}, 1].\]
The differences in computations are analogous as for the orthogonal Grassmannians $OG(n,2n)$ and $OG(n,2n+1)$. \newline

The residue-type formula (already simplified to include only the residues with respect to the variables $z_{k,i}$) is the following.

\begin{thm}\label{thm:formula-D} Let $\alpha \in \h^*_{\T \times \U(V)}(\homm(V,W))^\W$. Then
\begin{align*}
\int\limits_{\Fl^D_d(W)} \kappa_{\T}(\alpha) &= \frac{1}{|\W|} \int\limits_{\homm(V,W) \git \S} \kappa_{\T}^{\S}(\alpha \cdot \tilde{e}) \\
&= \frac{1}{|\W|} \Res_{\mathbf{z}=\infty}\fr{ \alpha \cdot \prod_{\substack{ i \neq j \\  (i,j) \in I_{Fl}}} (z_{k,i} - z_{k,j}) \prod_{\substack{i,j = 1 \\ i < j}}^{d_k} (z_{k,i} + z_{k,j}) \prod_{i=1}^{d_k} z_{k,i}}{ \prod_{l=1}^{d_k} \prod_{m=1}^{n}(t_m -z_{k,l})(t_m + z_{k,l})},
\end{align*}
where the residue is taken with respect to $z_{k,1} = \infty, \dots, z_{k,d_k} = \infty$. The set $I_{Fl}$ is the indexing set of Theorem \ref{thm:formula-A}.
\end{thm}

\chapter{Applications}\label{ch:applications}

As an application of our methods we present some examples of computations using our residue formulas.

\section{Schur polynomials}\label{sc:schur}

A partition is a nonincreasing sequence $\lambda = (\lambda_1 \geq \lambda_2 \geq \dots \geq \lambda_n) \in \mathbb{N}^n$ for some $n$. We fix such an $n$. We define the sum $\lambda + \mu$ of two partitions $\lambda$, $\mu$ componentwise. Let $\rho = (n,n-1,\ldots,1)$ denote the standard partition. \newline

To a partition $\lambda \in \mathbb{N}^n$ we assign a certain polynomial, called the Schur polynomial, defined as follows.
\begin{equation}\label{eq:jacobi}
 s_{(\lambda_1,\dots ,\lambda_n)}(z_1,\dots,z_n) =  \frac{\det\big(z_{j}^{\lambda_{i}+n-i}\big)_{1\leq i,j\leq n}} {\prod_{i<j}(z_i - z_j)}.
\end{equation}

Schur polynomials were introduced by Jacobi in \cite{jacobi1841}. They form an additive basis of the ring of symmetric polynomials with integer coefficients which is adapted for applications in representation theory and for the study of the geometry of complex Grassmannians. In representation theory, Schur polynomials are the characters of the finite-dimensional irreducible representations of $GL_n$ (\cite{schur1901}). In geometry, Schur polynomials represent the fundamental cohomology classes of the Schubert subvarieties of complex Grassmannians (\cite{fulton1999}). 

\section{The Pragacz--Ratajski theorem}\label{sc:pr}

We reformulate in the context of equivariant cohomology the well-known result of Pragacz and  Ratajski (\cite{pragacz1997}, Theorem 5.13) on push-forwards of Schur classes of vector bundles. \newline

To any symmetric polynomial $\phi$ in $r$ variables and any equivariant vector bundle $E$ of rank $r$ we associate an equivariant characteristic class, denoted by $\phi(E)$, defined as the  specialization of the polynomial $\phi$ with the equivariant Chern roots of $E$.\newline

Let $\T$ be a maximal torus in $Sp(n)$ acting on $LG(n)$ in the standard way. Lastly, we denote by $\mathcal{R}$ the tautological subbundle of $LG(n)$, endowed with the induced $\T$ action.  Set $p: LG(n)\to pt$ and let
 \[ \int\limits_{LG(n)}:\h_{\T}^*(LG(n))\to \h_{\T}^*(pt)\] denote the push-forward to a point in $\T$-equivariant cohomology.\newline

\begin{thm}\label{pragacz-ratajski}\label{thm:pr}
 The Schur class $s_{\lambda}(\R)$ has a nonzero image under $\int\limits_{LG(n)}$ only if $\lambda = 2\mu+\rho$ for some partition $\mu$. If $\lambda = 2\mu+\rho$, then the image is
\[\int\limits_{LG(n)} s_{\lambda}(\R)=s_{\mu}(t_1^2, \dots, t_n^2).\]
 \label{pr}
\end{thm}

\begin{proof}
Applying the push-forward formula of Corollary \ref{cor:lg-2} for the Lagrangian Grassmannian to the cohomology class represented by the polynomial $s_{\lambda}$, one gets
\[\int_{LG(n)} s_{(\lambda_1,\dots ,\lambda_n)}(\R) =  \Res_{\z= \infty} \frac{s_\lambda (z_1,\dots, z_n)\prod_{i<j}(z_j - z_i)}{\prod_{i=1}^n(t_i^2 - z_i^2)\prod_{i<j}(t_i^2 - t_j^2)}.\]

Using the definition \eqref{eq:jacobi} of the Schur polynomial one has
\[\int_{LG(n)} s_{(\lambda_1,\dots ,\lambda_n)}(\R) =  \frac{1}{\prod_{i<j}(t_i^2 - t_j^2)} \Res_{\z= \infty} \frac{\det(z_j^{\lambda_i + n -1})}{\prod_{i=1}^n (t_i^2 - z_i^2)}.\]

The iterated residue at infinity in $n$ variables $\z = (z_1,\dots,z_n)$ is defined using the iterated residue at zero according to the following equality 
\[Res_{\z=\infty} f(z_1, \dots, z_n) = (-1)^n \Res_{\z=0} \frac{1}{z_1^2\dots z_n^2} f(z_1^{-1}, \dots, z_n^{-1}).\] 
It follows that
\[\int_{LG(n)} s_{(\lambda_1,\dots ,\lambda_n)}(\R) =  \frac{(-1)^n}{\prod_{i<j}(t_i^2 - t_j^2)} \Res_{\z= 0} \frac{\det(z_j^{-\lambda_i - n +1})}{z_1^2 \dots z_n^2 \prod_{i=1}^n (t_i^2 - z_i^{-2})}.\]
The iterated residue at zero equals the coefficient at $z_1^{-1}\dots z_n^{-1}$ of the Laurent series expansion. \newline
Since
\[ \frac{1}{z_1^2 \dots z_n^2 \prod_{i=1}^n (t_i^2 - z_i^{-2})} = (-1)^n \sum_{k_1, \dots, k_n \geq 0} (t_1 z_1)^{2 k_1} \dots (t_n z_n)^{2 i_n}, \]
for $i_1, \dots, i_n \geq 0$, one has

\[Res_{\z=0}\frac{z_1^{-i_1}\dots z_n^{-i_n}}{z_1^2\dots z_n^2 \prod_{i=1}^n (t_i^2 - z_i^{-2})} = 
\begin{cases} 0 & \textrm{if } \exists_j,  i_j \textrm{ is even}\\
 (-1)^n t_1^{i_1 - 1} \dots t_n^{i_n - 1} & \textrm{else}
\end{cases}\]

As a corollary,
\[ \int_{LG(n)} s_{\lambda}(\R) = 
\begin{cases} 0 & \textrm{if }  \exists_i,  \lambda_i +n -i \textrm{ is even}\\
  \frac{1}{\prod_{i<j}(t_i^2 - t_j^2)} \det(t_j^{\lambda_i + n - i -1}) & \textrm{else}
\end{cases}\]

In the later case, each $\lambda_i$ can be written as $\lambda_i = n - i + 2 \mu_i$ for some integers $\mu_i \in \mathbb{Z}$. 
Since $\lambda$ is a partition, the sequence $\mu=(\mu_1, \dots, \mu_n)$ is decreasing. Moreover, since $\mu_n = \frac{\lambda_{n}-1}{2}$ is nonnegative, it follows that $\mu$ is a partition. Hence, one can write $\lambda$ as a sum of two partitions
\[\lambda = 2 \mu + \rho.\]
Then
\[ \int_{LG(n)} s_{(\lambda_1,\dots ,\lambda_n)}(\R) = \frac{1}{\prod_{i<j}(t_i^2 - t_j^2)} \det(t_j^{2(\mu_i + n - i)}) = s_{\mu}(t_1^2, \dots, t_n^2). \qedhere\]

\end{proof}

\chapter*{Appendix A}\label{appendixA}

Let $\S=\S_1 \times \dots \times \S_k$ be a product of tori of dimensions $d_1,d_2,\dots,d_k$ and let

\begin{itemize}
	\item  $\S_1^\# = \{z_{1,1},\dots, z_{1,d_1}\}$ be a basis of characters of $\S_1$,
	\item $\S_2^\# = \{z_{2,1},\dots, z_{2,d_2}\}$ be a basis of characters of $\S_2$, \newline

	 \hspace{1em} \vdots
	\item $\S_k^\#=\{z_{k,1},\dots, z_{k,d_k}\}$ be a basis of characters of $\S_k$.
\end{itemize}
Consider the action of $\S$ on $\homm(V,W) = \bigoplus_{i=1}^{k-1} \homm(V_i, V_{i+1}) \oplus \homm(V_k,W)$ defined by:
\[(g_1, \dots, g_k)(A_1, \dots, A_{k-1}, B) = (g_2 A_1 g_1^{-1}, g_3 A_2 g_2^{-1},\dots, g_k A_{k-1} g_{k-1}^{-1}, B g_k^{-1}),\]
where

\[
g_1 = 
\begin{bmatrix} 
z_{1,1} &  &  &  \\
 & z_{1,2} &  & \\
 &  & \ddots & \\
 & & & z_{1,d_1}
\end{bmatrix},
g_2 = 
\begin{bmatrix} 
z_{2,1} &  &  &  \\
 & z_{2,2} &  & \\
 &  & \ddots & \\
 & & & z_{2,d_2}
\end{bmatrix}, \dots 
\]

The torus $\S$ acts on a $d_2 \times d_1$-matrix $A_1 \in \homm(V_1,V_2)$ via

\[ \begin{bmatrix} 
z_{2,1} &  &  &  \\
 & z_{2,2} &  & \\
 &  & \ddots & \\
 & & & z_{2,d_2}
\end{bmatrix}
\cdot A_1 \cdot
\begin{bmatrix} 
z_{1,1}^{-1} &  &   \\
  & \ddots & \\
 & & z_{1,d_1}^{-1}
\end{bmatrix}
 \]
so the action multiplies the rows of $A_1$ by $z_{2,1}, z_{2,2},\dots, z_{2,d_2}$ and multiplies the columns by $z_{1,1}^{-1}, \dots, z_{1,d_1}^{-1}$. Analogously for the $d_3 \times d_2$-matrix $A_2 \in \homm(V_2,V_3)$ the action multiplies the rows of $A_2$ by $z_{3,1}, z_{3,2},\dots, z_{3,d_3}$ and the columns by $z_{2,1}^{-1}, z_{2,2}^{-1},\dots, z_{2,d_2}^{-1}$ and similarly for the remaining $A_i \in \homm(V_i, V_{i+1})$ for $i < k$. The action is different on the last component of $\homm(V,W)$---the matrix $d_k \times n$-matrix $B \in \homm(V_k,W)$ is only multiplied on the right---so the action multiplies the columns of $B$ by $z_{k,1}^{-1}, z_{k,2}^{-1},\dots, z_{k,d_k}^{-1}$. \newline

Viewing $\homm(V,W)$ as a vector space $\C^N$ with $N=d_1 + \dots + d_k$ and coordinates  $\{ a^m_{i,j}, b_{i,j}\}$, where for $m=1,\dots,k-1$
\[ \{ a^m_{i,j} \}_{\substack{i=1,\dots,d_{m+1} \\ j= 1, \dots, d_m}}\]
are the coordinates of the matrix $A_m$ and 
\[ \{ b_{i,j} \}_{\substack{i=1,\dots,d_{k} \\ j= 1, \dots, d_n}}\] 
are the coordinates of the matrix $B$, the torus $\S$ acts on $\homm(V,W)$ by multiplying the coordinate $a^m_{i,j}$ by $z_{m+1,i} z_{m,j}^{-1}$ and the coordinate $b_{i,j}$ by $z_{k,j}^{-1}$. It follows that the weights of the action of $\S$ in $\homm(V,W)$ are

\[ \bigcup_{m=1}^{k-1}  \{ z_{m+1,i} - z_{m,j} \}_{\substack{i=1,\dots,d_{m+1} \\ j= 1, \dots, d_m} } \cup \{ -z_{k,j} \}_{j=1,\dots,d_k} = (\S_2^\# - \S_1^\#) \cup (\S_3^\# - \S_2^\#) \cup \dots \cup \S_k^\# \]

Let us consider the following substitution, yielding a new basis of characters, motivated by the above set of weights of the action:

\[
\begin{cases}
	v_{i,j} = z_{i+1,j} - z_{i,j} & \mbox{ for } i=1,\dots,k-1 \mbox{ and } j=1,\dots,d_i \\
	u_i = z_{k,i} & \mbox{ for } i=1,\dots,d_k
\end{cases}
\]
With this substitution the weights of the $\S$ action are $u_i$, $v_{i,j}$ as above and their linear combinations. More precisely, $\S$ acts on the coordinates $b_{i,j}$ with weights $u_i$, on the diagonal coordinates of the matrices $A_m$ by weights $v_{m,j}$ and on the remaining coordinates (the non-diagonal entries of $A_m$'s) as follows. The weight at $a^m_{i,j}$ equals $z_{m+1,i} - z_{m,j}$ and from the definition of $v_{k-1,i}$ one has $z_{k,i} = z_{k-1,i} + v_{k-1,i}$. By a basic recursion procedure one get
\[z_{k-l, i} = u_i - v_{k-1,i} - v_{k-2,i} - \dots - v_{k-l, i}, \]
or, equivalently
\[z_{m,i} = u_i - v_{k-1,i} - v_{k-2,i} - \dots -v_{m,i}\]
Hence the weight of the action at $a^m_{i,j}$ equals
\[u_i - u_j -(v_{k-1,i} + v_{k-2,i} + \dots +v_{m+1,i}) + (v_{k-1,j} + v_{k-2,j} + \dots + v_{m,j}) = \]
\[=u_i - u_j - \sum_{n=m+1}^{k-1} v_{n,i} + \sum_{n=m}^{k-1} v_{n,j}\]

In particular, one can express in the new variables the cohomology classes in $\h^*_{\T \times \S}(pt)$ which appear in computations in Ch.~\ref{ch:formulas}, Sect.~\ref{sc:partial}. \newline

The class $\tilde{e}$ such that $\kappa_{\T}^{\S}(\tilde{e}) = e$ equals
\begin{align*}
\tilde{e} &= \prod_{\substack{i,j = 1 \\ i \neq j}}^{d_1} (z_{1,i} - z_{1,j})  \prod_{\substack{i,j = 1 \\ i \neq j}}^{d_2} (z_{2,i} - z_{2,j}) \dots \prod_{\substack{i,j = 1 \\ i \neq j}}^{d_k} (z_{k,i} - z_{k,j}) \\
& = (-1)^M \Vand(\S_1^\#)^2 \dots \Vand(\S_k^\#)^2,
\end{align*}
where $\Vand(A)$ for a finite ordered set $A$ denotes the Vandermonde determinant, $\Vand(A) = \prod_{i<j} (a_i - a_j)$, and $M= \prod_{i=1}^k \binom{d_i}{2}$. 
In the new coordinates $u_i, v_{i,j}$ this expression equals
\[\prod_{\substack{i,j = 1 \\ i \neq j}}^{d_1} (u_i - u_j - \sum_{n=1}^{k-1} (v_{n,i} - v_{n,j}) )  \prod_{\substack{i,j = 1 \\ i \neq j}}^{d_2} ( u_i - u_j - \sum_{n=2}^{k-1}(v_{n,i} - v_{n,j}) )  \dots \prod_{\substack{i,j = 1 \\ i \neq j}}^{d_k} (u_i - u_j) \]	

The $\T \times \S$-equivariant Euler class at zero equals
\[ e^{\T \times \S}(0) = \prod_{m=1}^{k-1} \prod_{i=1}^{d_{m+1}} \prod_{j=1}^{d_m} (z_{m+1,i} - z_{m, j})\cdot \prod_{i=1}^{n} \prod_{j=1}^{d_k}(t_{i}-z_{k,j})\]
In the variables $u_i, v_{i,j}$ this class equals

\[e^{\T \times \S}(0) = \prod_{m=1}^{k-1} \prod_{i=1}^{d_{m+1}} \prod_{j=1}^{d_m} (u_i - u_j - \sum_{n=m+1}^{k-1} v_{n,i} + \sum_{n=m}^{k-1} v_{n,j}) \prod_{i=1}^{n} \prod_{j=1}^{d_k}(t_{i}-u_j).\] 

\chapter*{Appendix B}\label{appendixB}


In the proof of Theorem \ref{thm:formula-partial-vu} one needs to simplify the following expression

\begin{equation} \fr{ \prod_{\substack{i,j = 1 \\ i \neq j}}^{d_1} (u_i - u_j)  \prod_{\substack{i,j = 1 \\ i \neq j}}^{d_2} ( u_i - u_j)  \dots \prod_{\substack{i,j = 1 \\ i \neq j}}^{d_k} (u_i - u_j)}{\prod_{m=1}^{k-1} \prod_{i=1}^{d_{m+1}} \prod_{\substack{j=1 \\ j \neq i}}^{d_m} (u_i - u_j)} \label{quotient}
\end{equation}
Both the numerator at the denominator are products of factors $(u_i-u_j)$, only indexed by different multisets. The quotient is indexed by the difference of these multisets (difference in the multiset sense, counting elements with multiplicities). To efficiently compute the multiset indexing the quotient we draw the multisets on the plane (more precisely on the $d_k \times d_k$ grid, because the indices $i,j$ range from $1$ do $d_k$), depicting the multiplicities by colour intensity. For example, in the picture below

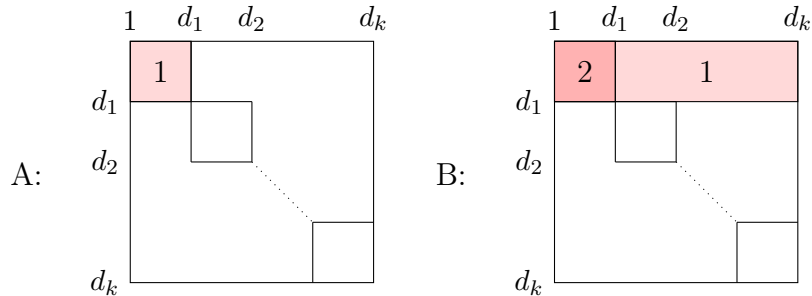
\begin{center}
\begin{tabular}{cccc}

\begin{tikzpicture}[scale=0.8]
\node at (0,0){};
\node at (0,4){};
\node at (0,2){A:};
\end{tikzpicture} 
 & 
\begin{tikzpicture}[scale=0.8]

\filldraw[fill=red!15!white] (0,4) rectangle (1,3);

\draw (0,0) rectangle (4,4);
\draw (1,4) -- (1,2);
\draw (2,3) -- (2,2);
\draw (3,1) -- (3,0);
\draw (0,3) -- (2,3);
\draw (1,2) -- (2,2);
\draw (3,1) -- (4,1);
\draw[dotted] (2,2) --(3,1);

\node at (0.5, 3.5){$1$};
\node at (1.5, 2.5){};
\node at (3.5, 0.5){};
\node[above] at (0, 4){\small{$1$}};
\node[above] at (1, 4){\small{$d_1$}};
\node[above] at (2, 4){\small{$d_2$}};
\node[above] at (4, 4){\small{$d_k$}};
\node[left] at (0, 3){\small{$d_1$}};
\node[left] at (0, 2){\small{$d_2$}};
\node[left] at (0, 0){\small{$d_k$}};

\end{tikzpicture}

&
\begin{tikzpicture}[scale=0.8]
\node at (0,0){};
\node at (0,4){};
\node at (0,2){B:};
\end{tikzpicture}  & 

\begin{tikzpicture}[scale=0.8]

\filldraw[fill=red!15!white] (0,4) rectangle (4,3);
\filldraw[fill=red!30!white] (0,4) rectangle (1,3);

\draw (0,0) rectangle (4,4);
\draw (1,4) -- (1,2);
\draw (2,3) -- (2,2);
\draw (3,1) -- (3,0);
\draw (0,3) -- (2,3);
\draw (1,2) -- (2,2);
\draw (3,1) -- (4,1);
\draw[dotted] (2,2) --(3,1);

\node at (0.5, 3.5){$2$};
\node at (2.5, 3.5){$1$};
\node at (3.5, 0.5){};
\node[above] at (0, 4){\small{$1$}};
\node[above] at (1, 4){\small{$d_1$}};
\node[above] at (2, 4){\small{$d_2$}};
\node[above] at (4, 4){\small{$d_k$}};
\node[left] at (0, 3){\small{$d_1$}};
\node[left] at (0, 2){\small{$d_2$}};
\node[left] at (0, 0){\small{$d_k$}};

\end{tikzpicture} \\
\end{tabular}
\end{center}
picture A depicts the multiset consisting of $(i,j)$ such that $i=1,\dots,d_1$ and $j=1,\dots,d_1$, each with multiplicity one, whereas picture B depicts the multiset consisting of $(i,j)$ such that $i=1,\dots,d_1$ and $j=1,\dots,d_k$, with all elements $(i,j)$ such that $i,j= 1\dots,d_1$ have multiplicity two and the remaining ones have multiplicity one.

\subsection{The numerator}

The numerator of expression \eqref{quotient} is the following 
\[num(0, \mathbf{u}) :=  \prod_{\substack{i,j = 1 \\ i \neq j}}^{d_1} (u_i - u_j )  \prod_{\substack{i,j = 1 \\ i \neq j}}^{d_2} ( u_i - u_j )  \dots \prod_{\substack{i,j = 1 \\ i \neq j}}^{d_k} (u_i - u_j) = \bigstar \]
and is indexed by the multiset $\mathcal{A}$
\[ \bigstar =  \prod_{\substack{i\neq j \\ (i,j) \in \mathcal{A} }} (u_i - u_j ),\]
where $\mathcal{A} = \mathcal{A}_1 \uplus \dots \uplus \mathcal{A}_k$, with the multisets $\mathcal{A}_i$ given by

\begin{align*}
\mathcal{A}_1 &= \{ (i,j): i,j = 1, \dots, d_1 \},\\
\mathcal{A}_2 &= \{ (i,j): i,j = 1, \dots, d_2 \},\\
 &\vdots \\
\mathcal{A}_k &= \{ (i,j): i,j = 1, \dots, d_k \}.
\end{align*}
The multiset $\mathcal{A} = \biguplus_{i=1}^k \mathcal{A}_i$ is depicted below.

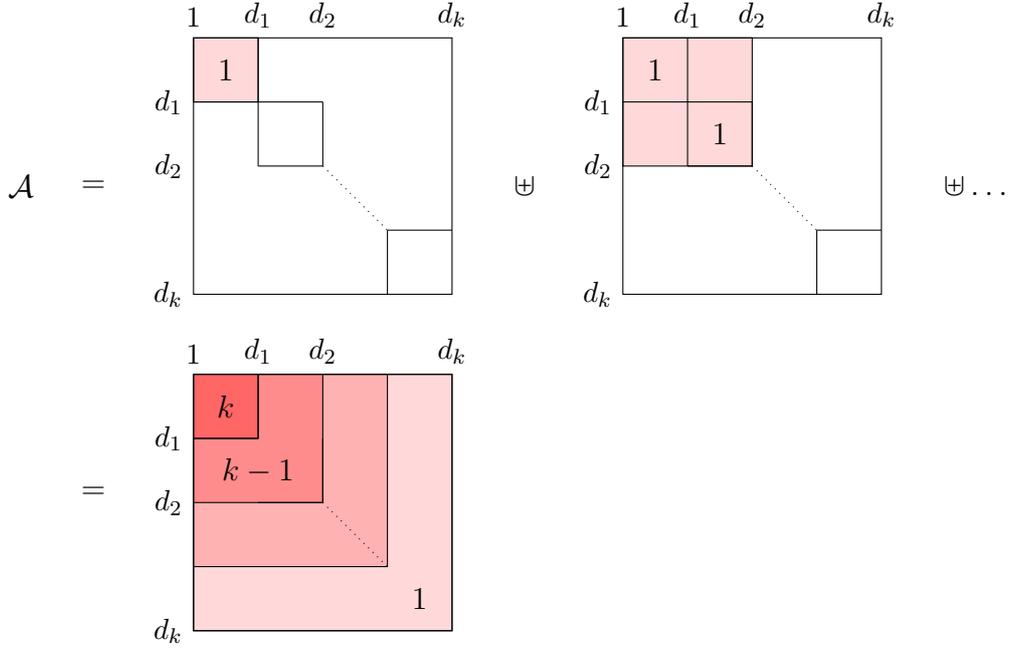
\begin{center}
\begin{tabular}{cccccc}
\begin{tikzpicture}[scale=0.8]
\node at (0,0){};
\node at (0,4){};
\node at (0,2){$\mathcal{A}$};
\end{tikzpicture}
&
\begin{tikzpicture}[scale=0.8]
\node at (0,0){};
\node at (0,4){};
\node at (0,2){$=$};
\end{tikzpicture}
 & 
\begin{tikzpicture}[scale=0.85]

\filldraw[fill=red!15!white] (0,4) rectangle (1,3);

\draw (0,0) rectangle (4,4);
\draw (1,4) -- (1,2);
\draw (2,3) -- (2,2);
\draw (3,1) -- (3,0);
\draw (0,3) -- (2,3);
\draw (1,2) -- (2,2);
\draw (3,1) -- (4,1);
\draw[dotted] (2,2) --(3,1);

\node at (0.5, 3.5){$1$};
\node at (1.5, 2.5){};
\node at (3.5, 0.5){};
\node[above] at (0, 4){\small{$1$}};
\node[above] at (1, 4){\small{$d_1$}};
\node[above] at (2, 4){\small{$d_2$}};
\node[above] at (4, 4){\small{$d_k$}};
\node[left] at (0, 3){\small{$d_1$}};
\node[left] at (0, 2){\small{$d_2$}};
\node[left] at (0, 0){\small{$d_k$}};

\end{tikzpicture}


&
\begin{tikzpicture}[scale=0.8]
\node at (0,0){};
\node at (0,4){};
\node at (0,2){$\uplus$};
\end{tikzpicture}
 & 

\begin{tikzpicture}[scale=0.85]

\filldraw[fill=red!15!white] (0,4) rectangle (2,2);

\draw (0,0) rectangle (4,4);
\draw (1,4) -- (1,2);
\draw (2,3) -- (2,2);
\draw (3,1) -- (3,0);
\draw (0,3) -- (2,3);
\draw (1,2) -- (2,2);
\draw (3,1) -- (4,1);
\draw[dotted] (2,2) --(3,1);

\node at (0.5, 3.5){$1$};
\node at (1.5, 2.5){$1$};
\node at (3.5, 0.5){};
\node[above] at (0, 4){\small{$1$}};
\node[above] at (1, 4){\small{$d_1$}};
\node[above] at (2, 4){\small{$d_2$}};
\node[above] at (4, 4){\small{$d_k$}};
\node[left] at (0, 3){\small{$d_1$}};
\node[left] at (0, 2){\small{$d_2$}};
\node[left] at (0, 0){\small{$d_k$}};

\end{tikzpicture}
&
\begin{tikzpicture}[scale=0.8]
\node at (0,0){};
\node at (0,4){};
\node at (0,2){$\uplus \dots $};
\end{tikzpicture}

\\

\begin{tikzpicture}
\node at (0,0){};
\node at (0,4){};
\node at (0,2){};
\end{tikzpicture}
& 

\begin{tikzpicture}
\node at (0,0){};
\node at (0,4){};
\node at (0,2){$=$};
\end{tikzpicture}
& 

\begin{tikzpicture}[scale=0.85]
\filldraw[fill=red!15!white] (0,4) rectangle (4,0);
\filldraw[fill=red!30!white] (0,4) rectangle (3,1);
\filldraw[fill=red!45!white] (0,4) rectangle (2,2);
\filldraw[fill=red!60!white] (0,4) rectangle (1,3);

\draw (0,0) rectangle (4,4);
\draw (1,4) -- (1,3);
\draw (2,3) -- (2,2);
\draw (0,3) -- (1,3);
\draw (1,2) -- (2,2);
\draw[dotted] (2,2) --(3,1);

\node at (0.5, 3.5){$k$};
\node at (1, 2.5){$k-1$};
\node at (3.5, 0.5){$1$};
\node[above] at (0, 4){\small{$1$}};
\node[above] at (1, 4){\small{$d_1$}};
\node[above] at (2, 4){\small{$d_2$}};
\node[above] at (4, 4){\small{$d_k$}};
\node[left] at (0, 3){\small{$d_1$}};
\node[left] at (0, 2){\small{$d_2$}};
\node[left] at (0, 0){\small{$d_k$}};

\end{tikzpicture}
& & & \\
\end{tabular}
\end{center}

In other words, $\mathcal{A} = \bigcup_{i=1}^k \mathcal{A}_i$ is a multiset with multiplicities denoted by colour intensity, ranging from 1 (for the south-east corner) and increasing by one with each colour change (hence reaching $k$ in the north-west corner $P_1$).

\subsection{The denominator}

The denominator of the expression \eqref{quotient} is the following
\[ \prod_{m=1}^{k-1} \prod_{i=1}^{d_{m+1}} \prod_{j=1}^{d_m} (u_i - u_j )\]
and is a product indexed by the multiset $\mathcal{B} = \mathcal{B}_1 \uplus \dots \uplus \mathcal{B}_{k-1}$, with the multisets $\mathcal{B}_i$ are given by

\begin{align*}
\mathcal{B}_1 &= \{ (i,j): i = 1, \dots, d_2, j=1,\dots,d_1 \},\\
\mathcal{B}_2 &= \{ (i,j): i = 1, \dots, d_3, j=1,\dots,d_2 \},\\
 &\vdots \\
\mathcal{B}_{k-1} &= \{ (i,j): i = 1, \dots, d_k, j=1,\dots,d_{k-1} \}.
\end{align*}
The multiset $\mathcal{B} = \bigcup_{m=1}^{k-1} \mathcal{B}_m$ is depicted below.

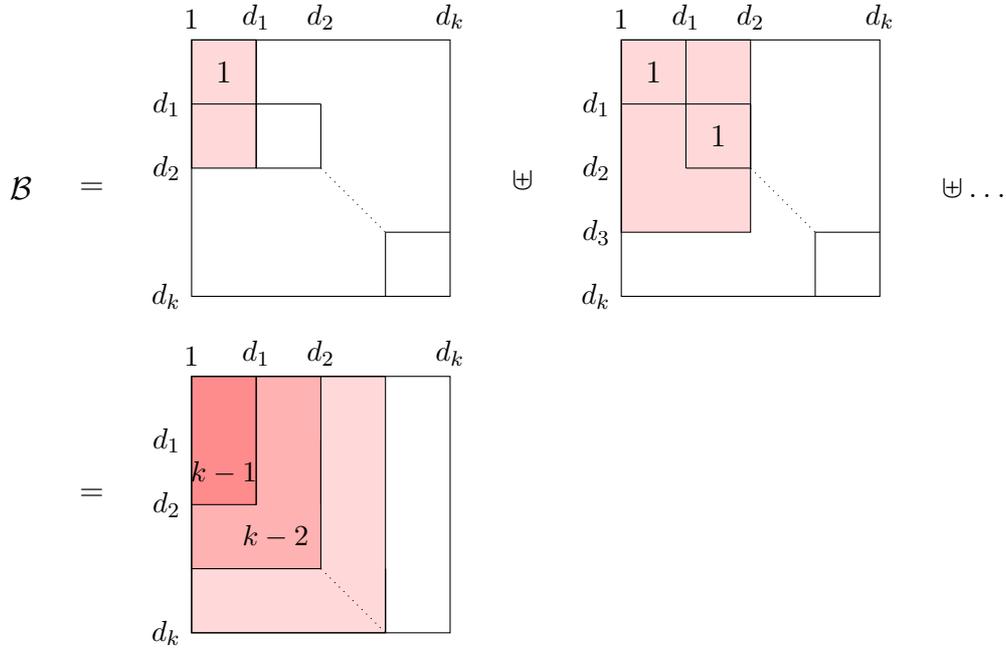
\begin{center}
\begin{tabular}{cccccc}

\begin{tikzpicture}[scale=0.8]
\node at (0,0){};
\node at (0,4){};
\node at (0,2){$\mathcal{B}$};
\end{tikzpicture}
 & 
\begin{tikzpicture}[scale=0.8]
\node at (0,0){};
\node at (0,4){};
\node at (0,2){$=$};
\end{tikzpicture}
 &  

\begin{tikzpicture}[scale=0.85]
\filldraw[fill=red!15!white] (0,4) rectangle (1,2);


\draw (0,0) rectangle (4,4);
\draw (1,4) -- (1,2);
\draw (2,3) -- (2,2);
\draw (3,1) -- (3,0);
\draw (0,3) -- (2,3);
\draw (1,2) -- (2,2);
\draw (3,1) -- (4,1);
\draw[dotted] (2,2) --(3,1);

\node at (0.5, 3.5){$1$};
\node at (1.5, 2.5){};
\node at (3.5, 0.5){};
\node[above] at (0, 4){\small{$1$}};
\node[above] at (1, 4){\small{$d_1$}};
\node[above] at (2, 4){\small{$d_2$}};
\node[above] at (4, 4){\small{$d_k$}};
\node[left] at (0, 3){\small{$d_1$}};
\node[left] at (0, 2){\small{$d_2$}};
\node[left] at (0, 0){\small{$d_k$}};

\end{tikzpicture}


&
\begin{tikzpicture}[scale=0.85]
\node at (0,0){};
\node at (0,4){};
\node at (0,2){$\uplus$};
\end{tikzpicture}
 & 

\begin{tikzpicture}[scale=0.85]
\filldraw[fill=red!15!white] (0,4) rectangle (2,1);


\draw (0,0) rectangle (4,4);
\draw (1,4) -- (1,2);
\draw (2,3) -- (2,2);
\draw (3,1) -- (3,0);
\draw (0,3) -- (2,3);
\draw (1,2) -- (2,2);
\draw (3,1) -- (4,1);
\draw[dotted] (2,2) --(3,1);

\node at (0.5, 3.5){$1$};
\node at (1.5, 2.5){$1$};
\node at (3.5, 0.5){};
\node[above] at (0, 4){\small{$1$}};
\node[above] at (1, 4){\small{$d_1$}};
\node[above] at (2, 4){\small{$d_2$}};
\node[above] at (4, 4){\small{$d_k$}};
\node[left] at (0, 3){\small{$d_1$}};
\node[left] at (0, 2){\small{$d_2$}};
\node[left] at (0, 1){\small{$d_3$}};
\node[left] at (0, 0){\small{$d_k$}};
\end{tikzpicture}
&
\begin{tikzpicture}[scale=0.8]
\node at (0,0){};
\node at (0,4){};
\node at (0,2){$\uplus \dots$};
\end{tikzpicture}

\\

\begin{tikzpicture}
\node at (0,0){};
\node at (0,4){};
\node at (0,2){};
\end{tikzpicture}
 & 

\begin{tikzpicture}
\node at (0,0){};
\node at (0,4){};
\node at (0,2){$=$};
\end{tikzpicture}
 &
\begin{tikzpicture}[scale=0.85]

\filldraw[fill=red!15!white] (0,4) rectangle (3,0);
\filldraw[fill=red!30!white] (0,4) rectangle (2,1);
\filldraw[fill=red!45!white] (0,4) rectangle (1,2);


\draw (0,0) rectangle (4,4);
\draw (1,4) -- (1,2);
\draw (2,3) -- (2,2);
\draw (3,1) -- (3,0);
\draw[dotted] (2,1) --(3,0);

\node at (0.5, 2.5){\small{$k-1$}};
\node at (1.3, 1.5){\small{$k-2$}};
\node at (3.5, 0.5){};
\node[above] at (0, 4){\small{$1$}};
\node[above] at (1, 4){\small{$d_1$}};
\node[above] at (2, 4){\small{$d_2$}};
\node[above] at (4, 4){\small{$d_k$}};
\node[left] at (0, 3){\small{$d_1$}};
\node[left] at (0, 2){\small{$d_2$}};
\node[left] at (0, 0){\small{$d_k$}};

\end{tikzpicture}
& & & 
\\
\end{tabular}
\end{center}
In other words, $\mathcal{B} = \bigcup_{m=1}^{k-1} \mathcal{B}_m$ is a multiset with multiplicities denoted by colour intensity, ranging from 0 (for the south-east corner) and increasing by one with each colour change (hence reaching $k-1$ in the north-west corner).

\subsection{The quotient}

The quotient \eqref{quotient} is indexed by the multiset $\mathcal{C} = \mathcal{A} - \mathcal{B}$, which equals:

\begin{center}
\begin{tabular}{cccccc}

\begin{tikzpicture}[scale=0.8]
\node at (0,0){};
\node at (0,4){};
\node at (0,2){$\mathcal{C}$};
\end{tikzpicture}
 & 

\begin{tikzpicture}[scale=0.8]
\node at (0,0){};
\node at (0,4){};
\node at (0,2){$=$};
\end{tikzpicture}
 & 

\begin{tikzpicture}[scale=0.85]
\filldraw[fill=red!15!white] (0,4) rectangle (4,0);
\filldraw[fill=red!30!white] (0,4) rectangle (3,1);
\filldraw[fill=red!45!white] (0,4) rectangle (2,2);
\filldraw[fill=red!60!white] (0,4) rectangle (1,3);

\draw (0,0) rectangle (4,4);
\draw (1,4) -- (1,3);
\draw (2,3) -- (2,2);
\draw (0,3) -- (1,3);
\draw (1,2) -- (2,2);
\draw[dotted] (2,2) --(3,1);

\node at (0.5, 3.5){$k$};
\node at (1, 2.5){$k-1$};
\node at (3.5, 0.5){$1$};
\node[above] at (0, 4){\small{$1$}};
\node[above] at (1, 4){\small{$d_1$}};
\node[above] at (2, 4){\small{$d_2$}};
\node[above] at (4, 4){\small{$d_k$}};
\node[left] at (0, 3){\small{$d_1$}};
\node[left] at (0, 2){\small{$d_2$}};
\node[left] at (0, 0){\small{$d_k$}};

\end{tikzpicture}
&
\begin{tikzpicture}[scale=0.8]
\node at (0,0){};
\node at (0,4){};
\node at (0,2){$-$};
\end{tikzpicture}
 & 
\begin{tikzpicture}[scale=0.85]

\filldraw[fill=red!15!white] (0,4) rectangle (3,0);
\filldraw[fill=red!30!white] (0,4) rectangle (2,1);
\filldraw[fill=red!45!white] (0,4) rectangle (1,2);


\draw (0,0) rectangle (4,4);
\draw (1,4) -- (1,2);
\draw (2,3) -- (2,2);
\draw (3,1) -- (3,0);
\draw[dotted] (2,1) --(3,0);

\node at (0.5, 2.5){\small{$k-1$}};
\node at (1.3, 1.5){\small{$k-2$}};
\node at (3.5, 0.5){};
\node[above] at (0, 4){\small{$1$}};
\node[above] at (1, 4){\small{$d_1$}};
\node[above] at (2, 4){\small{$d_2$}};
\node[above] at (4, 4){\small{$d_k$}};
\node[left] at (0, 3){\small{$d_1$}};
\node[left] at (0, 2){\small{$d_2$}};
\node[left] at (0, 0){\small{$d_k$}};

\end{tikzpicture}
&
\begin{tikzpicture}[scale=0.8]
\node at (0,0){};
\node at (0,4){};
\node at (0,2){};
\end{tikzpicture}

 \\

\begin{tikzpicture}[scale=0.8]
\node at (0,0){};
\node at (0,4){};
\node at (0,2){};
\end{tikzpicture} 
&

\begin{tikzpicture}[scale=0.8]
\node at (0,0){};
\node at (0,4){};
\node at (0,2){$=$};
\end{tikzpicture} 
&
\begin{tikzpicture}[scale=0.85]

\filldraw[fill=red!15!white] (0,4) -- (0,3) -- (1,3) -- (1,2) -- (2,2) -- (2,1) -- (3,1) -- (3,0) -- (4,0) -- (4,4) -- cycle;


\draw (0,0) rectangle (4,4);
\draw (1,4) -- (1,2);
\draw (2,3) -- (2,2);
\draw (3,1) -- (3,0);
\draw (0,3) -- (2,3);
\draw (1,2) -- (2,2);
\draw (3,1) -- (4,1);
\draw[dotted] (2,2) --(3,1);

\node at (0.5, 3.5){$1$};
\node at (1.5, 2.5){$1$};
\node at (3.5, 0.5){$1$};
\node[above] at (0, 4){\small{$1$}};
\node[above] at (1, 4){\small{$d_1$}};
\node[above] at (2, 4){\small{$d_2$}};
\node[above] at (4, 4){\small{$d_k$}};
\node[left] at (0, 3){\small{$d_1$}};
\node[left] at (0, 2){\small{$d_2$}};
\node[left] at (0, 0){\small{$d_k$}};

\end{tikzpicture}
 & & & 
\\
\end{tabular}
\end{center}
Hence the quotient is indexed over the set $\mathcal{C}$ consisting of pairs $(i,j)$ in the shaded area in the picture, each appearing with multiplicity one.

\bibliographystyle{alpha}
\bibliography{thesis}

\end{document}